\documentclass[12pt]{amsart}      
\usepackage{amssymb}
\usepackage{eucal}
\usepackage{amsmath}
\usepackage{amscd}
\usepackage{multicol}
\usepackage[all]{xy}           
\usepackage{graphicx}
\usepackage{color}
\usepackage{colordvi}
\usepackage{xspace}
\usepackage{axodraw}
\usepackage{epsfig}
\usepackage{float}
\usepackage{hyperref}

\begin{document}  

\newcommand{\nc}{\newcommand}
\newcommand{\delete}[1]{}
\nc{\dfootnote}[1]{{}}          
\nc{\ffootnote}[1]{\dfootnote{#1}}
\nc{\mfootnote}[1]{\footnote{#1}} 
\nc{\ofootnote}[1]{\footnote{\tiny Older version: #1}} 

\nc{\mlabel}[1]{\label{#1}}  
\nc{\mcite}[1]{\cite{#1}}  
\nc{\mref}[1]{\ref{#1}}  

\nc{\mbibitem}[1]{\bibitem{#1}} 
\nc{\mkeep}[1]{\marginpar{{\bf #1}}} 

\newtheorem{theorem}{Theorem}[section]
\newtheorem{prop}[theorem]{Proposition}
\newtheorem{defn}[theorem]{Definition}
\newtheorem{lemma}[theorem]{Lemma}
\newtheorem{coro}[theorem]{Corollary}
\newtheorem{prop-def}{Proposition-Definition}[section]
\newtheorem{claim}{Claim}[section]
\newtheorem{remark}[theorem]{Remark}
\newtheorem{propprop}{Proposed Proposition}[section]
\newtheorem{conjecture}{Conjecture}
\newtheorem{exam}[theorem]{Example}
\newtheorem{assumption}{Assumption}
\newtheorem{condition}[theorem]{Assumption}
\newtheorem{question}[theorem]{Question}

\newtheorem{tempexer}{Exercise}[section]
\newtheorem{temprmk}[theorem]{Remark}

\newenvironment{exer}{\begin{tempexer}\rm}{\end{tempexer}}
\newenvironment{rmk}{\begin{temprmk}\rm}{\end{temprmk}}

\renewcommand{\labelenumi}{{\rm(\alph{enumi})}}
\renewcommand{\theenumi}{(\alph{enumi})}

\nc{\tred}[1]{\textcolor{red}{#1}}
\nc{\tblue}[1]{\textcolor{blue}{#1}}
\nc{\tgreen}[1]{\textcolor{green}{#1}}

\nc{\wvec}[2]{{\scriptsize{ [
    \begin{array}{c} #1 \\ #2 \end{array}   ]}}}

\nc{\adec}{\check{;}}
\nc{\dftimes}{\widetilde{\otimes}} \nc{\dfl}{\succ}
\nc{\dfr}{\prec} \nc{\dfc}{\circ} \nc{\dfb}{\bullet}
\nc{\dft}{\star} \nc{\dfcf}{{\mathbf k}} \nc{\spr}{\cdot}
\nc{\disp}[1]{\displaystyle{#1}}
\nc{\bin}[2]{ (_{\stackrel{\scs{#1}}{\scs{#2}}})}  
\nc{\binc}[2]{ \left (\!\! \begin{array}{c} \scs{#1}\\
    \scs{#2} \end{array}\!\! \right )}  
\nc{\bincc}[2]{  \left ( {\scs{#1} \atop
    \vspace{-.5cm}\scs{#2}} \right )}  
\nc{\sarray}[2]{\begin{array}{c}#1 \vspace{.1cm}\\ \hline
    \vspace{-.35cm} \\ #2 \end{array}}
\nc{\bs}{\bar{S}} \nc{\dcup}{\stackrel{\bullet}{\cup}}
\nc{\cprod}{\ast}
\nc{\dbigcup}{\stackrel{\bullet}{\bigcup}} \nc{\etree}{\big |}
\nc{\la}{\longrightarrow} \nc{\fe}{\'{e}} \nc{\rar}{\rightarrow}
\nc{\dar}{\downarrow} \nc{\dap}[1]{\downarrow
\rlap{$\scriptstyle{#1}$}} \nc{\uap}[1]{\uparrow
\rlap{$\scriptstyle{#1}$}} \nc{\defeq}{\stackrel{\rm def}{=}}
\nc{\dis}[1]{\displaystyle{#1}}
\nc{\dotcup}{\,\displaystyle{\bigcup^\bullet}\ }
\nc{\barot}{{\otimes}}
\nc{\sdotcup}{\tiny{\displaystyle{\bigcup^\bullet}\ }}
\nc{\hcm}{\ \hat{,}\ }
\nc{\hcirc}{\hat{\circ}} \nc{\hts}{\hat{\shpr}}
\nc{\lts}{\stackrel{\leftarrow}{\shpr}}
\nc{\rts}{\stackrel{\rightarrow}{\shpr}} \nc{\lleft}{[}
\nc{\lright}{]} \nc{\uni}[1]{\tilde{#1}} \nc{\wor}[1]{\check{#1}}
\nc{\free}[1]{\bar{#1}} \nc{\den}[1]{\check{#1}} \nc{\lrpa}{\wr}
\nc{\dprod}{\ast_P}
\nc{\gzeta}{\bar{\zeta}}
\nc{\mprod}{\pm}
\nc{\curlyl}{\left \{ \begin{array}{c} {} \\ {} \end{array}
    \right .  \!\!\!\!\!\!\!}
\nc{\curlyr}{ \!\!\!\!\!\!\!
    \left . \begin{array}{c} {} \\ {} \end{array}
    \right \} }
\nc{\leaf}{\ell}       
\nc{\longmid}{\left | \begin{array}{c} {} \\ {} \end{array}
    \right . \!\!\!\!\!\!\!}
\nc{\mult}{m}       
\nc{\msh}{\ast}
\nc{\ot}{\otimes} \nc{\sot}{{\scriptstyle{\ot}}}
\nc{\otm}{\overline{\ot}}
\nc{\ora}[1]{\stackrel{#1}{\rar}}
\nc{\ola}[1]{\stackrel{#1}{\la}}
\nc{\scs}[1]{\scriptstyle{#1}} \nc{\mrm}[1]{{\rm #1}}
\nc{\margin}[1]{\marginpar{\rm #1}}   
\nc{\dirlim}{\displaystyle{\lim_{\longrightarrow}}\,}
\nc{\invlim}{\displaystyle{\lim_{\longleftarrow}}\,}
\nc{\mvp}{\vspace{0.5cm}} \nc{\svp}{\vspace{2cm}}
\nc{\un}{u}                 
\nc{\vp}{\vspace{8cm}} \nc{\proofbegin}{\noindent{\bf Proof: }}
\nc{\proofend}{$\blacksquare$ \vspace{0.5cm}}
\nc{\sha}{{\mbox{\cyr X}}}  
\nc{\ncsha}{{\mbox{\cyr X}^{\mathrm NC}}} \nc{\ncshao}{{\mbox{\cyr X}^{\mathrm NC,\,0}}}
\nc{\shpr}{\mbox{\cyrs X}}    
\nc{\shprm}{\,\overline{\diamond}\,}    
\nc{\shpro}{\diamond^0}    
\nc{\shprr}{\diamond^r}     
\nc{\shprp}{\shpr_v}
\nc{\shprl}{\shpr_\ell}
\nc{\shprw}{\shpr_w}
\nc{\shpra}{\overline{\diamond}^r}
\nc{\shpru}{\check{\diamond}} \nc{\catpr}{\diamond_l}
\nc{\rcatpr}{\diamond_r} \nc{\lapr}{\diamond_a}
\nc{\sqcupm}{\ot}
\nc{\lepr}{\diamond_e} \nc{\vep}{\varepsilon} \nc{\labs}{\mid\!}
\nc{\rabs}{\!\mid} \nc{\hsha}{\widehat{\sha}}
\nc{\lsha}{\stackrel{\leftarrow}{\sha}}
\nc{\rsha}{\stackrel{\rightarrow}{\sha}} \nc{\lc}{\lfloor}
\nc{\rc}{\rfloor}
\nc{\lm}{\,\slash}
\nc{\rtm}{\backslash\,}
\nc{\sqmon}[1]{\langle #1\rangle}
\nc{\forest}{\calf} \nc{\ass}[1]{\alpha({#1})}
\nc{\altx}{\Lambda_X} \nc{\vecT}{\vec{T}} \nc{\onetree}{{\, \bullet\, }}
\nc{\Ao}{\check{A}}
\nc{\seta}{\underline{\Ao}}
\nc{\deltaa}{\overline{\delta}}
\nc{\trho}{\tilde{\rho}}

\nc{\mmbox}[1]{\mbox{\ #1\ }} \nc{\ann}{\mrm{ann}}
\nc{\ABD}{\mrm{Algebraic Birkhoff Decomposition}\xspace}
\nc{\Aut}{\mrm{Aut}} \nc{\can}{\mrm{can}}
\nc{\bread}{\mrm{b}}
\nc{\colim}{\mrm{colim}}
\nc{\Cont}{\mrm{Cont}} \nc{\rchar}{\mrm{char}}
\nc{\cok}{\mrm{coker}} \nc{\dtf}{{R-{\rm tf}}} \nc{\dtor}{{R-{\rm
tor}}}
\renewcommand{\det}{\mrm{det}}
\nc{\depth}{{\mrm d}}
\nc{\Div}{{\mrm Div}} \nc{\End}{\mrm{End}} \nc{\Ext}{\mrm{Ext}}
\nc{\FG}{\mrm{FG}}
\nc{\fp}{\mrm{fp}}
\nc{\Fil}{\mrm{Fil}} \nc{\Frob}{\mrm{Frob}} \nc{\Gal}{\mrm{Gal}}
\nc{\GL}{\mrm{GL}} \nc{\Hom}{\mrm{Hom}} \nc{\hsr}{\mrm{H}}
\nc{\hpol}{\mrm{HP}} \nc{\id}{\mrm{id}}
\nc{\mht}{\mrm{h}}
\nc{\im}{\mrm{im}}
\nc{\incl}{\mrm{incl}} \nc{\length}{\mrm{length}}
\nc{\LR}{\mrm{LR}} \nc{\mchar}{\rm char}
\nc{\mapped}{operated\xspace}
\nc{\Mapped}{Operated\xspace}
\nc{\NC}{\mrm{NC}}
\nc{\mpart}{\mrm{part}} \nc{\ql}{{\QQ_\ell}} \nc{\pmchar}{\partial\mchar}
\nc{\qp}{{\QQ_p}}
\nc{\rank}{\mrm{rank}} \nc{\rba}{\rm{RBA }} \nc{\rbas}{\rm{RBAs }}
\nc{\rbw}{\rm{RBW }} \nc{\rbws}{\rm{RBWs }} \nc{\rcot}{\mrm{cot}}
\nc{\rest}{\rm{controlled}\xspace}
\nc{\rdef}{\mrm{def}} \nc{\rdiv}{{\rm div}} \nc{\rtf}{{\rm tf}}
\nc{\rtor}{{\rm tor}} \nc{\res}{\mrm{res}}
\nc{\Set}{{\mathbf{Set}}}
\nc{\SL}{\mrm{SL}}
\nc{\Spec}{\mrm{Spec}} \nc{\tor}{\mrm{tor}} \nc{\Tr}{\mrm{Tr}}
\nc{\mtr}{\mrm{sk}}

\nc{\ab}{\mathbf{Ab}} \nc{\Alg}{\mathbf{Alg}}
\nc{\Algo}{\mathbf{Alg}^0} \nc{\Bax}{\mathbf{Bax}}
\nc{\Baxo}{\mathbf{Bax}^0} \nc{\RB}{\mathbf{RB}}
\nc{\RBo}{\mathbf{RB}^0} \nc{\BRB}{\mathbf{RB}}
\nc{\Dend}{\mathbf{DD}} \nc{\bfk}{{\bf k}} \nc{\bfone}{{\bf 1}}
\nc{\base}[1]{{a_{#1}}} \nc{\detail}{\marginpar{\bf More detail}
    \noindent{\bf Need more detail!}
    \svp}
\nc{\Diff}{\mathbf{Diff}} \nc{\gap}{\marginpar{\bf
Incomplete}\noindent{\bf Incomplete!!}
    \svp}
\nc{\FMod}{\mathbf{FMod}} \nc{\mset}{\mathbf{MSet}}
\nc{\rb}{\mathrm{RB}} \nc{\Int}{\mathbf{Int}}
\nc{\Mon}{\mathbf{Mon}}
\nc{\Map}{\mathrm{Map}}
\nc{\mzv}{{\mathrm{MZV}}}
\nc{\remarks}{\noindent{\bf Remarks: }} \nc{\Rep}{\mathbf{Rep}}
\nc{\Rings}{\mathbf{Rings}} \nc{\Sets}{\mathbf{Sets}}
\nc{\DT}{\mathbf{DT}}

\nc{\BA}{{\mathbb A}} \nc{\CC}{{\mathbb C}} \nc{\DD}{{\mathbb D}}
\nc{\EE}{{\mathbb E}} \nc{\FF}{{\mathbb F}} \nc{\GG}{{\mathbb G}}
\nc{\HH}{{\mathbb H}} \nc{\LL}{{\mathbb L}} \nc{\NN}{{\mathbb N}}
\nc{\QQ}{{\mathbb Q}} \nc{\RR}{{\mathbb R}} \nc{\TT}{{\mathbb T}}
\nc{\VV}{{\mathbb V}} \nc{\ZZ}{{\mathbb Z}}


\nc{\cala}{{\mathcal A}} \nc{\calc}{{\mathcal C}}
\nc{\cald}{{\mathcal D}} \nc{\cale}{{\mathcal E}}
\nc{\calf}{{\mathcal F}} \nc{\calfr}{{{\mathcal F}^{\,r}}}
\nc{\calfo}{{\mathcal F}^0} \nc{\calfro}{{\mathcal F}^{\,r,0}}
\nc{\oF}{{\overline{F}}}  \nc{\calg}{{\mathcal G}}
\nc{\calh}{{\mathcal H}} \nc{\cali}{{\mathcal I}}
\nc{\calj}{{\mathcal J}} \nc{\call}{{\mathcal L}}
\nc{\calm}{{\mathcal M}}
\nc{\oM}{\overline{M}}
\nc{\caln}{{\mathcal N}}
\nc{\calo}{{\mathcal O}} \nc{\calp}{{\mathcal P}}
\nc{\calr}{{\mathcal R}} \nc{\calt}{{\mathcal T}}
\nc{\caltr}{{\mathcal T}^{\,r}}
\nc{\calu}{{\mathcal U}} \nc{\calv}{{\mathcal V}}
\nc{\calw}{{\mathcal W}} \nc{\calx}{{\mathcal X}}
\nc{\CA}{\mathcal{A}}

\nc{\fraka}{{\mathfrak a}} \nc{\frakB}{{\mathfrak B}}
\nc{\frakb}{{\mathfrak b}} \nc{\frakd}{{\mathfrak d}}
\nc{\oD}{\overline{D}}
\nc{\frakD}{{\mathcal D}}
\nc{\frakF}{{\mathfrak F}} \nc{\frakg}{{\mathfrak g}}
\nc{\frakI}{{\mathcal I}}
\nc{\frakL}{{\mathcal L}}
\nc{\frakm}{{\mathfrak m}}
\nc{\ofrakm}{\bar{\frakm}}
\nc{\frakM}{{\mathcal M}}
\nc{\frakMo}{{\mathfrak M}^0} \nc{\frakp}{{\mathfrak p}}
\nc{\frakP}{{\mathcal P}}
\nc{\frakR}{{\mathcal R}}
\nc{\frakS}{{\mathcal S}} \nc{\frakSo}{{\mathfrak S}^0}
\nc{\fraks}{{\mathfrak s}} \nc{\os}{\overline{\fraks}}
\nc{\frakT}{{\mathfrak T}}
\nc{\oT}{\overline{T}}
\nc{\frakV}{{\mathcal V}}
\nc{\oV}{\overline{V}}
\nc{\frakW}{{\mathcal W}}
\nc{\frakw}{{\mathfrak w}}
\nc{\oW}{\overline{W}}
\nc{\frakX}{{\mathfrak X}} \nc{\frakXo}{{\mathfrak X}^0}
\nc{\frakx}{{\mathbf x}}
\nc{\frakTx}{\frakT}      
\nc{\frakTa}{\frakT^a}        
\nc{\frakTxo}{\frakTx^0}   
\nc{\caltao}{\calt^{a,0}}   
\nc{\ox}{\overline{\frakx}} \nc{\fraky}{{\mathfrak y}}
\nc{\frakz}{{\mathfrak z}} \nc{\oX}{\overline{X}}

\font\cyr=wncyr10
\font\cyrs=wncyr6
\nc{\redtext}[1]{\textcolor{red}{#1}}

\def\s1tree{\!\!\includegraphics[scale=0.41]{1tree.eps}}
\def\sa2tree{\!\!\includegraphics[scale=0.41]{2tree.eps}}
\def\sb3tree{\!\!\includegraphics[scale=0.41]{3tree.eps}}
\def\sc4tree{\!\!\includegraphics[scale=0.41]{4tree.eps}}

\def\1tree{\!\!\includegraphics[scale=0.51]{1tree.eps}}
\def\2tree{\!\!\includegraphics[scale=0.51]{2tree.eps}}
\def\3tree{\!\!\includegraphics[scale=0.51]{3tree.eps}}
\def\4tree{\!\!\includegraphics[scale=0.51]{4tree.eps}}
\def\5tree{\!\!\includegraphics[scale=0.51]{5tree.eps}}
\def\6tree{\!\!\includegraphics[scale=0.51]{6tree.eps}}
\def\7tree{\!\!\includegraphics[scale=0.51]{7tree.eps}}
\def\8tree{\!\!\includegraphics[scale=0.51]{8tree.eps}}
\def\9tree{\!\!\includegraphics[scale=0.51]{9tree.eps}}
\def\a1tree{\!\!\includegraphics[scale=0.51]{10tree.eps}}
\def\b1tree{\!\!\includegraphics[scale=0.51]{11tree.eps}}
\def\c1tree{\!\!\includegraphics[scale=0.51]{12tree.eps}}
\def\d1tree{\!\!\includegraphics[scale=0.51]{13tree.eps}}
\def\e1tree{\!\!\includegraphics[scale=0.51]{14tree.eps}}
\def\f1tree{\!\!\includegraphics[scale=0.51]{15tree.eps}}
\def\dec1tree{\!\!\includegraphics[scale=0.51]{dectree1.eps}}
\def\xtree{\!\!\includegraphics[scale=0.41]{xtree.eps}}
\def\xyztree{\!\!\includegraphics[scale=0.41]{xyztree.eps}}


\def\ccxtree{\includegraphics[scale=1]{ccxtree}}
\def\ccxtreedec{\includegraphics[scale=1]{ccxtreedec}}
\def\ccetree{\includegraphics[scale=1]{ccetree}}
\def\ccllcrcr{\includegraphics[scale=1]{ccllcrcr}}
\def\cclclcrcr{\includegraphics[scale=.5]{cclclcrcr}}
\def\ccllxryr{\includegraphics[scale=.5]{ccllxryr}}
\def\cclxlyrzr{\includegraphics[scale=1]{cclxlyrzr}}
\def\ccIII{\includegraphics[scale=1]{ccIII}}

\def\lxtree{\includegraphics[scale=1]{lxtree}}
\def\xdtree{\includegraphics[scale=1]{xdtree}}
\def\xytree{\includegraphics[scale=1]{xytree}}
\def\lxdtree{\includegraphics[scale=1]{lxdtree}}
\def\dtree{\includegraphics[scale=0.5]{dtree}}
\def\xyleft{\includegraphics[scale=.5]{xyleft}}
\def\xyright{\includegraphics[scale=.7]{xyright}}
\def\yzdtree{\includegraphics[scale=.7]{yzdtree}}
\def\xyztreea{\includegraphics[scale=.7]{xyztree1}}
\def\xyztreeb{\includegraphics[scale=.7]{xyztree2}}
\def\xyztreec{\includegraphics[scale=.7]{xyztree3}}
\def\tableps{\includegraphics[scale=1.4]{tableps}}

\def\decTree{\!\!\!\includegraphics[scale=1.41]{decTree.eps}}
\def\kyldec31{\!\!\includegraphics[scale=1.5]{kyldec31.eps}}

\def\yldec31{\!\!\includegraphics[scale=1.5]{ntreedec3.eps}}

\def\xyldec43{\!\!\includegraphics[scale=1.0]{mtreedec2.eps}}

\def\DecTreeProd1{\!\!\includegraphics[scale=1.5]{DecTreeProd1.eps}}
\def\TreeProd1{\!\!\includegraphics[scale=1.5]{TreeProd1.eps}}
\def\MPath{\!\!\includegraphics[scale=0.8]{MPath.eps}}
\def\DecMPath{\!\!\includegraphics[scale=0.8]{DecMPath.eps}}
\def\RiseMPath{\!\!\includegraphics[scale=0.8]{RiseMPath.eps}}
\def\DecomMPath{\!\!\includegraphics[scale=1.2]{DecomMPath.eps}}
\def\BijDecMPath{\!\!\includegraphics[scale=0.8]{BijDecMPath.eps}}
\def\MPathProd{\!\!\includegraphics[scale=1]{MPathProd.eps}}


\def\xlb2{{\scalebox{0.40}{ 
\begin{picture}(30,42)(38,-38)
\SetWidth{0.5} \SetColor{Black} \Vertex(45,-3){5.66}
\SetWidth{1.0} \Line(45,-3)(45,-33) \SetWidth{0.5}
\Vertex(45,-33){5.66}
\put(50,-50){\em\huge x}
\end{picture}}}}

\def\ylb2{{\scalebox{0.40}{ 
\begin{picture}(30,42)(38,-38)
\SetWidth{0.5} \SetColor{Black} \Vertex(45,-3){5.66}
\SetWidth{1.0} \Line(45,-3)(45,-33) \SetWidth{0.5}
\Vertex(45,-33){5.66}
\put(50,-50){\em\huge y}
\end{picture}}}}

\def\xyld31{{\scalebox{0.40}{ 
\begin{picture}(42,42)(23,-38)
\SetColor{Black} \SetWidth{0.5} \Vertex(45,-3){5.66}
\Vertex(30,-33){5.66} \Vertex(60,-33){5.66} \SetWidth{1.0}
\Line(45,-3)(30,-33) \Line(60,-33)(45,-3)
\put(20,-58){\em\huge x}
\put(50,-58){\em\huge y}
\end{picture}}}}

\def\xylf41{{\scalebox{0.40}{ 
\begin{picture}(52,80)(38,-10)
\SetColor{Black} \SetWidth{0.5} \Vertex(45,27){5.66}
\Vertex(45,-3){5.66} \SetWidth{1.0} \Line(45,27)(45,-3)
\put(45,-25){\em\huge x}
\SetWidth{0.5} \Vertex(60,57){5.66} \SetWidth{1.0}
\Line(45,27)(60,57) \SetWidth{0.5} \Vertex(75,27){5.66}
\SetWidth{1.0} \Line(75,27)(60,57)
\put(75,7){\em\huge y}
\end{picture}}}}

\def\xylg42{{\scalebox{0.40}{ 
\begin{picture}(52,80)(8,-10)
\SetColor{Black} \SetWidth{0.5} \Vertex(45,27){5.66}
\Vertex(45,-3){5.66} \SetWidth{1.0} \Line(45,27)(45,-3)
\put(39,-25){\em\huge y}
\SetWidth{0.5} \Vertex(15,27){5.66} \Vertex(30,57){5.66}
\SetWidth{1.0} \Line(15,27)(30,57) \Line(45,27)(30,57)
\put(5,2){\em\huge x}
\end{picture}}}}


\def\ta1{{\scalebox{0.25}{ 
\begin{picture}(12,12)(38,-38)
\SetWidth{0.5} \SetColor{Black} \Vertex(45,-33){5.66}
\end{picture}}}}

\def\tb2{{\scalebox{0.25}{ 
\begin{picture}(12,42)(38,-38)
\SetWidth{0.5} \SetColor{Black} \Vertex(45,-3){5.66}
\SetWidth{1.0} \Line(45,-3)(45,-33) \SetWidth{0.5}
\Vertex(45,-33){5.66}
\end{picture}}}}

\def\tc3{{\scalebox{0.25}{ 
\begin{picture}(12,72)(38,-38)
\SetWidth{0.5} \SetColor{Black} \Vertex(45,27){5.66}
\SetWidth{1.0} \Line(45,27)(45,-3) \SetWidth{0.5}
\Vertex(45,-33){5.66} \SetWidth{1.0} \Line(45,-3)(45,-33)
\SetWidth{0.5} \Vertex(45,-3){5.66}
\end{picture}}}}

\def\td31{{\scalebox{0.25}{ 
\begin{picture}(42,42)(23,-38)
\SetWidth{0.5} \SetColor{Black} \Vertex(45,-3){5.66}
\Vertex(30,-33){5.66} \Vertex(60,-33){5.66} \SetWidth{1.0}
\Line(45,-3)(30,-33) \Line(60,-33)(45,-3)
\end{picture}}}}

\def\xtd31{{\scalebox{0.35}{ 
\begin{picture}(70,42)(13,-35)
\SetWidth{0.5} \SetColor{Black} \Vertex(45,-3){5.66}
\Vertex(30,-33){5.66} \Vertex(60,-33){5.66} \SetWidth{1.0}
\Line(45,-3)(30,-33) \Line(60,-33)(45,-3)
\put(38,-38){\em \huge x}
\end{picture}}}}

\def\ytd31{{\scalebox{0.35}{ 
\begin{picture}(70,42)(13,-35)
\SetWidth{0.5} \SetColor{Black} \Vertex(45,-3){5.66}
\Vertex(30,-33){5.66} \Vertex(60,-33){5.66} \SetWidth{1.0}
\Line(45,-3)(30,-33) \Line(60,-33)(45,-3)
\put(38,-38){\em \huge y}
\end{picture}}}}

\def\xldec41r{{\scalebox{0.35}{ 
\begin{picture}(70,42)(13,-45)
\SetColor{Black}
\SetWidth{0.5} \Vertex(45,-3){5.66}
\Vertex(30,-33){5.66} \Vertex(60,-33){5.66}
\Vertex(60,-63){5.66}
\SetWidth{1.0}
\Line(45,-3)(30,-33) \Line(60,-33)(45,-3)
\Line(60,-33)(60,-63)
\put(38,-38){\em \huge x}

\end{picture}}}}

\def\xyrlong{{\scalebox{0.35}{ 
\begin{picture}(70,72)(13,-48)
\SetColor{Black}
\SetWidth{0.5} \Vertex(45,-3){5.66}
\Vertex(30,-33){5.66} \Vertex(60,-33){5.66} \SetWidth{1.0}
\Line(45,-3)(30,-33) \Line(60,-33)(45,-3)
\put(38,-38){\em\huge x}
\SetWidth{0.5}
\Vertex(45,-63){5.66} \Vertex(75,-63){5.66} \SetWidth{1.0}
\Line(60,-33)(45,-63) \Line(60,-33)(75,-63)
\put(55,-63){\em\huge y}
\end{picture}}}}

\def\xyllong{{\scalebox{0.35}{ 
\begin{picture}(70,72)(13,-48)
\SetColor{Black}
\SetWidth{0.5} \Vertex(45,-3){5.66}
\Vertex(30,-33){5.66} \Vertex(60,-33){5.66} \SetWidth{1.0}
\Line(45,-3)(30,-33) \Line(60,-33)(45,-3)
\put(40,-33){\em\huge y}
\SetWidth{0.5}
\Vertex(15,-63){5.66} \Vertex(45,-63){5.66} \SetWidth{1.0}
\Line(30,-33)(15,-63) \Line(30,-33)(45,-63)
\put(25,-63){\em\huge x}
\end{picture}}}}

\def\xyldec43{{\scalebox{0.35}{ 
\begin{picture}(70,62)(13,-25)
\SetColor{Black}
\SetWidth{0.5} \Vertex(45,-3){5.66}
\Vertex(15,-33){5.66} \Vertex(45,-38){5.66}
\Vertex(75,-33){5.66}
\SetWidth{1.0}
\Line(45,-3)(15,-33) \Line(45,-3)(45,-38)
\Line(45,-3)(74,-33)
\put(25,-33){\em\huge x}
\put(50,-33){\em\huge y}
\end{picture}}}}

\def\te4{{\scalebox{0.25}{ 
\begin{picture}(12,102)(38,-8)
\SetWidth{0.5} \SetColor{Black} \Vertex(45,57){5.66}
\Vertex(45,-3){5.66} \Vertex(45,27){5.66} \Vertex(45,87){5.66}
\SetWidth{1.0} \Line(45,57)(45,27) \Line(45,-3)(45,27)
\Line(45,57)(45,87)
\end{picture}}}}

\def\tf41{{\scalebox{0.25}{ 
\begin{picture}(42,72)(38,-8)
\SetWidth{0.5} \SetColor{Black} \Vertex(45,27){5.66}
\Vertex(45,-3){5.66} \SetWidth{1.0} \Line(45,27)(45,-3)
\SetWidth{0.5} \Vertex(60,57){5.66} \SetWidth{1.0}
\Line(45,27)(60,57) \SetWidth{0.5} \Vertex(75,27){5.66}
\SetWidth{1.0} \Line(75,27)(60,57)
\end{picture}}}}

\def\tg42{{\scalebox{0.25}{ 
\begin{picture}(42,72)(8,-8)
\SetWidth{0.5} \SetColor{Black} \Vertex(45,27){5.66}
\Vertex(45,-3){5.66} \SetWidth{1.0} \Line(45,27)(45,-3)
\SetWidth{0.5} \Vertex(15,27){5.66} \Vertex(30,57){5.66}
\SetWidth{1.0} \Line(15,27)(30,57) \Line(45,27)(30,57)
\end{picture}}}}

\def\th43{{\scalebox{0.25}{ 
\begin{picture}(42,42)(8,-8)
\SetWidth{0.5} \SetColor{Black} \Vertex(45,-3){5.66}
\Vertex(15,-3){5.66} \Vertex(30,27){5.66} \SetWidth{1.0}
\Line(15,-3)(30,27) \Line(45,-3)(30,27) \Line(30,27)(30,-3)
\SetWidth{0.5} \Vertex(30,-3){5.66}
\end{picture}}}}

\def\thII43{{\scalebox{0.25}{ 
\begin{picture}(72,57) (68,-128)
    \SetWidth{0.5}
    \SetColor{Black}
    \Vertex(105,-78){5.66}
    \SetWidth{1.5}
    \Line(105,-78)(75,-123)
    \Line(105,-78)(105,-123)
    \Line(105,-78)(135,-123)
    \SetWidth{0.5}
    \Vertex(75,-123){5.66}
    \Vertex(105,-123){5.66}
    \Vertex(135,-123){5.66}
  \end{picture}
  }}}

\def\thj44{{\scalebox{0.25}{ 
\begin{picture}(42,72)(8,-8)
\SetWidth{0.5} \SetColor{Black} \Vertex(30,57){5.66}
\SetWidth{1.0} \Line(30,57)(30,27) \SetWidth{0.5}
\Vertex(30,27){5.66} \SetWidth{1.0} \Line(45,-3)(30,27)
\SetWidth{0.5} \Vertex(45,-3){5.66} \Vertex(15,-3){5.66}
\SetWidth{1.0} \Line(15,-3)(30,27)
\end{picture}}}}

\def\xthj44{{\scalebox{0.35}{ 
\begin{picture}(42,72)(8,-8)
\SetWidth{0.5} \SetColor{Black} \Vertex(30,57){5.66}
\SetWidth{1.0} \Line(30,57)(30,27) \SetWidth{0.5}
\Vertex(30,27){5.66} \SetWidth{1.0} \Line(45,-3)(30,27)
\SetWidth{0.5} \Vertex(45,-3){5.66} \Vertex(15,-3){5.66}
\SetWidth{1.0} \Line(15,-3)(30,27)
\put(25,-3){\em\huge x}
\end{picture}}}}

\def\ti5{{\scalebox{0.25}{ 
\begin{picture}(12,132)(23,-8)
\SetWidth{0.5} \SetColor{Black} \Vertex(30,117){5.66}
\SetWidth{1.0} \Line(30,117)(30,87) \SetWidth{0.5}
\Vertex(30,87){5.66} \Vertex(30,57){5.66} \Vertex(30,27){5.66}
\Vertex(30,-3){5.66} \SetWidth{1.0} \Line(30,-3)(30,27)
\Line(30,27)(30,57) \Line(30,87)(30,57)
\end{picture}}}}

\def\tj51{{\scalebox{0.25}{ 
\begin{picture}(42,102)(53,-38)
\SetWidth{0.5} \SetColor{Black} \Vertex(61,27){4.24}
\SetWidth{1.0} \Line(75,57)(90,27) \Line(60,27)(75,57)
\SetWidth{0.5} \Vertex(90,-3){5.66} \Vertex(60,27){5.66}
\Vertex(75,57){5.66} \Vertex(90,-33){5.66} \SetWidth{1.0}
\Line(90,-33)(90,-3) \Line(90,-3)(90,27) \SetWidth{0.5}
\Vertex(90,27){5.66}
\end{picture}}}}

\def\tk52{{\scalebox{0.25}{ 
\begin{picture}(42,102)(23,-8)
\SetWidth{0.5} \SetColor{Black} \Vertex(60,57){5.66}
\Vertex(45,87){5.66} \SetWidth{1.0} \Line(45,87)(60,57)
\SetWidth{0.5} \Vertex(30,57){5.66} \SetWidth{1.0}
\Line(30,57)(45,87) \SetWidth{0.5} \Vertex(30,-3){5.66}
\SetWidth{1.0} \Line(30,-3)(30,27) \SetWidth{0.5}
\Vertex(30,27){5.66} \SetWidth{1.0} \Line(30,57)(30,27)
\end{picture}}}}

\def\tl53{{\scalebox{0.25}{ 
\begin{picture}(42,102)(8,-8)
\SetWidth{0.5} \SetColor{Black} \Vertex(30,57){5.66}
\Vertex(30,27){5.66} \SetWidth{1.0} \Line(30,57)(30,27)
\SetWidth{0.5} \Vertex(30,87){5.66} \SetWidth{1.0}
\Line(30,27)(45,-3) \SetWidth{0.5} \Vertex(15,-3){5.66}
\SetWidth{1.0} \Line(15,-3)(30,27) \Line(30,57)(30,87)
\SetWidth{0.5} \Vertex(45,-3){5.66}
\end{picture}}}}

\def\tm54{{\scalebox{0.25}{ 
\begin{picture}(42,72)(8,-38)
\SetWidth{0.5} \SetColor{Black} \Vertex(30,-3){5.66}
\SetWidth{1.0} \Line(30,27)(30,-3) \Line(30,-3)(45,-33)
\SetWidth{0.5} \Vertex(15,-33){5.66} \SetWidth{1.0}
\Line(15,-33)(30,-3) \SetWidth{0.5} \Vertex(45,-33){5.66}
\SetWidth{1.0} \Line(30,-33)(30,-3) \SetWidth{0.5}
\Vertex(30,-33){5.66} \Vertex(30,27){5.66}
\end{picture}}}}

\def\tn55{{\scalebox{0.25}{ 
\begin{picture}(42,72)(8,-38)
\SetWidth{0.5} \SetColor{Black} \Vertex(15,-33){5.66}
\Vertex(45,-33){5.66} \Vertex(30,27){5.66} \SetWidth{1.0}
\Line(45,-33)(45,-3) \SetWidth{0.5} \Vertex(45,-3){5.66}
\Vertex(15,-3){5.66} \SetWidth{1.0} \Line(30,27)(45,-3)
\Line(15,-3)(30,27) \Line(15,-3)(15,-33)
\end{picture}}}}

\def\tp56{{\scalebox{0.25}{ 
\begin{picture}(66,111)(0,0)
\SetWidth{0.5} \SetColor{Black} \Vertex(30,66){5.66}
\Vertex(45,36){5.66} \SetWidth{1.0} \Line(30,66)(45,36)
\Line(15,36)(30,66) \SetWidth{0.5} \Vertex(30,6){5.66}
\Vertex(60,6){5.66} \SetWidth{1.0} \Line(60,6)(45,36)
\SetWidth{0.5}
\SetWidth{1.0} \Line(45,36)(30,6) \SetWidth{0.5}
\Vertex(15,36){5.66}
\end{picture}}}}

\def\tq57{{\scalebox{0.25}{ 
\begin{picture}(81,111)(0,0)
\SetWidth{0.5} \SetColor{Black} \Vertex(45,36){5.66}
\Vertex(30,6){5.66} \Vertex(60,6){5.66} \SetWidth{1.0}
\Line(60,6)(45,36) \SetWidth{0.5}
\SetWidth{1.0} \Line(45,36)(30,6) \SetWidth{0.5}
\Vertex(75,36){5.66} \SetWidth{1.0} \Line(45,36)(60,66)
\Line(60,66)(75,36) \SetWidth{0.5} \Vertex(60,66){5.66}
\end{picture}}}}

\def\tr58{{\scalebox{0.25}{ 
\begin{picture}(81,111)(0,0)
\SetWidth{0.5} \SetColor{Black} \Vertex(60,6){5.66}
\Vertex(75,36){5.66} \SetWidth{1.0} \Line(60,66)(75,36)
\SetWidth{0.5} \Vertex(60,66){5.66}
\SetWidth{1.0} \Line(60,36)(60,66) \Line(60,6)(60,36)
\SetWidth{0.5} \Vertex(60,36){5.66} \Vertex(45,36){5.66}
\SetWidth{1.0} \Line(60,66)(45,36)
\end{picture}}}}

\def\ts59{{\scalebox{0.25}{ 
\begin{picture}(81,111)(0,0)
\SetWidth{0.5} \SetColor{Black}
\Vertex(75,36){5.66} \SetWidth{1.0} \Line(60,66)(75,36)
\SetWidth{0.5} \Vertex(60,66){5.66}
\SetWidth{1.0} \Line(60,36)(60,66) \SetWidth{0.5}
\Vertex(60,36){5.66} \Vertex(45,36){5.66} \SetWidth{1.0}
\Line(60,66)(45,36) \Line(75,6)(75,36) \SetWidth{0.5}
\Vertex(75,6){5.66}
\end{picture}}}}

\def\tt591{{\scalebox{0.25}{ 
\begin{picture}(81,111)(0,0)
\SetWidth{0.5} \SetColor{Black}
\Vertex(75,36){5.66} \SetWidth{1.0} \Line(60,66)(75,36)
\SetWidth{0.5} \Vertex(60,66){5.66}
\SetWidth{1.0} \Line(60,36)(60,66) \SetWidth{0.5}
\Vertex(60,36){5.66} \Vertex(45,36){5.66} \SetWidth{1.0}
\Line(60,66)(45,36) \SetWidth{0.5} \Vertex(45,6){5.66}
\SetWidth{1.0} \Line(45,6)(45,36)
\end{picture}}}}


\def\ydec31{\!\!\includegraphics[scale=0.5]{ydec31.eps}}

\def\kyldec31{\!\!\includegraphics[scale=0.5]{kyldec31.eps}}

\def\yldec31{\!\!\includegraphics[scale=0.5]{ntreedec3.eps}}





\def\lta1{{\scalebox{0.25}{ 
\begin{picture}(0,45)(60,-15)
\SetWidth{1.5} \SetColor{Black} \Line(60,30)(60,-15)
\end{picture}}}}

\def\ltb2{{\scalebox{0.25}{ 
\begin{picture}(12,45)(53,-15)
\SetWidth{1.5} \SetColor{Black} \Line(60,30)(60,-15)
\SetWidth{0.5} \Vertex(60,0){5.66}
\end{picture}}}}

\def\ltc3{{\scalebox{0.25}{ 
\begin{picture}(12,75)(53,-15)
\SetWidth{0.5} \SetColor{Black} \Vertex(60,30){5.66}
\SetWidth{1.5} \Line(60,60)(60,-15) \SetWidth{0.5}
\Vertex(60,0){5.66}
\end{picture}}}}

\def\ltd31{{\scalebox{0.25}{ 
\begin{picture}(75,90)(0,0)
\SetWidth{0.5} \SetColor{Black}
\Vertex(60,15){5.66} \SetWidth{1.5} \Line(45,45)(60,15)
\Line(60,15)(75,45) \Line(60,15)(60,0)
\end{picture}}}}

\def\lte4{{\scalebox{0.25}{ 
\begin{picture}(66,120)(0,0)
\SetWidth{0.5} \SetColor{Black}
\Vertex(60,45){5.66} \Vertex(60,75){5.66} \SetWidth{1.5}
\Line(60,105)(60,0) \SetWidth{0.5} \Vertex(60,15){5.66}
\end{picture}}}}

\def\ltf41l{{\scalebox{0.25}{ 
\begin{picture}(75,120)(0,0)
\SetWidth{0.5} \SetColor{Black}
\Vertex(60,15){5.66} \SetWidth{1.5} \Line(60,0)(60,15)
\Line(60,15)(45,45) \Line(60,15)(75,45) \SetWidth{0.5}
\Vertex(45,45){5.66} \SetWidth{1.5} \Line(45,45)(45,75)
\end{picture}}}}

\def\ltg41r{{\scalebox{0.25}{ 
\begin{picture}(81,120)(0,0)
\SetWidth{0.5} \SetColor{Black}
\Vertex(60,15){5.66} \SetWidth{1.5} \Line(75,45)(75,75)
\Line(60,0)(60,15) \Line(60,15)(45,45) \Line(60,15)(75,45)
\SetWidth{0.5} \Vertex(75,45){5.66}
\end{picture}}}}

\delete{
\begin{picture}(81,120)(0,0)
\SetWidth{0.5} \SetColor{Black}
\Vertex(60,15){5.66} \SetWidth{1.5} \Line(75,45)(75,75)
\Line(60,0)(60,15) \Line(60,15)(45,45) \Line(60,15)(75,45)
\SetWidth{0.5} \Vertex(75,45){5.66}
\end{picture}}

\def\lth42{{\scalebox{0.25}{ 
\begin{picture}(75,150)(0,0)
\SetWidth{0.5} \SetColor{Black}
\Vertex(60,45){5.66} \SetWidth{1.5} \Line(60,30)(60,45)
\Line(60,45)(45,75) \Line(60,45)(75,75) \SetWidth{0.5}
\Vertex(60,15){5.66} \SetWidth{1.5} \Line(60,0)(60,30)
\end{picture}}}}

\def\lti43{{\scalebox{0.25}{ 
\begin{picture}(75,150)(0,0)
\SetWidth{0.5} \SetColor{Black}
\SetWidth{1.5} \Line(60,30)(60,45) \SetWidth{0.5}
\Vertex(60,15){5.66} \SetWidth{1.5} \Line(60,0)(60,30)
\Line(60,15)(75,45) \Line(60,15)(45,45)
\end{picture}}}}



\def\ma1{{\scalebox{0.35}{ 
\begin{picture}(45,12)(25,-38)
\SetColor{Black}
\SetWidth{0.5} \Vertex(45,-33){5.66}
\end{picture}}}}

\def\mb2{{\scalebox{0.35}{ 
\begin{picture}(55,32)(5,-38)
\SetColor{Black}
\SetWidth{0.5} \Vertex(15,-33){5.66}
\SetWidth{1.0} \Line(15,-33)(45,-33)
\SetWidth{0.5} \Vertex(45,-33){5.66}
\end{picture}}}}

\def\xmb2{{\scalebox{0.35}{ 
\begin{picture}(55,32)(5,-38)
\SetColor{Black}
\SetWidth{0.5} \Vertex(15,-33){5.66}
\SetWidth{1.0} \Line(15,-33)(45,-33)
\put(25,-23){\em\huge x}
\SetWidth{0.5} \Vertex(45,-33){5.66}
\end{picture}}}}

\def\xpmb2{{\scalebox{0.35}{ 
\begin{picture}(55,32)(5,-38)
\SetColor{Black}
\SetWidth{0.5} \Vertex(15,-33){5.66}
\SetWidth{1.0} \Line(15,-33)(45,-33)
\put(25,-23){\em\huge $x'$}
\SetWidth{0.5} \Vertex(45,-33){5.66}
\end{picture}}}}

\def\motz1{{\scalebox{0.35}{ 
\begin{picture}(55,32)(5,-38)
\SetColor{Black}
\SetWidth{0.5} \Vertex(15,-33){5.66}
\SetWidth{1.0} \Line(15,-33)(45,-33)
\put(25,-23){\em\huge x_{i_1}}
\SetWidth{0.5} \Vertex(45,-33){5.66}
\end{picture}}}}

\def\motzb{{\scalebox{0.35}{ 
\begin{picture}(55,32)(5,-38)
\SetColor{Black}
\SetWidth{0.5} \Vertex(15,-33){5.66}
\SetWidth{1.0} \Line(15,-33)(45,-33)
\put(25,-23){\em\huge x_{i_{b-1}}}
\SetWidth{0.5} \Vertex(45,-33){5.66}
\end{picture}}}}

\def\ymb2{{\scalebox{0.35}{ 
\begin{picture}(55,32)(5,-38)
\SetColor{Black}
\SetWidth{0.5} \Vertex(15,-33){5.66}
\SetWidth{1.0} \Line(15,-33)(45,-33)
\put(25,-23){\em\huge y}
\SetWidth{0.5} \Vertex(45,-33){5.66}
\end{picture}}}}


\def\mc3{{\scalebox{0.35}{ 
\begin{picture}(85,42)(5,-38)
\SetColor{Black}
\SetWidth{0.5} \Vertex(15,-33){5.66}
\SetWidth{1.0} \Line(15,-33)(45,-33)
\SetWidth{0.5} \Vertex(45,-33){5.66}
\SetWidth{1.0} \Line(45,-33)(75,-33)
\SetWidth{0.5} \Vertex(75,-33){5.66}
\end{picture}}}}

\def\xymc3{{\scalebox{0.35}{ 
\begin{picture}(85,42)(5,-38)
\SetColor{Black}
\SetWidth{0.5} \Vertex(15,-33){5.66}
\SetWidth{1.0} \Line(15,-33)(45,-33)
\put(25,-23){\em\huge x}
\SetWidth{0.5} \Vertex(45,-33){5.66}
\SetWidth{1.0} \Line(45,-33)(75,-33)
\put(55,-23){\em \huge y}
\SetWidth{0.5} \Vertex(75,-33){5.66}
\end{picture}}}}

\def\md31{{\scalebox{0.35}{ 
\begin{picture}(100,42)(20,-38)
\SetColor{Black}
\SetWidth{0.5} \Vertex(60,-3){5.66}
\Vertex(30,-33){5.66} \Vertex(90,-33){5.66} \SetWidth{1.0}
\Line(60,-3)(30,-33) \Line(90,-33)(60,-3)
\end{picture}}}}

\def\amd31{{\scalebox{0.35}{ 
\begin{picture}(100,42)(20,-38)
\SetColor{Black}
\SetWidth{0.5} \Vertex(60,-3){5.66}
\Vertex(30,-33){5.66} \Vertex(90,-33){5.66} \SetWidth{1.0}
\Line(60,-3)(30,-33) \put(30,-18){\em\huge $\omega$}
\Line(90,-33)(60,-3) \put(85,-18){\em\huge $\omega$}
\end{picture}}}}

\def\bmd31{{\scalebox{0.35}{ 
\begin{picture}(100,42)(20,-38)
\SetColor{Black}
\SetWidth{0.5} \Vertex(60,-3){5.66}
\Vertex(30,-33){5.66} \Vertex(90,-33){5.66} \SetWidth{1.0}
\Line(60,-3)(30,-33) \put(30,-18){\em\huge $\beta$}
\Line(90,-33)(60,-3) \put(85,-18){\em\huge $\beta$}
\end{picture}}}}

\def\me4{{\scalebox{0.35}{ 
\begin{picture}(110,42)(0,-38)
\SetColor{Black}
\SetWidth{0.5} \Vertex(10,-33){5.66}
\SetWidth{1.0} \Line(10,-33)(40,-33)
\SetWidth{0.5} \Vertex(40,-33){5.66}
\SetWidth{1.0} \Line(40,-33)(70,-33)
\SetWidth{0.5} \Vertex(70,-33){5.66}
\SetWidth{1.0} \Line(70,-33)(100,-33)
\SetWidth{0.5} \Vertex(100,-33){5.66}
\end{picture}}}}

\def\mf41{{\scalebox{0.35}{ 
\begin{picture}(110,42)(0,-38)
\SetColor{Black}
\SetWidth{0.5} \Vertex(10,-33){5.66}
\SetWidth{1.0} \Line(10,-33)(40,-33)
\SetWidth{0.5} \Vertex(40,-33){5.66}
\Vertex(70,-3){5.66} \Vertex(100,-33){5.66}
\SetWidth{1.0}
\Line(40,-33)(70,-3) \Line(100,-33)(70,-3)
\end{picture}}}}

\def\xmf41{{\scalebox{0.35}{ 
\begin{picture}(110,42)(0,-38)
\SetColor{Black}
\SetWidth{0.5} \Vertex(10,-33){5.66}
\SetWidth{1.0} \Line(10,-33)(40,-33)
\put(20,-23){\em\huge x}
\SetWidth{0.5} \Vertex(40,-33){5.66}
\Vertex(70,-3){5.66} \Vertex(100,-33){5.66}
\SetWidth{1.0}
\Line(40,-33)(70,-3) \Line(100,-33)(70,-3)
\end{picture}}}}

\def\axmf41{{\scalebox{0.35}{ 
\begin{picture}(110,42)(0,-38)
\SetColor{Black}
\SetWidth{0.5} \Vertex(10,-33){5.66}
\SetWidth{1.0} \Line(10,-33)(40,-33)
\put(20,-23){\em\huge x}
\SetWidth{0.5} \Vertex(40,-33){5.66}
\Vertex(70,-3){5.66} \Vertex(100,-33){5.66}
\SetWidth{1.0}
\Line(40,-33)(70,-3) \Line(100,-33)(70,-3)
\put(40,-18){\em \huge $\omega$}
\put(90,-18){\em \huge $\omega$}
\end{picture}}}}

\def\mg42{{\scalebox{0.35}{ 
\begin{picture}(110,42)(0,-38)
\SetColor{Black}
\SetWidth{0.5} \Vertex(10,-33){5.66}
\SetWidth{1.0} \Line(70,-33)(100,-33)
\SetWidth{0.5} \Vertex(70,-33){5.66}
\Vertex(40,-3){5.66} \Vertex(100,-33){5.66}
\SetWidth{1.0}
\Line(10,-33)(40,-3) \Line(70,-33)(40,-3)
\end{picture}}}}

\def\mh43{{\scalebox{0.35}{ 
\begin{picture}(110,42)(0,-38)
\SetColor{Black}
\SetWidth{0.5} \Vertex(10,-33){5.66}
\SetWidth{1.0} \Line(10,-33)(40,-3)
\SetWidth{0.5} \Vertex(40,-3){5.66}
\SetWidth{1.0} \Line(40,-3)(70,-3)
\SetWidth{0.5} \Vertex(70,-3){5.66}
\SetWidth{1.0} \Line(70,-3)(100,-33)
\SetWidth{0.5} \Vertex(100,-33){5.66}
\end{picture}}}}

\def\xmh43{{\scalebox{0.35}{ 
\begin{picture}(110,42)(0,-38)
\SetColor{Black}
\SetWidth{0.5} \Vertex(10,-33){5.66}
\SetWidth{1.0} \Line(10,-33)(40,-3)
\SetWidth{0.5} \Vertex(40,-3){5.66}
\SetWidth{1.0} \Line(40,-3)(70,-3)
\put(50,3){\em\huge x}
\SetWidth{0.5} \Vertex(70,-3){5.66}
\SetWidth{1.0} \Line(70,-3)(100,-33)
\SetWidth{0.5} \Vertex(100,-33){5.66}
\end{picture}}}}

\def\ymh43{{\scalebox{0.35}{ 
\begin{picture}(110,42)(0,-38)
\SetColor{Black}
\SetWidth{0.5} \Vertex(10,-33){5.66}
\SetWidth{1.0} \Line(10,-33)(40,-3)
\SetWidth{0.5} \Vertex(40,-3){5.66}
\SetWidth{1.0} \Line(40,-3)(70,-3)
\put(50,3){\em\huge y}
\SetWidth{0.5} \Vertex(70,-3){5.66}
\SetWidth{1.0} \Line(70,-3)(100,-33)
\SetWidth{0.5} \Vertex(100,-33){5.66}
\end{picture}}}}

\def\axmh43{{\scalebox{0.35}{ 
\begin{picture}(110,42)(0,-38)
\SetColor{Black}
\SetWidth{0.5} \Vertex(10,-33){5.66}
\SetWidth{1.0} \Line(10,-33)(40,-3) \put(13,-17){\em \huge $\omega$}
\SetWidth{0.5} \Vertex(40,-3){5.66}
\SetWidth{1.0} \Line(40,-3)(70,-3)
\put(50,3){\em\huge x}
\SetWidth{0.5} \Vertex(70,-3){5.66}
\SetWidth{1.0} \Line(70,-3)(100,-33) \put(90,-17){\em\huge $\omega$}
\SetWidth{0.5} \Vertex(100,-33){5.66}
\end{picture}}}}

\def\mi5{{\scalebox{0.35}{ 
\begin{picture}(140,42)(0,-38)
\SetColor{Black}
\SetWidth{0.5} \Vertex(10,-33){5.66}
\SetWidth{1.0} \Line(10,-33)(40,-33)
\SetWidth{0.5} \Vertex(40,-33){5.66}
\SetWidth{1.0} \Line(40,-33)(70,-33)
\SetWidth{0.5} \Vertex(70,-33){5.66}
\SetWidth{1.0} \Line(70,-33)(100,-33)
\SetWidth{0.5} \Vertex(100,-33){5.66}
\SetWidth{1.0} \Line(100,-33)(130,-33)
\SetWidth{0.5} \Vertex(130,-33){5.66}
\end{picture}}}}

\def\mj51{{\scalebox{0.35}{ 
\begin{picture}(140,42)(0,-38)
\SetColor{Black}
\SetWidth{0.5} \Vertex(10,-33){5.66}
\SetWidth{1.0} \Line(10,-33)(40,-3)
\SetWidth{0.5} \Vertex(40,-3){5.66}
\SetWidth{1.0} \Line(40,-3)(70,-33)
\SetWidth{0.5} \Vertex(70,-33){5.66}
\SetWidth{1.0} \Line(70,-33)(100,-33)
\SetWidth{0.5} \Vertex(100,-33){5.66}
\SetWidth{1.0} \Line(100,-33)(130,-33)
\SetWidth{0.5} \Vertex(130,-33){5.66}
\end{picture}}}}

\def\mk52{{\scalebox{0.35}{ 
\begin{picture}(140,42)(0,-38)
\SetColor{Black}
\SetWidth{0.5} \Vertex(10,-33){5.66}
\SetWidth{1.0} \Line(10,-33)(40,-33)
\SetWidth{0.5} \Vertex(40,-33){5.66}
\SetWidth{1.0} \Line(40,-33)(70,-3)
\SetWidth{0.5} \Vertex(70,-3){5.66}
\SetWidth{1.0} \Line(70,-3)(100,-33)
\SetWidth{0.5} \Vertex(100,-33){5.66}
\SetWidth{1.0} \Line(100,-33)(130,-33)
\SetWidth{0.5} \Vertex(130,-33){5.66}
\end{picture}}}}

\def\ml53{{\scalebox{0.35}{ 
\begin{picture}(140,42)(0,-38)
\SetColor{Black}
\SetWidth{0.5} \Vertex(10,-33){5.66}
\SetWidth{1.0} \Line(10,-33)(40,-33)
\SetWidth{0.5} \Vertex(40,-33){5.66}
\SetWidth{1.0} \Line(40,-33)(70,-33)
\SetWidth{0.5} \Vertex(70,-33){5.66}
\SetWidth{1.0} \Line(70,-33)(100,-3)
\SetWidth{0.5} \Vertex(100,-3){5.66}
\SetWidth{1.0} \Line(100,-3)(130,-33)
\SetWidth{0.5} \Vertex(130,-33){5.66}
\end{picture}}}}

\def\mm54{{\scalebox{0.35}{ 
\begin{picture}(140,42)(0,-38)
\SetColor{Black}
\SetWidth{0.5} \Vertex(10,-33){5.66}
\SetWidth{1.0} \Line(10,-33)(40,-3)
\SetWidth{0.5} \Vertex(40,-3){5.66}
\SetWidth{1.0} \Line(40,-3)(70,-3)
\SetWidth{0.5} \Vertex(70,-3){5.66}
\SetWidth{1.0} \Line(70,-3)(100,-33)
\SetWidth{0.5} \Vertex(100,-33){5.66}
\SetWidth{1.0} \Line(100,-33)(130,-33)
\SetWidth{0.5} \Vertex(130,-33){5.66}
\end{picture}}}}

\def\xymm54{{\scalebox{0.35}{ 
\begin{picture}(140,42)(0,-38)
\SetColor{Black}
\SetWidth{0.5} \Vertex(10,-33){5.66}
\SetWidth{1.0} \Line(10,-33)(40,-3)
\SetWidth{0.5} \Vertex(40,-3){5.66}
\SetWidth{1.0} \Line(40,-3)(70,-3)
\put(50,3){\em\huge x}
\SetWidth{0.5} \Vertex(70,-3){5.66}
\SetWidth{1.0} \Line(70,-3)(100,-33)
\SetWidth{0.5} \Vertex(100,-33){5.66}
\SetWidth{1.0} \Line(100,-33)(130,-33)
\put(110,-27){\em\huge y}
\SetWidth{0.5} \Vertex(130,-33){5.66}
\end{picture}}}}

\def\mn55{{\scalebox{0.35}{ 
\begin{picture}(140,42)(0,-38)
\SetColor{Black}
\SetWidth{0.5} \Vertex(10,-33){5.66}
\SetWidth{1.0} \Line(10,-33)(40,-33)
\SetWidth{0.5} \Vertex(40,-33){5.66}
\SetWidth{1.0} \Line(40,-33)(70,-3)
\SetWidth{0.5} \Vertex(70,-3){5.66}
\SetWidth{1.0} \Line(70,-3)(100,-3)
\SetWidth{0.5} \Vertex(100,-3){5.66}
\SetWidth{1.0} \Line(100,-3)(130,-33)
\SetWidth{0.5} \Vertex(130,-33){5.66}
\end{picture}}}}

\def\xymn55{{\scalebox{0.35}{ 
\begin{picture}(140,42)(0,-38)
\SetColor{Black}
\SetWidth{0.5} \Vertex(10,-33){5.66}
\SetWidth{1.0} \Line(10,-33)(40,-33)
\put(20,-27){\em\huge x}
\SetWidth{0.5} \Vertex(40,-33){5.66}
\SetWidth{1.0} \Line(40,-33)(70,-3)
\SetWidth{0.5} \Vertex(70,-3){5.66}
\SetWidth{1.0} \Line(70,-3)(100,-3)
\put(80,3){\em\huge y}
\SetWidth{0.5} \Vertex(100,-3){5.66}
\SetWidth{1.0} \Line(100,-3)(130,-33)
\SetWidth{0.5} \Vertex(130,-33){5.66}
\end{picture}}}}

\def\mo56{{\scalebox{0.35}{ 
\begin{picture}(140,42)(0,-38)
\SetColor{Black}
\SetWidth{0.5} \Vertex(10,-33){5.66}
\SetWidth{1.0} \Line(10,-33)(40,-3)
\SetWidth{0.5} \Vertex(40,-3){5.66}
\SetWidth{1.0} \Line(40,-3)(70,-3)
\SetWidth{0.5} \Vertex(70,-3){5.66}
\SetWidth{1.0} \Line(70,-3)(100,-3)
\SetWidth{0.5} \Vertex(100,-3){5.66}
\SetWidth{1.0} \Line(100,-3)(130,-33)
\SetWidth{0.5} \Vertex(130,-33){5.66}
\end{picture}}}}

\def\xymo56{{\scalebox{0.35}{ 
\begin{picture}(140,42)(0,-38)
\SetColor{Black}
\SetWidth{0.5} \Vertex(10,-33){5.66}
\SetWidth{1.0} \Line(10,-33)(40,-3)
\SetWidth{0.5} \Vertex(40,-3){5.66}
\SetWidth{1.0} \Line(40,-3)(70,-3)
\put(50,3){\em\huge x}
\SetWidth{0.5} \Vertex(70,-3){5.66}
\SetWidth{1.0} \Line(70,-3)(100,-3)
\put(80,3){\em\huge y}
\SetWidth{0.5} \Vertex(100,-3){5.66}
\SetWidth{1.0} \Line(100,-3)(130,-33)
\SetWidth{0.5} \Vertex(130,-33){5.66}
\end{picture}}}}

\def\mpp57{{\scalebox{0.35}{ 
\begin{picture}(140,42)(0,-38)
\SetColor{Black}
\SetWidth{0.5} \Vertex(10,-33){5.66}
\SetWidth{1.0} \Line(10,-33)(40,-3)
\SetWidth{0.5} \Vertex(40,-3){5.66}
\SetWidth{1.0} \Line(40,-3)(70,-33)
\SetWidth{0.5} \Vertex(70,-33){5.66}
\SetWidth{1.0} \Line(70,-33)(100,-3)
\SetWidth{0.5} \Vertex(100,-3){5.66}
\SetWidth{1.0} \Line(100,-3)(130,-33)
\SetWidth{0.5} \Vertex(130,-33){5.66}
\end{picture}}}}

\def\abmpp57{{\scalebox{0.35}{ 
\begin{picture}(140,62)(0,-38)
\SetColor{Black}
\SetWidth{0.5} \Vertex(10,-33){5.66}
\SetWidth{1.0} \Line(10,-33)(40,-3) \put(12,-13){\em \huge $\omega$}
\SetWidth{0.5} \Vertex(40,-3){5.66}
\SetWidth{1.0} \Line(40,-3)(70,-33) \put(50,-13){\em\huge $\omega$}
\SetWidth{0.5} \Vertex(70,-33){5.66}
\SetWidth{1.0} \Line(70,-33)(100,-3) \put(75,-13){\em\huge $\beta$}
\SetWidth{0.5} \Vertex(100,-3){5.66}
\SetWidth{1.0} \Line(100,-3)(130,-33) \put(120,-13){\em\huge $\beta$}
\SetWidth{0.5} \Vertex(130,-33){5.66}
\end{picture}}}}

\def\mq58{{\scalebox{0.35}{ 
\begin{picture}(140,62)(0,-38)
\SetColor{Black}
\SetWidth{0.5} \Vertex(10,-33){5.66}
\SetWidth{1.0} \Line(10,-33)(40,-3)
\SetWidth{0.5} \Vertex(40,-3){5.66}
\SetWidth{1.0} \Line(40,-3)(70,23)
\SetWidth{0.5} \Vertex(70,23){5.66}
\SetWidth{1.0} \Line(70,23)(100,-3)
\SetWidth{0.5} \Vertex(100,-3){5.66}
\SetWidth{1.0} \Line(100,-3)(130,-33)
\SetWidth{0.5} \Vertex(130,-33){5.66}
\end{picture}}}}

\def\abmq58{{\scalebox{0.35}{ 
\begin{picture}(140,62)(0,-38)
\SetColor{Black}
\SetWidth{0.5} \Vertex(10,-33){5.66}
\SetWidth{1.0} \Line(10,-33)(40,-3)
\put(10,-13){\em\huge $\omega$}
\SetWidth{0.5} \Vertex(40,-3){5.66}
\SetWidth{1.0} \Line(40,-3)(70,23)
\put(40,13){\em\huge $\beta$}
\SetWidth{0.5} \Vertex(70,23){5.66}
\SetWidth{1.0} \Line(70,23)(100,-3) \put(90,13){\em\huge$\beta$}
\SetWidth{0.5} \Vertex(100,-3){5.66}
\SetWidth{1.0} \Line(100,-3)(130,-33)
\put(120,-13){\em\huge$\omega$}
\SetWidth{0.5} \Vertex(130,-33){5.66}
\end{picture}}}}

\def\uuxdyd{{\scalebox{0.35}{ 
\begin{picture}(200,62)(-30,-68)
\SetColor{Black}
\SetWidth{0.5} \Vertex(10,-33){5.66}
\SetWidth{1.0} \Line(10,-33)(40,-3)
\SetWidth{0.5} \Vertex(40,-3){5.66}
\SetWidth{1.0} \Line(40,-3)(70,-3)
\put(50,3){\em\huge x}
\SetWidth{0.5} \Vertex(70,-3){5.66}
\SetWidth{1.0} \Line(70,-3)(100,-33)
\SetWidth{0.5} \Vertex(100,-33){5.66}
\SetWidth{1.0} \Line(100,-33)(130,-33)
\put(110,-27){\em\huge y}
\SetWidth{0.5} \Vertex(130,-33){5.66}
\Vertex(-20,-63){5.66} \Vertex(160,-63){5.66}
\SetWidth{1.0} \Line(-20,-63)(10,-33)
\Line(130,-33)(160,-63)
\end{picture}}}}

\def\uxuydd{{\scalebox{0.35}{ 
\begin{picture}(200,62)(-30,-68)
\SetColor{Black}
\SetWidth{0.5} \Vertex(10,-33){5.66}
\SetWidth{1.0} \Line(10,-33)(40,-33)
\put(20,-27){\em\huge x}
\SetWidth{0.5} \Vertex(40,-33){5.66}
\SetWidth{1.0} \Line(40,-33)(70,-3)
\SetWidth{0.5} \Vertex(70,-3){5.66}
\SetWidth{1.0} \Line(70,-3)(100,-3)
\put(80,3){\em\huge y}
\SetWidth{0.5} \Vertex(100,-3){5.66}
\SetWidth{1.0} \Line(100,-3)(130,-33)
\SetWidth{0.5} \Vertex(130,-33){5.66}
\Vertex(-20,-63){5.66} \Vertex(160,-63){5.66}
\SetWidth{1.0} \Line(-20,-63)(10,-33)
\Line(130,-33)(160,-63)
\end{picture}}}}

\def\bigdecm{{\scalebox{0.5}{ 
\begin{picture}(140,82)(250,-10)
\SetColor{Black}
\DashLine(10,-33)(610,-33){10}

\SetColor{Black}
\SetWidth{0.5} \Vertex(10,-33){5.66}
\SetWidth{1.0} \Line(10,-33)(40,-3)
\put(10,-13){\em\huge $\alpha$}
\SetWidth{0.5} \Vertex(40,-3){5.66}
\SetWidth{1.0} \Line(40,-3)(70,23)
\put(40,13){\em\huge $\beta$}
\SetWidth{0.5} \Vertex(70,23){5.66}
\SetWidth{1.0} \Line(70,23)(100,23)
\put(80,33){\em\huge $a$}
\SetWidth{0.5} \Vertex(100,23){5.66}
\SetWidth{1.0} \Line(100,23)(130,53)
\put(100,43){\em\huge $\gamma$}
\SetWidth{0.5} \Vertex(130,53){5.66}
\SetWidth{1.0} \Line(130,53)(160,53)
\put(140,63){\em\huge $b$}
\SetWidth{0.5} \Vertex(160,53){5.66}
\SetWidth{1.0} \Line(160,53)(190,53)
\put(170,63){\em\huge $c$}
\SetWidth{0.5} \Vertex(190,53){5.66}
\SetWidth{1.0} \Line(190,53)(220,23)
\put(205,43){\em\huge $\gamma$}
\SetWidth{0.5} \Vertex(220,23){5.66}
\SetWidth{1.0} \Line(220,23)(250,23)
\put(230,33){\em\huge $d$}
\SetWidth{0.5} \Vertex(250,23){5.66}
\SetWidth{1.0} \Line(250,23)(280,-3)
\put(265,13){\em\huge $\beta$}
\SetWidth{0.5} \Vertex(280,-3){5.66}
\SetWidth{1.0} \Line(280,-3)(310,-3)
\put(290,03){\em\huge $e$}
\SetWidth{0.5} \Vertex(310,-3){5.66}
\SetWidth{1.0} \Line(310,-3)(340,23)
\put(315,13){\em\huge $\delta$}
\SetWidth{0.5} \Vertex(340,23){5.66}
\SetWidth{1.0} \Line(340,23)(370,23)
\put(350,33){\em\huge $f$}
\SetWidth{0.5} \Vertex(370,23){5.66}
\SetWidth{1.0} \Line(370,23)(400,53)
\put(378,38){\em\huge $\sigma$}
\SetWidth{0.5} \Vertex(400,53){5.66}
\SetWidth{1.0} \Line(400,53)(430,53)
\put(410,63){\em\huge $g$}
\SetWidth{0.5} \Vertex(430,53){5.66}
\SetWidth{1.0} \Line(430,53)(460,23)
\put(450,38){\em\huge $\sigma$}
\SetWidth{0.5} \Vertex(460,23){5.66}
\SetWidth{1.0} \Line(460,23)(490,53)
\put(468,38){\em\huge $\tau$}
\SetWidth{0.5} \Vertex(490,53){5.66}
\SetWidth{1.0} \Line(490,53)(520,53)
\put(500,63){\em\huge $h$}
\SetWidth{0.5} \Vertex(520,53){5.66}
\SetWidth{1.0} \Line(520,53)(550,23)
\put(540,38){\em\huge $\tau$}
\SetWidth{0.5} \Vertex(550,23){5.66}
\SetWidth{1.0} \Line(550,23)(580,-3)
\put(570,8){\em\huge $\delta$}
\SetWidth{0.5} \Vertex(580,-3){5.66}
\SetWidth{1.0} \Line(580,-3)(610,-33)
\put(600,-18){\em\huge $\alpha$}
\SetWidth{0.5} \Vertex(610,-33){5.66}

\end{picture}}}}

\def\bigdecl{{\scalebox{0.5}{ 
\begin{picture}(140,82)(250,-10)
\SetColor{Black}
\DashLine(10,-33)(610,-33){10}

\SetColor{Black}
\SetWidth{0.5} \Vertex(10,-33){5.66}
\SetWidth{1.0} \Line(10,-33)(40,-3)
\SetWidth{0.5} \Vertex(40,-3){5.66}
\SetWidth{1.0} \Line(40,-3)(70,23)
\SetWidth{0.5} \Vertex(70,23){5.66}
\SetWidth{1.0} \Line(70,23)(100,23)
\put(80,33){\em\huge $a$}
\SetWidth{0.5} \Vertex(100,23){5.66}
\SetWidth{1.0} \Line(100,23)(130,53)
\SetWidth{0.5} \Vertex(130,53){5.66}
\SetWidth{1.0} \Line(130,53)(160,53)
\put(140,63){\em\huge $b$}
\SetWidth{0.5} \Vertex(160,53){5.66}
\SetWidth{1.0} \Line(160,53)(190,53)
\put(170,63){\em\huge $c$}
\SetWidth{0.5} \Vertex(190,53){5.66}
\SetWidth{1.0} \Line(190,53)(220,23)
\SetWidth{0.5} \Vertex(220,23){5.66}
\SetWidth{1.0} \Line(220,23)(250,23)
\put(230,33){\em\huge $d$}
\SetWidth{0.5} \Vertex(250,23){5.66}
\SetWidth{1.0} \Line(250,23)(280,-3)
\SetWidth{0.5} \Vertex(280,-3){5.66}
\SetWidth{1.0} \Line(280,-3)(310,-3)
\put(290,03){\em\huge $e$}
\SetWidth{0.5} \Vertex(310,-3){5.66}
\SetWidth{1.0} \Line(310,-3)(340,23)
\SetWidth{0.5} \Vertex(340,23){5.66}
\SetWidth{1.0} \Line(340,23)(370,53)
\SetWidth{0.5} \Vertex(370,53){5.66}
\SetWidth{1.0} \Line(370,53)(400,53)
\put(380,58){\em\huge $f$}
\SetWidth{0.5} \Vertex(400,53){5.66}
\SetWidth{1.0} \Line(400,53)(430,23)
\SetWidth{0.5} \Vertex(430,23){5.66}
\SetWidth{1.0} \Line(430,23)(460,23)
\put(440,33){\em\huge $g$}
\SetWidth{0.5} \Vertex(460,23){5.66}
\SetWidth{1.0} \Line(460,23)(490,53)
\SetWidth{0.5} \Vertex(490,53){5.66}
\SetWidth{1.0} \Line(490,53)(520,53)
\put(500,63){\em\huge $h$}
\SetWidth{0.5} \Vertex(520,53){5.66}
\SetWidth{1.0} \Line(520,53)(550,23)
\SetWidth{0.5} \Vertex(550,23){5.66}
\SetWidth{1.0} \Line(550,23)(580,-3)
\SetWidth{0.5} \Vertex(580,-3){5.66}
\SetWidth{1.0} \Line(580,-3)(610,-33)
\SetWidth{0.5} \Vertex(610,-33){5.66}

\end{picture}}}}

\def\bigdect{{\scalebox{0.4}{ 
\begin{picture}(140,120)(0,-60)
\SetColor{Black}
\SetWidth{0.5} \Vertex(70,60){5.66}
\put(48,60){\em\huge$\alpha$}
\SetWidth{1.0} \Line(70,60)(0,20)
\SetWidth{0.5} \Vertex(0,20){5.66}
\put(-15,25){\em\huge$\beta$}
\SetWidth{1.0} \Line(70,60)(70,20)
\SetWidth{0.5} \Vertex(70,20){5.66}
\put(50,20){\em\huge$e$}
\SetWidth{1.0} \Line(70,60)(140,20)
\SetWidth{0.5} \Vertex(140,20){5.66}
\put(150,25){\em\huge$\delta$}

\SetWidth{1.0} \Line(0,20)(-50,-20)
\SetWidth{0.5} \Vertex(-50,-20){5.66}
\put(-70,-20){\em\huge $a$}
\SetWidth{1.0} \Line(0,20)(0,-20)
\SetWidth{0.5} \Vertex(0,-20){5.66}
\put(-20,-20){\em\huge$\gamma$}
\SetWidth{1.0} \Line(0,20)(50,-20)
\SetWidth{0.5} \Vertex(50,-20){5.66}
\put(50,-38){\em\huge$d$}

\SetWidth{1.0} \Line(0,-20)(-30,-50)
\SetWidth{0.5} \Vertex(-30,-50){5.66}
\put(-45,-68){\em\huge$b$}
\SetWidth{1.0} \Line(0,-20)(30,-50)
\SetWidth{0.5} \Vertex(30,-50){5.66}
\put(25,-68){\em\huge$c$}

\SetWidth{1.0} \Line(140,20)(100,-10)
\SetWidth{0.5} \Vertex(100,-10){5.66}
\put(80,-20){\em\huge$f$}
\SetWidth{1.0} \Line(140,20)(140,-20)
\SetWidth{0.5} \Vertex(140,-20){5.66}
\put(150,-30){\em\huge$\sigma$}
\SetWidth{1.0} \Line(140,-20)(140,-60)
\SetWidth{0.5} \Vertex(140,-60){5.66}
\put(150,-70){\em\huge$g$}
\SetWidth{1.0} \Line(140,20)(180,-10)
\SetWidth{0.5} \Vertex(180,-10){5.66}
\put(190,-30){\em\huge$\tau$}
\SetWidth{1.0} \Line(180,-10)(180,-60)
\SetWidth{0.5} \Vertex(180,-60){5.66}
\put(190,-70){\em\huge$h$}

\end{picture}}}}

\def\bigdectl{{\scalebox{0.4}{ 
\begin{picture}(140,120)(0,-60)
\SetColor{Black}
\SetWidth{0.5} \Vertex(70,60){5.66}
\SetWidth{1.0} \Line(70,60)(0,20)
\SetWidth{0.5} \Vertex(0,20){5.66}
\SetWidth{1.0} \Line(70,60)(70,20)
\SetWidth{0.5} \Vertex(70,20){5.66}
\put(50,20){\em\huge$e$}
\SetWidth{1.0} \Line(70,60)(140,20)
\SetWidth{0.5} \Vertex(140,20){5.66}

\SetWidth{1.0} \Line(0,20)(-50,-20)
\SetWidth{0.5} \Vertex(-50,-20){5.66}
\put(-70,-20){\em\huge $a$}
\SetWidth{1.0} \Line(0,20)(0,-20)
\SetWidth{0.5} \Vertex(0,-20){5.66}
\SetWidth{1.0} \Line(0,20)(50,-20)
\SetWidth{0.5} \Vertex(50,-20){5.66}
\put(50,-38){\em\huge$d$}

\SetWidth{1.0} \Line(0,-20)(-30,-50)
\SetWidth{0.5} \Vertex(-30,-50){5.66}
\put(-45,-68){\em\huge$b$}
\SetWidth{1.0} \Line(0,-20)(30,-50)
\SetWidth{0.5} \Vertex(30,-50){5.66}
\put(25,-68){\em\huge$c$}

\SetWidth{1.0} \Line(140,20)(100,-10)
\SetWidth{0.5} \Vertex(100,-10){5.66}
\put(80,-20){\em\huge$f$}
\SetWidth{1.0} \Line(140,20)(140,-20)
\SetWidth{0.5} \Vertex(140,-20){5.66}
\SetWidth{1.0} \Line(140,-20)(140,-60)
\SetWidth{0.5} \Vertex(140,-60){5.66}
\put(150,-70){\em\huge$g$}
\SetWidth{1.0} \Line(140,20)(180,-10)
\SetWidth{0.5} \Vertex(180,-10){5.66}
\SetWidth{1.0} \Line(180,-10)(180,-60)
\SetWidth{0.5} \Vertex(180,-60){5.66}
\put(190,-70){\em\huge$h$}
\end{picture}}}}

\def\bigdectls{{\scalebox{0.4}{ 
\begin{picture}(180,120)(0,-60)
\SetColor{Black}
\SetWidth{0.5} \Vertex(70,60){5.66}
\SetWidth{1.0} \Line(70,60)(0,20)
\SetWidth{0.5} \Vertex(0,20){5.66}
\SetWidth{1.0} \Line(70,60)(70,20)
\SetWidth{0.5} \Vertex(70,20){5.66}
\put(50,20){\em\huge$e$}
\SetWidth{1.0} \Line(70,60)(140,20)
\SetWidth{0.5} \Vertex(140,20){5.66}

\SetWidth{1.0} \Line(0,20)(-50,-20)
\SetWidth{0.5} \Vertex(-50,-20){5.66}
\put(-70,-20){\em\huge $a$}
\SetWidth{1.0} \Line(0,20)(0,-20)
\SetWidth{0.5} \Vertex(0,-20){5.66}
\SetWidth{1.0} \Line(0,20)(50,-20)
\SetWidth{0.5} \Vertex(50,-20){5.66}
\put(50,-38){\em\huge$d$}

\SetWidth{1.0} \Line(0,-20)(-30,-50)
\SetWidth{0.5} \Vertex(-30,-50){5.66}
\put(-45,-68){\em\huge$b$}
\SetWidth{1.0} \Line(0,-20)(30,-50)
\SetWidth{0.5} \Vertex(30,-50){5.66}
\put(25,-68){\em\huge$c$}

\SetWidth{1.0} \Line(140,20)(100,-10)
\SetWidth{0.5} \Vertex(100,-10){5.66}
\SetWidth{1.0} \Line(100,-10)(100,-60)
\SetWidth{0.5} \Vertex(100,-60){5.66}
\put(150,-40){\em\huge$g$}
\SetWidth{1.0} \Line(140,20)(140,-20)
\SetWidth{0.5} \Vertex(140,-20){5.66}
\put(80,-70){\em\huge$f$}
\SetWidth{1.0} \Line(140,20)(180,-10)
\SetWidth{0.5} \Vertex(180,-10){5.66}
\SetWidth{1.0} \Line(180,-10)(180,-60)
\SetWidth{0.5} \Vertex(180,-60){5.66}
\put(190,-70){\em\huge$h$}
\end{picture}}}}

\def\bigdecta{{\scalebox{0.4}{ 
\begin{picture}(180,120)(-40,-60)
\SetColor{Black}
\SetWidth{0.5} \Vertex(70,60){5.66}
\SetWidth{1.0} \Line(70,60)(0,20)
\SetWidth{0.5} \Vertex(0,20){5.66}
\put(60,30){\em\huge$e$}
\SetWidth{1.0} \Line(70,60)(140,20)
\SetWidth{0.5} \Vertex(140,20){5.66}

\SetWidth{1.0} \Line(0,20)(-50,-20)
\SetWidth{0.5} \Vertex(-50,-20){5.66}
\put(-20,-10){\em\huge $a$}
\SetWidth{1.0} \Line(0,20)(0,-20)
\SetWidth{0.5} \Vertex(0,-20){5.66}
\SetWidth{1.0} \Line(0,20)(50,-20)
\SetWidth{0.5} \Vertex(50,-20){5.66}
\put(10,-20){\em\huge$d$}

\SetWidth{1.0} \Line(0,-20)(-30,-50)
\Line(0,-20)(0,-50)
\SetWidth{0.5} \Vertex(-30,-50){5.66}
\Vertex(0,-50){5.66}
\put(-22,-68){\em\huge$b$}
\SetWidth{1.0} \Line(0,-20)(30,-50)
\SetWidth{0.5} \Vertex(30,-50){5.66}
\put(8,-68){\em\huge$c$}

\SetWidth{1.0} \Line(140,20)(100,-20)
\SetWidth{0.5} \Vertex(100,-20){5.66}
\SetWidth{1.0} \Line(100,-20)(70,-50)
\Line(100,-20)(130,-50)
\SetWidth{0.5} \Vertex(70,-50){5.66}
\Vertex(130,-50){5.66}
\put(90,-60){\em\huge$f$}
\put(140,-10){\em\huge$g$}
\SetWidth{1.0} \Line(140,20)(180,-10)
\SetWidth{0.5} \Vertex(180,-10){5.66}
\SetWidth{1.0} \Line(180,-10)(150,-50)
\Line(180,-10)(210,-50)
\SetWidth{0.5} \Vertex(150,-50){5.66}
\Vertex(210,-50){5.66}
\put(175,-60){\em\huge$h$}
\end{picture}}}}

\def\motzstand{{\scalebox{0.35}{ 
\begin{picture}(300,62)(-20,-68)
\SetColor{Black}
\SetWidth{0.5} \Vertex(-50,-63){5.66}
\SetWidth{1.0} \Line(-50,-63)(-20,-63)
\Vertex(-20,-63){5.66}
\put(-40,-55){\huge $x_1$}
\SetWidth{1.0} \Line(-20,-63)(10,-33)
\SetWidth{0.5} \Vertex(10,-33){5.66}
\SetWidth{1.0} \Line(10,-33)(40,-33)
\put(15,-22){\huge $x_2$}
\SetWidth{0.5} \Vertex(40,-33){5.66}
\SetWidth{1.0} \Line(40,-33)(70,-63)
\SetWidth{0.5} \Vertex(70,-63){5.66}
\SetWidth{1.0} \Line(70,-63)(100,-63)
\put(75,-52){\huge $x_3$}
\SetWidth{0.5} \Vertex(100,-63){5.66}
\SetWidth{1.0} \Line(100,-63)(130,-33)
\SetWidth{0.5} \Vertex(130,-33){5.66}
\SetWidth{1.0} \Line(130,-33)(160,-33)
\put(135,-22){\huge $x_4$}
\SetWidth{0.5} \Vertex(160,-33){5.66}
\SetWidth{1.0} \Line(160,-33)(190,-63)
\SetWidth{0.5} \Vertex(190,-63){5.66}
\SetWidth{1.0} \Line(190,-63)(220,-63)
\put(195,-52){\huge $x_5$}
\SetWidth{0.5} \Vertex(220,-63){5.66}
\SetWidth{1.0} \Line(220,-63)(250,-63)
\put(225,-52){\huge $x_6$}
\SetWidth{0.5} \Vertex(250,-63){5.66}
\end{picture}}}}

\def\rhsmotzstand{{\scalebox{0.35}{ 
\begin{picture}(700,62)(-120,-68)
\SetColor{Black}
\SetWidth{0.5} \Vertex(-100,-63){5.66}
\put(-75,-60){\huge $\circ$}
\SetWidth{0.5} \Vertex(-50,-63){5.66}
\SetWidth{1.0} \Line(-50,-63)(-20,-63)
\put(-40,-55){\huge $x_1$}
\Vertex(-20,-63){5.66}
\put(5,-60){\huge $\circ$}
\SetWidth{0.5} \Vertex(30,-63){5.66}

\SetWidth{1.0} \Line(30,-63)(60,-33)
\SetWidth{0.5} \Vertex(60,-33){5.66}
\SetWidth{1.0} \Line(60,-33)(90,-33)
\put(65,-22){\huge $x_2$}
\SetWidth{0.5} \Vertex(90,-33){5.66}
\SetWidth{1.0} \Line(90,-33)(120,-63)
\SetWidth{0.5} \Vertex(120,-63){5.66}
\put(135,-60){\huge $\circ$}

\SetWidth{0.5} \Vertex(170,-63){5.66}
\SetWidth{1.0} \Line(170,-63)(200,-63)
\put(175,-52){\huge $x_3$}
\SetWidth{0.5} \Vertex(200,-63){5.66}
\put(225,-60){\huge $\circ$}

\SetWidth{0.5} \Vertex(250,-63){5.66}
\SetWidth{1.0} \Line(250,-63)(280,-33)
\SetWidth{0.5} \Vertex(280,-33){5.66}
\SetWidth{1.0} \Line(280,-33)(310,-33)
\put(290,-22){\huge $x_4$}
\SetWidth{0.5} \Vertex(310,-33){5.66}
\SetWidth{1.0} \Line(310,-33)(340,-63)
\SetWidth{0.5} \Vertex(340,-63){5.66}
\put(365,-60){\huge $\circ$}

\SetWidth{0.5} \Vertex(390,-63){5.66}
\SetWidth{1.0} \Line(390,-63)(420,-63)
\put(395,-52){\huge $x_5$}
\SetWidth{0.5} \Vertex(420,-63){5.66}
\put(445,-60){\huge $\circ$}

\SetWidth{0.5} \Vertex(470,-63){5.66}
\put(495,-60){\huge $\circ$}

\SetWidth{0.5} \Vertex(520,-63){5.66}
\SetWidth{1.0} \Line(520,-63)(550,-63)
\put(525,-52){\huge $x_6$}
\SetWidth{0.5} \Vertex(550,-63){5.66}
\put(575,-60){\huge $\circ$}

\SetWidth{0.5} \Vertex(600,-63){5.66}
\end{picture}}}}



\def\lxtree{\includegraphics[scale=.7]{lxtree}}
\def\xdtree{\includegraphics[scale=.7]{xdtree}}
\def\xdtreend{\includegraphics[scale=.2]{xdtreend}} 
\def\xytree{\includegraphics[scale=.7]{xytree}}
\def\xytreend{\includegraphics[scale=.2]{xytreend}} 

\def\xdtreed{\includegraphics[scale=.3]{xdtreed}}
\def\xytreed{\includegraphics[scale=.3]{xytreed}}
\def\lxdtree{\includegraphics[scale=.7]{lxdtree}}
\def\dtree{\includegraphics[scale=0.5]{dtree}}
\def\xyleft{\includegraphics[scale=.7]{xyleft}}
\def\xyright{\includegraphics[scale=.7]{xyright}}
\def\yzdtree{\includegraphics[scale=.7]{yzdtree}}
\def\xyztreea{\includegraphics[scale=.7]{xyztree1}}
\def\xyztreeb{\includegraphics[scale=.7]{xyztree2}}
\def\xyztreec{\includegraphics[scale=.7]{xyztree3}}
\def\tableps{\includegraphics[scale=1.4]{tableps}}

\def\ccxtree{\includegraphics[scale=1]{ccxtree}}
\def\ccxtreedec{\includegraphics[scale=1]{ccxtreedec}}
\def\ccetree{\includegraphics[scale=1]{ccetree}}
\def\ccllcrcr{\includegraphics[scale=1]{ccllcrcr}}
\def\cclclcrcr{\includegraphics[scale=.5]{cclclcrcr}}
\def\ccllxryr{\includegraphics[scale=.5]{ccllxryr}}
\def\cclxlyrzr{\includegraphics[scale=1]{cclxlyrzr}}
\def\ccIII{\includegraphics[scale=1]{ccIII}}


\title{Algebraic Birkhoff decomposition and its applications}
\author{Li Guo}
\address{Department of Mathematics and Computer Science,
         Rutgers University,
         Newark, NJ 07102}
\email{liguo@rutgers.edu}



\begin{abstract}
Central in the Hopf algebra approach to the renormalization of perturbative quantum field theory of Connes and Kreimer
is their Algebraic Birkhoff Decomposition. In this tutorial article, we introduce their decomposition and prove it by the Atkinson Factorization in Rota-Baxter algebra. We then give some applications of this decomposition in the study of divergent integrals and multiple zeta values.
\end{abstract}


\maketitle

\tableofcontents

\delete{
\noindent
{\bf Key words: } Algebraic Birkhoff Decomposition, Hopf algebra, Rota-Baxter algebra, renormalization, multiple zeta values.
}
\setcounter{section}{0}
{\ }

\section{Introduction}
\mlabel{intro}
This paper is based on lecture series given at the International Instructional Conference: Langlands and Geometric Langlands Program, June 18-21, 2007 in Guangzhou and at the International School and Conference of Noncommutative Geometry, August 15-30, 2007 in Tianjin, China.
The purpose of this paper, as well as of the lecture series themselves, is to give a self-contained introduction to the Algebraic Birkhoff Decomposition and its related Rota-Baxter algebra, and to give some ideas on its applications, with graduate students and non-experts in mind.

The Algebraic Birkhoff Decomposition
of Connes and Kreimer is a fundamental result in their seminal work~\mcite{C-K1} on Hopf algebra approach to renormalization of perturbative quantum field theory. We briefly describe the related history. Readers not familiar with the physics terminology need not to be concerned as the terminology will not be needed in the rest of the paper. See Sections \mref{sec:int} and \mref{sec:mzv} for a more mathematical (but unprecise) discussion of renormalization.
\smallskip

The perturbative approach to quantum field theory (pQFT) is perhaps
the best experimentally confirmed physical theory in the realm of high energy
physics.
Basic phenomena in particle physics, such as collision, merging and emitting of particles, find an intuitive graphical representation in terms of Feynman
graphs. Amplitudes, for example, of these phenomena are given
by Feynman integrals read off from the sum of Feynman graphs
following a set of Feynman rules.

But these Feynman integrals are usually divergent.
So physicists established a process motivated by
physical insight to extract finite values from these divergent
integrals, after getting rid of a so-called counter-term. This
process is called renormalization of a pQFT~\mcite{Co}. Following
the hierarchical structure of the divergencies, one has to start
with the counter-term related to the inner most divergence, and
expand the process outwards order by order. The algorithm is described by the well-known Bogoliubov formula (Zimmermann's forest formula)~\mcite{Zi}.
Despite its great success in physics, this process was well-known for its lack of a solid mathematical
foundation. The Feynman graphs appeared to be
unrelated to any mathematical structure that might underlie the renormalization prescription.

Such a structure was uncovered in a series of papers by Connes and Kreimer~\mcite{C-K1,CK2,Kr}.
Let $\FG$ be the set of one point irreducible (1PI) Feynman graphs in a renormalizable QFT.
Connes and Kreimer defined a connected filtered Hopf algebra structure on $\calh_\FG=\CC[\FG]$.
Let $\CC[\vep^{-1},\vep]]$ be the ring of Laurent series.
Then a set of Feynman rules, together with a dimensional regularization
applied to the Feynman integrals, amounts to an algebra
homomorphism  $\phi: \calh_\FG \to \CC[\vep^{-1},\vep]].$
Then they showed that $\phi$ factors, in analogue to the Birkhoff decomposition of a loop map, into a product of
algebraic homomorphisms
$\phi_+: \calh_\FG \to \CC[[\vep]]$ and
$\phi_-: \calh_\FG \to \CC[\vep^{-1}].$
The renormalized value of a Feynman integral corresponding to a Feynman graph $\Gamma$ is given by $\phi_+(\Gamma)$ evaluated at $\vep=0$.
\smallskip

This decomposition establishes a bridge that allows the exchange of ideas between physics and mathematics. In one direction, the decomposition provides the renormalization of quantum field theory with a mathematical foundation which was missing before, opening the door of further mathematical understanding of renormalization. Recently, the related Riemann-Hilbert correspondence and motivic Galois groups were studied by Connes and Marcolli~\mcite{CMa}, and motivic properties of Feynman graphs and integrals were studied by Bloch, Esnault and Kreimer~\mcite{BEK}.
In the other direction, the mathematical formulation of renormalization provided by this decomposition allows the method of renormalization dealing with divergent Feynman integrals in physics to be applied to divergent problems in mathematics that could not be dealt with in the past, such as the divergence in multiple zeta values~\mcite{G-Z,G-Z2,Zh2} and Chen symbol integrals~\mcite{M-P1,M-P2}.

This decomposition also links the physics theory of renormalization to an area of mathematics, namely Rota-Baxter algebra, that has evolved in parallel to the development of QFT renormalization for several decades.

The introduction of Rota-Baxter algebra by G. Baxter~\mcite{Ba} in 1960 was motivated by Spitzer's identity~\mcite{Sp} that appeared in 1956 and which was regarded as a remarkable formula in the
fluctuation theory of probability. Soon Atkinson~\mcite{At} proved a simple yet useful factorization theorem in Rota-Baxter algebras.
The identity of Spitzer took its algebraic form through the work of Cartier, Rota and Smith~\mcite{Ca,R-S} (1972).

This was during the same period when the renormalization theory of pQFT was developed, through the the work of Bogoliubov and Parasiuk~\mcite{BP} (1957), Hepp~\mcite{He}(1966) and Zimmermann~\mcite{Zi} (1969), later known as the BPHZ prescription.

Recently QFT renormalization and Rota-Baxter algebra are tied together through the algebraic formulation of Connes and Kreimer for the former and a generalization of classical results in the latter~\mcite{EGK1,E-G-K2,E-G-K3}. More precisely, generalizations of Spitzer's identity and Atkinson factorization give the twisted antipode formula and \ABD in the work of Connes and Kreimer.
\smallskip

We will start with a review in Section~\mref{sec:abd} of Hopf algebras in order to state the \ABD. In Section~\mref{sec:rb}, we discuss Atkinson Factorization and Spitzer's identity in Rota-Baxter algebra, and their generalizations for complete filtered Rota-Baxter algebras. We then derive \ABD and give its explicit form from these generalizations. In Section~\mref{sec:int}, we illustrate how \ABD can be applied to give renormalized values of a system of divergent integrals. In Section~\mref{sec:mzv}, we apply \ABD to the study of divergent multiple zeta values.

\medskip

\noindent
{\bf Acknowledgements.} The author thanks the organizers of the two conferences, especially Lizhen Ji of the Guangzhou conference, and Matilde Marcolli and Guoliang Yu of the Tianjin conference for their invitations. Thanks also go to the hosting institutes, South China University of Technology and Chern Institute of Mathematics at Nankai University, for their hospitality and to the participants for their interest. He acknowledges the support of NSF grant DMS 0505445 and appreciates the detailed comments of the referee.

\section{The Algebraic Birkhoff Decomposition}
\mlabel{sec:abd}
In this section, we provide the necessary background on Hopf algebra to state the \ABD. We will not go into its physics applications for which we refer the reader to the original papers~\mcite{C-K0,C-K1,CK2} of Connes and Kreimer and the survey articles such as~\mcite{EGsu,Ma}.

\subsection{Bialgebra}
\mlabel{sec:bial}
We start with a review of bialgebras and Hopf algebras. Further details can be found in
the lecture notes~\mcite{Ca3,Sch}, as well as the
standard references~\mcite{Ab,Ka,Sw}.

In this paper, a ring or algebra is assumed to be unitary unless otherwise specified. Let $\bfk$ denote the commutative ring on which all algebras and modules are based. All tensor products are also taken over $\bfk$.

\subsubsection{Algebra by diagrams}
\mlabel{sec:algd}
The concept of Hopf algebra originated from the work of Hopf~\mcite{Ho} on manifolds and developed further in topology~\mcite{MM} and representation theory~\mcite{Ab,Hoc} in the 1950-1970s. Other than its theoretical study~\mcite{Ab,Sw}, Hopf algebra found connections with combinatorics~\mcite{JR} and quantum groups~\mcite{Dr} in the 1980s. However, unlike the immediate success in quantum groups, the connection with combinatorics was largely ignored until the 1990s when Hopf algebras of combinatorial nature found applications in number theory~\mcite{Ho2}, noncommutative geometry~\mcite{CMo} and quantum field theory~\mcite{C-K0}.

The structure of a Hopf algebra is built from a compatible pair of the dual structures of an algebra and a coalgebra. In combinatorial terms, in analogy to the product in an algebra that puts two elements together to form a more complicated element, the coproduct in a coalgebra decomposes an elements into pairs of simpler elements.
To motivate the precise definition of a coalgebra, we give
the following interpretation of algebra in terms of
commutative diagrams.

A $\bfk$-algebra $A$ can be equivalently defined as a $\bfk$-module $A$ together with $\bfk$-module homomorphisms
$\mult=\mult_A:A\otimes_\bfk A\to A$ and
$\un=\un_A: \bfk\to A$ such that the following diagrams commute.
\begin{equation}
{\bf (Associativity)}\qquad \xymatrix{ A\ot A \ot A
\ar[rr]^{\mult\ot \id_A}
    \ar[d]_{\id_A\ot \mult} && A\ot A \ar[d]^\mult\\
A\ot A \ar[rr]^\mult && A } \mlabel{eq:ass}
\end{equation}
\begin{equation}
{\bf (Unit)} \qquad \xymatrix{ \bfk \ot A\ar[r]^{\un\ot \id_A}
\ar[rd]_{\alpha_\ell} & A\ot A \ar[d]^\mult & A\ot \bfk
\ar[l]_{\id_A\ot \un}
\ar[ld]^{\alpha_r} \\
& A & } \mlabel{eq:unit}
\end{equation}
Here $\alpha_\ell$ (resp. $\alpha_r$) is the
isomorphism sending $k \ot a$ (resp. $a\ot k$) to $ka$.

Let $(A,\mult_A,\un_A)$ and $(A',\mult_{A'},\un_{A'})$ be two $\bfk$-algebras. An algebra homomorphism from $A$ to $A'$ can be equivalently defined to be a $\bfk$-module homomorphism
$f: A\to A'$ such that the following diagrams commute.
\begin{equation}
\xymatrix{ A\ot A \ar^{\mult_A}[rr]
\ar^{f\ot f}[dd] && A \ar_{f}[dd]\\
&& \\
A'\ot A' \ar^{\mult_{A'}}[rr] && A'
}
\qquad
\xymatrix{ & A \ar_{f}[dd] \\
\bfk \ar^{\un_A}[ur] \ar_{\un_{A'}}[dr] & \\
& A'
}
\mlabel{eq:hom}
\end{equation}

\subsubsection{Coalgebra}
A {\bf $\bfk$-coalgebra} is obtained by reversing the arrows in the
diagrams (\mref{eq:ass}) and (\mref{eq:unit}) of a
$\bfk$-algebra. More precisely, a $\bfk$-coalgebra is a triple
$(C,\Delta,\varepsilon)$ where $C$ is a $\bfk$-module, $\Delta:
C\to C\otimes C$ and $\varepsilon: C\to \bfk$ are $\bfk$-linear
maps that make the following diagrams commute.

\begin{equation}
{\bf (Coassociativity)} \qquad \xymatrix{
    C \ar[rr]^{\Delta} \ar[d]_{\Delta} && C\otimes C
    \ar[d]^{\id_C\otimes \Delta}\\
    C\otimes C \ar[rr]^{\Delta\otimes \id_C} && C\otimes C\otimes
    C
} \mlabel{coass}
\end{equation}
\begin{equation}
{\bf (Counit)} \qquad \xymatrix{\bfk \otimes C & C\ot C
\ar[l]_{\vep\otimes \id_C}
    \ar[r]^{\id_C\otimes \vep}
& C\otimes \bfk \\
& C \ar[lu]^{\beta_\ell} \ar[u]_{\Delta} \ar[ru]_{\beta_r} & }
\mlabel{coun}
\end{equation}
Here $\beta_\ell$ (resp. $\beta_r$) is the isomorphism sending
$a\in C$ to $1\ot a$ (resp. $a\ot 1$).

The $\bfk$-algebra $\bfk$ has a natural structure of a
$\bfk$-coalgebra with
\[ \Delta_\bfk: \bfk\to \bfk\otimes \bfk,\ c\mapsto c\otimes 1,\ c\in \bfk\]
and
\[ \vep_\bfk=\id_\bfk: \bfk\to \bfk. \]
We also denote the multiplication in $\bfk$ by $\mult_\bfk$.

For $a\in C$, we have $\Delta(a)=\sum_i a_{1\, i}\ot a_{2\,i}$
which we denote simply by $\sum_{(a)} a_{(1)}\ot a_{(2)}$ (Sweedler's notation).

\subsubsection{Bialgebra}
By superimposing a compatible pair of algebra and coalgebra, we obtain a bialgebra.
More precisely, a $\bfk$-{\bf bialgebra} is a quintuple
$(H,\mult,\un,\Delta,\vep)$ where $(H,\mult,\un)$ is a
$\bfk$-algebra and $(H,\Delta,\vep)$ is a $\bfk$-coalgebra such
that $\Delta:H\to H\ot H$ and $\vep: H\to \bfk$ are morphisms of $\bfk$-algebras, where $H\ot H$ is the
tensor product algebra with the product
$$\mult_{H\ot H} (a_1\ot b_1,a_2\ot b_2)=(a_1b_1)\ot (a_2b_2).$$
In terms of the diagram interpretation of algebra homomorphisms in Eq.~(\mref{eq:hom}), this amounts to the commutativity of the following diagrams.

\begin{equation}
\xymatrix{ H\ot H \ar^{\mult_H}[rr]
\ar^{\Delta\ot \Delta}[d] && H \ar^{\Delta}[d]\\
(H\ot H)\ot (H\ot H) \ar^(0.6){\mult_{H\ot H}}[rr] && H\ot H
}
\mlabel{eq:md0}
\end{equation}
which is more often rewritten as

\begin{equation}
\xymatrix{
 H\otimes H \ar[rr]^{\mult_H}
 \ar[d]_{\scriptstyle{(\id\otimes\tau\otimes \id)(\Delta\otimes
\Delta)}}  & & H \ar[d]^{\Delta}\\
H\otimes H\otimes H\otimes H \ar[rr]^(0.6){\mult_H\otimes \mult_H} &&
    H\otimes H
} \mlabel{md}
\end{equation}
where $\tau:=\tau_{H,H}: H\otimes H\to H\otimes H$ is defined by
$\tau_{H,H}(x\otimes y)=y\otimes x$;

\begin{equation}
\xymatrix{ H\otimes H  \ar[rr]^{\mult_H} \ar[d]_{\vep\ot \vep} & &
H \ar[d]^{\vep}\\
\bfk \ot \bfk \ar[rr]^{\mult_\bfk} && \bfk } \mlabel{me}
\end{equation}

\begin{equation}
\xymatrix{ & H \ar^{\Delta}[dd] \\
\bfk \ar^{\un_H}[ur] \ar_{\un_{H\ot H}}[dr] & \\
& H\ot H
}
\end{equation}
which is often rewritten as
\begin{equation}
\xymatrix{
\bfk \ar[rr]^{\un_H} \ar[d]_{\Delta_\bfk} && H \ar[d]^{\Delta}\\
\bfk\otimes \bfk \ar[rr]^{\un_H \otimes \un_H} && H\otimes H }
\mlabel{ed}
\end{equation}

\begin{equation}
\xymatrix{
& H \ar^{\vep}[dd] \\
\bfk \ar[ru]^{\un_H} \ar[rd]_{u_\bfk=\id_\bfk} &  \\
& \bfk  } \mlabel{ee}
\end{equation}

\subsubsection{Examples}
\mlabel{sss:exam}
Here are some simple examples of bialgebras. Further examples will be given in Section~\mref{sec:int} and \mref{sec:mzv}.

The {\bf divided power bialgebra} is defined by the quintuple
$(H,\mult,\un,\Delta,\vep)$ where
\begin{enumerate}
\item $H$ is the free $\bfk$-module $\bigoplus_{n=0}^\infty \bfk
a_n$ with basis $a_n$, $n\geq 0$; \item $\mult: H\otimes H \to H,
a_m\otimes a_{n} \mapsto \binc{m+n}{m}a_{m+n}$; \item $\un: \bfk
\to H,\ \bfone \mapsto a_{0}$; \item $\Delta: H\to H\otimes H,
a_{n}\mapsto \sum_{k=0}^{n}  a_{k}\otimes  a_{n-k}$; \item $\vep:
H\to \bfk,\ a_{n} \mapsto \delta_{0,n} \bfone$ where
$\delta_{0,n}$ is the Kronecker delta.
\end{enumerate}
The reader is invited to verify that this is a bialgebra.
\medskip

More general than the divided power bialgebra is the {\bf shuffle
product bialgebra} $Sh(V)$. (See~\mcite{A-G-K-O} for another generalization.)
Let $V$ be a free $\bfk$-module. Define $V^{\ot n}$ to be the
$n$-th tensor power of $V$:
$$ V^{\ot n} = V\underbrace{\ot \cdots \ot}_{n-{\rm factors}} V$$
with the convention that $V^{\ot 0}=\bfk$. Let
$$Sh(V)=\bigoplus_{n\geq 0} V^{\ot n}.$$

A shuffle of $a_1\otimes \ldots \otimes a_m$ and $b_1\otimes
\ldots \otimes b_n$ is a tensor list of $a_i$ and $b_j$ without
change the order of the $a_i$s and $b_j$s. The shuffle product
$(a_1\otimes \ldots \otimes a_m)\shpr (b_1\otimes \ldots \otimes
b_n)$ is the sum of shuffles of $a_1\otimes \ldots \otimes a_m$
and $b_1\otimes \ldots \otimes b_n$.
For example,
$$a_1 \shpr (b_1\otimes b_2) = a_1\otimes b_1\otimes b_2 +
b_1\otimes a_1\otimes b_2 + b_1\otimes b_2\otimes a_1.$$
Then
$\shpr$ is an associative product on $Sh(V)$, making $Sh(V)$ into
a $\bfk$-algebra with the unit $\un:\bfk\to A$ given by
$\un(k)=k\in \bfk \subseteq Sh(V)$.

Define a coproduct $\Delta: Sh(V)\to Sh(V)\ot Sh(V)$ by
\begin{eqnarray*}
\Delta(a_1\ot \cdots \ot a_n)&=&1\ot (a_1\ot \cdots \ot a_n)+
a_1 \ot (a_2\ot \cdots \ot a_n) + \cdots \\
&& + (a_1\ot \cdots a_{n-1})\ot a_n + (a_1\ot \cdots \ot a_n)\ot
1,\ n\geq 1
\end{eqnarray*}
and $\Delta(k)=k \bfone \ot \bfone$. Define a counit $\vep:
Sh(V)\to \bfk$ by
$$ \vep(a_1\ot \cdots \ot a_n)=0, n\geq 1,\ {\rm\ and\ }\
\vep(\bfone)=\bfone.$$
We then obtain a bialgebra. It is easy to see that the divided power bialgebra is the special case when $V$ a free $\bfk$-module of rank one.

\subsection{Connected bialgebra and Hopf algebra}
\mlabel{sec:conn}
The main goal of this section is show that a connected bialgebra
is automatically a Hopf algebra.
\subsubsection{Connected bialgebra}
\begin{defn}
{\rm
A bialgebra $H$ is called a {\bf graded bialgebra} if there are
$\bfk$-submodules $H_n,\ n\geq 0,$ of $H$ such that
\begin{enumerate}
\item $H_p H_q\subseteq H_{p+q}$; \item $\Delta(H_n) \subseteq
\bigoplus_{p+q=n} H_p\ot H_q.$
\end{enumerate}
Elements of $H_n$ are given degree $n$. $H$ is called {\bf
connected} if $H_0 =\im\, u (=\bfk)$.
}
\end{defn}
The term connected comes from the fact that the cohomology
bialgebra of a connected compact Lie group is a connected
bialgebra.
\begin{defn}
{\rm
A bialgebra $H$ is called a {\bf filtered bialgebra} if there are
$\bfk$-submodules $H^{(n)},\ n\geq 0$ of $H$ such that
\begin{enumerate}
\item $H^{(n)}\subseteq H^{(n+1)}$; \item $\cup_{n\geq 0} H^{(n)}
= H$; \item $H^{(p)} H^{(q)}\subseteq H^{(p+q)}$; \item
$\Delta(H^{(n)}) \subseteq \bigoplus_{p+q=n} H^{(p)}\ot H^{(q)}.$
\end{enumerate}
$H$ is called {\bf connected} if $H^{(0)}=\im\, \un (=\bfk)$.
}
\end{defn}
Obviously, a graded bialgebra is a filtered bialgebra with
filtration defined by $H^{(n)}=\sum_{k\leq n} H_k$.
We introduce a concept between a graded bialgebra and a filtered bialgebra that is suitable for our later applications.

\begin{defn}
{\rm
A bialgebra $H$ is called a {\bf connected, filtered, cograded bialgebra} if there are
$\bfk$-submodules $H_n,\ n\geq 0$ of $H$ such that
\begin{enumerate}
\item $H_p H_q\subseteq \sum_{k\leq p+q}H_k$;
\item $\Delta(H_n) \subseteq
\bigoplus_{p+q=n} H_p\ot H_q.$
\item $H_0 =\im\, u (=\bfk)$.
\end{enumerate}
}
\end{defn}

In the following we will only consider connected filtered cograded bialgebras. For connected filtered bialgebras, see~\mcite{F-G1}.

We note that $e:=\un \vep (:=\un \circ \vep)$ has the property that
 $e(\bfone_H)=\bfone_H$ and $e(x)=0$ for any
 $x\in \ker\vep_H$.
So the map $e:=\un \vep:H\to H$ is an idempotent. Therefore
$$H=\im (\un\vep)\oplus\ker (u\eta)
=\im \un \oplus \ker\vep =\bfk \bfone_H\oplus \ker\vep.$$ In fact,
$\ker\, \vep = \oplus_{n>0}H_n$.

\begin{theorem}
Let $H$ be a connected filtered cograded bialgebra. For any $x\in H_n$, the
element $\tilde{\Delta}(x):=\Delta(x)-x\ot 1-1\ot x$ is in
$\oplus_{p+q=n,p>0,q>0} H_p\ot H_q$.
 The map $\tilde{\Delta}$ is
coassociative on $\ker \vep$. \mlabel{eq:conn} \mlabel{thm:conn}
\end{theorem}
We will use the short hand notations:
$\tilde{\Delta}(x)=\sum_{(x)} x' \ot x'',
\tilde{\Delta}^2(x)=\sum_{(x)} x' \ot x'' \ot x'''.$
The later makes sense thanks to Theorem~\mref{thm:conn}.

\begin{proof}
By assumption, we have
$$\Delta(x)=y\ot \bfone_H +\bfone_H \ot z + w$$
with $w\in \oplus_{p+q=n,p>0,q>0} H^{(p)}\ot H^{(q)}$. By the
counicity of $\Delta$ in Eq.~(\mref{coun}), we have
$$x=\beta_\ell\,(\vep\ot \id_H)\,\Delta(x)=z$$
and
$$x=\beta_r\,(\id_H\ot \vep)\,\Delta(x)=y.$$

To verify the coassociativity of $\tilde{\Delta}$, we note that,
for $x\in H_n$, \allowdisplaybreaks{
\begin{align*}
(\Delta\ot \id_H)\Delta(x) &= (\Delta\ot \id_H) (x\ot
\bfone_H+\bfone_H\ot x +
\sum_{(x)}x'\ot x'')\\
&=\Delta(x)\ot \bfone_H +\Delta(\bfone_H)\ot x
+\sum_{(x)}\Delta(x') \ot x''\\
&= (x\ot \bfone_H +\bfone_H\ot x +\tilde{\Delta}(x))\ot \bfone_H
+\bfone_H\ot \bfone_H\ot x \\
&+ \sum_{(x)} \bigg ( \sum_{(x')}\bigg(x'\ot \bfone_H
+\bfone_H\ot x'+\tilde{\Delta}(x')\bigg)\bigg) \ot x''\\
&=(x\ot \bfone_H +\bfone_H\ot x +\tilde{\Delta}(x))\ot \bfone_H
+\bfone_H\ot \bfone_H\ot x \\
&+ \sum_{(x)} \sum_{(x')}\bigg(x'\ot \bfone_H \ot x'' +\bfone_H\ot
x' \ot x''\bigg) +(\tilde{\Delta}\ot
\id_H)\tilde{\Delta}(x)\\
&=x\otimes 1\otimes 1 +1\otimes x\otimes 1
                        +1\otimes 1\otimes x    \\
                &\ +\sum_{(x)} \big (
x'\otimes x''\otimes 1 + x'\otimes 1\otimes x'' + 1\otimes
x'\otimes
x''\big ) \\
           &\ +(\tilde{\Delta}\otimes \id_H)\tilde{\Delta} (x).
\end{align*}}
Similarly,
\begin{align*}
(\id_H\ot \Delta)\Delta(x) &= x\otimes 1\otimes 1 +1\otimes
x\otimes 1
                        +1\otimes 1\otimes x    \\
                &\ +\sum_{(x)} \big (
x'\otimes x''\otimes 1 + x'\otimes 1\otimes x'' + 1\otimes
x'\otimes
x''\big ) \\
           &\ +(\id_H\ot \tilde{\Delta})\tilde{\Delta} (x).
\end{align*}
So the coassociativity of $\tilde{\Delta}$ follows from that of
$\Delta$.
\end{proof}

\subsubsection{Convolution product}
\mlabel{sss:conv}
For a $\bfk$-algebra $A$ and a $\bfk$-coalgebra $C$, we define the
\textbf{convolution} of two linear maps $f,g$ in $\Hom(C,A)$ to be
the map $f*g\in\Hom(C,A)$ given by the composition
$$
C \xrightarrow{\Delta} C \ot C \xrightarrow{f\ot g} A \ot A
\xrightarrow{\mult} A.
$$
In other words,
$$
(f*g)(a) = \sum_{(a)} f(a_{(1)}) \, g(a_{(2)}).
$$

We next define a metric on $\Hom(C,A)$ when $C$ is connected filtered.
\begin{defn}
A {\bf filtered $\bfk$-algebra} is a $\bfk$-algebra $R$ together
with a decreasing filtration $R_n, \: n\geq 0, $ of nonunitary
subalgebras such that
$$
\bigcup_{n\geq 0} R_n = R, \quad R_n R_m \subseteq R_{n+m}.
$$
\end{defn}
It immediately follows that $R_0=R$ and each $R_n$ is an ideal of
$R$.

In a filtered $\bfk$-algebra $R$, we can use the subsets $\{R_n\}$
to define a metric on $R$ in the standard way. More precisely, for $x\in
R$, define
\begin{equation}
 o(x)=\left \{ \begin{array}{ll} \max \{k\ \big |\ x\in R_k\},& x\notin \cap_n R_n,\\
    \infty,& {\rm otherwise} \end{array} \right.
\mlabel{eq:metric}
\end{equation}
    and, for $x,y\in R$,
$$ \|y-x\|=2^{-o(y-x)}.$$
A filtered algebra is called {\bf complete} if $R$ is a complete
metric space with $\|y-x\|$, that is, every Cauchy sequence in $R$
converges. Equivalently, a filtered $\bfk$-algebra $R$ with
$\{R_n\}$ is complete if $\cap_n R_n = 0$ and if the resulting
embedding
$$
R \to \bar{R}:= \varprojlim R/R_n
$$
of $R$ into the inverse limit is an isomorphism.

When $R$ is a complete filtered $\QQ$-algebra, the functions
\begin{align}
\exp:& R_1 \to 1+R_1,\ \exp(a):=\sum_{n=0}^\infty \frac{a^n}{n!},
\mlabel{eq:exp} \\
\log:& 1+R_1 \to R_1,\ \log(1+a):=-\sum_{n=1}^\infty
\frac{(-a)^n}{n} \mlabel{eq:log}
\end{align}
are well-defined.

We also record the following simple fact for later reference.
\begin{lemma}
Let $R$ be a complete filtered algebra. The subset $1+R_1$ is a
group under the multiplication. \mlabel{lem:inv}
\end{lemma}
\begin{proof}
An element of $1+R_1$ is of the form $1-a$ with $a\in R_1$. Thus
$\sum_{n=0}^\infty a^n$ is a well-defined element in $1+R_1$ and
$$ (1-a) \big(\sum_{n=0}^\infty a^n\big)= \big(\sum_{n=0}^\infty a^n \big)(1-a)=1.$$
\end{proof}

\begin{theorem}
Let $C$ be a coalgebra and $A$ an algebra. Let $e=\un_A\vep_C$.
\begin{enumerate}
\item The triple $(\Hom(C,A),*,\un_A\vep_C)$ is an algebra.
\mlabel{it:dual} \item Let $H=\cup_{n\geq 0} H^{(n)}$ be a connected
filtered bialgebra. Let $R=\Hom(H,A)$. Define
$$ R_n=\{f\in \Hom(H,A) \Big |\ f(H^{(n-1)})=0 \}$$
for $n\geq 0$ with the convention that $H^{(-1)}=\emptyset$. Then
$(R,R_n)$ is a complete algebra. \mlabel{it:comdual} \item
Under the same hypotheses as \mref{it:comdual}, the set
$G=\{f\in\Hom(H,A)\Big|\ f(\bfone_H)=\bfone_A\}$ endowed with the
convolution product is a group. \mlabel{it:gp}
\end{enumerate}
\mlabel{thm:inv}
\end{theorem}
\begin{proof}
\mref{it:dual} By the associativity of $m$ and
coassociativity of $\Delta$,
\begin{align*}
(f * g) * h &= m ((f * g) \ot h) \Delta = m(m \ot \id_A)(f \ot g
\ot h)(\Delta \ot \id_C)\Delta
\\
&= m(\id_A \ot m)(f \ot g \ot h)(\id_C \ot \Delta)\Delta = m (f
\ot (g * h)) \Delta
         = f * (g * h).
\end{align*}

Also,
\begin{align*}
f * u_A \vep_C &= m (f \ot u_A \vep_C) \Delta
         = m (\id_A \ot u_A)(f \ot \id_\bfk)(\id_C \ot \vep_C) \Delta
         \\
         & =\alpha_\ell (f\ot \id_\bfk) \beta_\ell  = f,
\\
u_A\vep_C * f &= m (u_A\vep_C \ot f) \Delta
         = m (u_A \ot \id_A)(\id_\bfk \ot f)(\vep_C \ot \id_C) \Delta
         \\
         & = \alpha_r (\id_\bfk \ot f) \beta_r  = f.
\end{align*}

\mref{it:comdual} Let $f\in R_p$ and $g\in R_q$. Then for
$x\in H_k,\ k\leq {p+q-1}$, we have $ (f\cprod g)(x)=\sum_{(x)}
f(x_{(1)})g(x_{(2)})$ with $\deg x_{(1)}+\deg x_{(2)}=k.$ So in
each term of the sum, either $\deg x_{(1)}\leq p-1$ or $\deg
x_{(2)}\leq q-1$. So either $f(x_{(1)})=0$ or $g(x_{(2)})=0$. Thus
$(f\cprod g)(x)=0$. Therefore $R_p\cprod R_q\subseteq
R_{p+q}$.

To prove that the filtration is complete, we first note that
$$\cap_{n\geq 0} R_n=\{f\in R \Big | f(H_n)=0,\ \forall n\}
=0$$ since $H=\cup_{n\geq 0} H_n$. Next let $f_n$ be a Cauchy
sequence in $R$. Define a $f\in R$ as follows. Let $x\in
H_k,\ k\geq 0$. Then there is $N$ such that
$$ |f_n-f_N|\leq \frac{1}{2^k},\ \forall n\geq N.$$
This means that $f_n(x)=f_N(x),\ \forall x\in H_k$. Define
$f(x)=f_N(x)$ for such an $N$. It is easy to check that
$\lim_{n\to \infty} f_n=f.$

\mref{it:gp} $f\in \Hom(H,A)$ with $f(\bfone_H)=\bfone_A$ if and
only if $f=e+g$ with $g\in \Hom(H,A)$ such that $g(\bfone_H)=0$.
Therefore $g$ is in $R_1$. Then by Lemma~\mref{lem:inv}, $f$ is
invertible.
\end{proof}

\subsubsection{Antipode and Hopf algebra}
Let $(H,\mult,\un,\Delta,\vep)$ be a $\bfk$-bialgebra. A
$\bfk$-linear endomorphism $S$ of $H$ is called an {\bf antipode}
for $H$ if it is the inverse of $\id_H$ under the convolution
product:
\begin{equation}
 S\ast \id_H = \id_H\ast S =\un\, \vep.
\mlabel{anti}
\end{equation}
A {\bf  Hopf algebra} is a bialgebra $H$ with an antipode $S$. In
view of Theorem~\mref{thm:inv}, we have,

\begin{theorem}
Any connected filtered cograded bialgebra $H$ is a Hopf algebra. The antipode
is defined by~:
$$S(x)=\sum_{k\ge 0}(u\varepsilon-I)^{*k}(x).$$
It is also defined by $S(\bfone_H)=\bfone_H$ and recursively by
any of the two formulas for $x\in\ker\vep$.
\begin{eqnarray*}S(x) &=& -x-\sum_{(x)}S(x')x'', \\
        S(x)    &=&  -x-\sum_{(x)}x'S(x'').
\end{eqnarray*}
\mlabel{thm:pt}
\end{theorem}
\begin{proof}
The existence of the antipode and its first formula follows from
Theorem~\mref{thm:inv}. The two recursive formulas follow from
$e(\ker\vep)=0$ and the equalities
$$e=S\cprod \id_H=m(S\otimes I)\Delta, \quad
 e=\id_H\cprod S= m(I\otimes S)\Delta. $$
\end{proof}

The first formula of $S$ in Theorem~\mref{thm:pt} was first
obtained in~\mcite{F-G1} following a suggestion of E. Taft. There the
proof took the following form.
\begin{theorem}
Let $H$ be a connected, graded bialgebra and let $A$ be a
$\bfk$-algebra.
\begin{enumerate}
\item Let $f\in \Hom(H,A)$ be such that $f(\bfone_H)=0$. Then for
$x\in H^n$, $n\geq 0$, we have $f^{\cprod (n+1)}(x)=0$.
\mlabel{it:nil} \item The set $G=\{\varphi\in\Hom(H,A)\Big |\
\varphi(\bfone_H)=\bfone_A\}$ endowed with the convolution product
is a group. \mlabel{it:inv}
\end{enumerate}
\mlabel{thm:conv}
\end{theorem}
\begin{proof}
\mref{it:nil} We prove by induction on $n\geq 0$, with the case
when $n=0$ following from $f(\bfone_H)=0$ and the connectedness of
$H$.

Assume the statement holds for $0\leq n\leq k$. Then for $x\in
H_{k+1}$, we have
\begin{align*}
f^{\cprod (k+2)}(x)&= (f\cprod f^{\cprod (n+1)})(x)\\
&= m(f\ot f^{\cprod (n+1)})\Delta (x)\\
&= m (f\ot f^{\cprod (n+1)}) (\bfone_H\ot x+x\ot \bfone_H +\sum_{(x)}x'\ot x'')\\
&= f(\bfone_H)f^{\cprod(n+1)}(x)+f(x)f^{\cprod (n+1)}(\bfone_H)+
\sum_{(x)}f(x')f^{\cprod(n+1)}(x'').
\end{align*}
The first term is zero by our choice of $f$. The others terms are
zero by induction since $\bfone_H$ and $x''$ are in $H_n$ with
$n\leq k$.

\mref{it:inv} The set $G$ is obviously closed under convolution
multiplication with $e$ as identity. For $g\in G$, let $f=e-g$.
Then $f(\bfone_H)=0$. So by \mref{it:nil}, for $x\in H_n$, we have
$f^{\cprod (n+1)}(x)=0$. Thus $\sum_{k\geq 0} f^{\cprod
k}(x)=\sum_{k=0}^n f^{\cprod k}(x)$ is well-defined and
$$(g\cprod (\sum_{k=0}^n f^{\cprod k}))(x)=((e-f)\cprod (\sum_{k=0}^n
f^{\cprod k}))(x)=(e-f^{\cprod (n+1)})(x)=e(x).$$ Thus $g\cprod
(\sum_{k\geq 0} f^{\cprod k})=e$ and  $g$ is invertible.
\end{proof}

\subsection{Characters and derivations}
\mlabel{sec:char}

\begin{defn}
Let $H$ be a Hopf algebra and $A$ an algebra.
{\rm An element $f\in \Hom(H,A)$ is called a {\bf character} if $f$ is
an algebra homomorphism, and is called a {\bf derivation} (or {\bf
infinitesimal character}) if
\begin{equation}
 f(xy)=e (x) f(y)+e (y) f(x),\ \forall\, x,\, y\in H.
 \mlabel{deriv}
 \end{equation}
The set of characters (resp.
derivations) is denoted by $\mchar(H,A)$ (resp. $\pmchar(H,A)$).
}
\end{defn}
We note that $f(\bfone_H)=\bfone_A$ if $f$ is a character and
$f(\bfone_H)=0$ if $f$ is a derivation.
We recall the notations $R=\Hom(H,A)$ and $R_1=\{f\in
\Hom(H,A)\,\Big |\, f(\bfone_H)=0\}$.
\begin{prop} Let $A$ be a commutative algebra.
\begin{enumerate}
\item $\mchar(H,A)$ is a group under convolution.
 The inverse of
$\phi\in \mchar(H,A)$ is given by
 $\phi^{-1}=\phi \,S$.
\mlabel{it:char}
\item $\pmchar(H,A)$ is a Lie algebra under
convolution. \mlabel{it:der} \item The bijection $\exp: R_1\to
e+R_1$ in Eq. (\mref{eq:exp}) restricts to a bijection $\exp:
\pmchar (H,A) \to \mchar(H,A).$ \mlabel{it:bij}
\end{enumerate}
\mlabel{pp:dual} \mlabel{groupchar}
\end{prop}

\begin{proof}
\mref{it:char} Clearly $e$ is in $\mchar(H,A)$.
We note that $f\in \Hom(H,A)$ is multiplicative means that
$$ m_A(f\ot f)=f\,m_H.$$
Thus for $f,g\in \mchar(H,A)$, using the fact that $\Delta$ is an
algebra homomorphism,
 we have
\allowdisplaybreaks{
\begin{align*}
(f\cprod g) m_H &= m_A (f\ot g) \Delta_H m_H\\
    &= m_A (f\ot g)(m_H\ot m_H)\tau_{2,3} (\Delta_H\ot \Delta_H)\\
    &= m_A \big( (f m_H )\ot(g m_H)\big) \tau_{2,3} (\Delta_H\ot
    \Delta_H) \\
    &= m_A \big( m_A(f\ot f) \ot m_A(g\ot g)\big) \tau_{2,3}
    (\Delta_H\ot \Delta_H)\\
    &= m_A (m_A\ot m_A) (f\ot f\ot g\ot g)\tau_{2,3}
    (\Delta_H\ot \Delta_H) \\
    &= m_A (m_A\ot m_A) (f\ot g\ot f\ot g)(\Delta_H\ot \Delta_H) \\
    &= m_A \big( (m_A(f\ot g)\Delta_H)\ot (m_A(f\ot
    g)\Delta_H)\big) \\
    &= m_A \big( (f\cprod g)\ot (f\cprod g)\big)
    \end{align*}}
proving that $f\cprod g$ is in $\mchar(H,A)$.

Further, for $f\in \mchar(H,A)$,
\begin{align*}
(f S)\cprod f &= m_H((f S)\ot f) \Delta \\
&= m_A (f\ot f)(S\ot \it_H)\Delta \\
&= f m_H (S\ot \id_H) \Delta \\
&= f(S\cprod \id_H)=f e = e.
\end{align*}
Likewise, $f\cprod (fS)=e$. So $f^{\cprod(-1)}=f\,S$.
\medskip

\mref{it:der} As in \mref{it:char}, we first note that $f$ is a
derivation mean that
$$ f\, m_H=m_A (e\ot f + f\ot e).$$
Then for $f,g\in \pmchar(H,A)$, we have \allowdisplaybreaks{
\begin{align*}
(f\cprod g)\,m_H &= m_A (f\ot g) \Delta_H m_H\\
    &= m_A (f\ot g)\tau_{2,3} (\Delta_H\ot \Delta_H)\\
    &= m_A \Big (\big(m_A(e\ot f+f\ot e)\big)\ot \big(m_A(e\ot
    g+g\ot e)\big) \Big) \\
    &= m_A (m_A\ot m_A)(e\ot f\ot e\ot g+e\ot f\ot g\ot e \\
    & \quad +f\ot e\ot e\ot g
     +f\ot e\ot g\ot e) \tau_{2,3} (\Delta_H\ot \Delta_H)\\
    &= m_A (m_A\ot m_A)(e\ot e\ot f\ot g+e\ot g\ot f\ot e \\
    & \quad +f\ot e\ot e\ot g +f\ot g\ot e\ot e) (\Delta_H\ot \Delta_H)\\
    &=m_A\Big((e\cprod e)\ot (f\cprod g)+(e\cprod g)\ot (f\cprod e)
    +(f\cprod e)\ot (e\cprod g)+(f\cprod g)\ot (e\cprod e)\Big)\\
    &=m_A\Big(e\ot (f\cprod g)+ g\ot f +f\ot g+(f\cprod g)\ot e\Big).\\
\end{align*}}
Similarly,
$$(g\cprod f)\,m_H
=m_A\Big(e\ot (g\cprod f)+ f\ot g +g\ot f+(g\cprod f)\ot e\Big).$$
Therefore,
$$\big((f\cprod g)-(g\cprod f)\big)\,m_H
=m_A\big(e\ot (f\cprod g-g\cprod g)+(f\cprod g-g\cprod f)\ot
e\big)$$ as needed.
\medskip

\mref{it:bij} Let $f\in \pmchar(H,A)$. Using
$f(xy)=e(x)f(y)+e(y)f(x)$ and induction, we have
$$ f^{\cprod n}(xy)=\sum_{k=0}^n \binc{n}{k} f^{\cprod
k}(x)f^{\cprod (n-k)}(y),\ n\geq 0.$$ Thus
\begin{align*}
e^{\cprod f}(xy)&= \sum_{n\geq 0} \frac{1}{n!} f^{\cprod n}(xy)\\
    &= \sum_{n\geq 0} \frac{1}{n} \sum_{k=0}^n \binc{n}{k}
    f^{\cprod k}(x)f^{\cprod (n-k)}(y) \\
    &= \sum_{n\geq 0} \sum_{k=0}^n \frac{1}{k!}f^{\cprod k}(x)
    \frac{1}{(n-k)!} f^{\cprod (n-k)}(y) \\
    &= e^{\cprod f}(x)\, e^{\cprod f}(y).
\end{align*}
So $e^{\cprod f}$ is in $\mchar(H,A)$.
By a similar argument, we verify that if $\phi\in \mchar(H,A)$ then
$$\log_\cprod (\phi):=-\sum_{n\geq 1}
\frac{(1-\phi)^{\cprod n}}{n}$$
is in $\pmchar(H,A)$.
\end{proof}

\subsection{Algebraic Birkhoff Decomposition}
\mlabel{sec:algb}
Now we can state the \ABD.
See Theorem~\mref{thm:atke} for its interpretation in terms of Rota-Baxter algebra and Section \mref{ss:idea} for its application in renormalization.

\begin{theorem}
Let $H$ be a connected filtered cograded Hopf algebra over $\CC$. Let $K$ be a $\CC$-algebra and let $A=K[t^{-1},t]]$ be the algebra of Laurent series. Define $Q:A\to A$ by
$$Q\big(\sum_{n} a_n t^n \big)=\sum_{n<0} a_n t^n.$$
\begin{enumerate}
\item For
$\phi\in \mchar(H,A)$, there are unique linear maps $\phi_-:H\to K[t^{-1}]$ and $\phi_+:H\to K[[t]]$ such
that
\begin{equation}
 \phi=\phi_-^{\cprod (-1)}\cprod \phi_+.
 \mlabel{eq:decom}
\end{equation}
\mlabel{it:decom}
\item The elements $\phi_-$ and $\phi_+$
take the following forms on $\ker \vep$.
\begin{align}
\phi_-(x)&=
-Q(\phi(x)+\sum_{(x)}\phi_-(x')\phi(x'')),
\mlabel{eq:phi-}\\
\phi_+(x)&=
\tilde{Q}(\phi(x)+\sum_{(x)}\phi_-(x')\phi(x''))
\mlabel{eq:phi+}\\
&= \tilde{Q}(\phi(x)+\sum_{(x)}\phi(x')\phi_+(x''))
\end{align}
\mlabel{it:rec} \item The linear maps $\phi_-$ and $\phi_+$
are also algebra homomorphisms. \mlabel{it:decgp}
\end{enumerate}
\mlabel{thm:algBi}
\end{theorem}
Eq.~(\mref{eq:phi-}) is called twisted antipode formula of $\phi_-$~\mcite{Kr}.
\begin{proof}
\mref{it:decom} and \mref{it:rec} will be proved in \S~\mref{ss:atabc} using Theorem~\mref{thm:atke}.
For their original proofs, see~\mcite{C-K1}.

\mref{it:decgp} Once we prove that $\phi_-$ is a character, it
is immediate that $\phi_+$ is also character by
Proposition~\mref{pp:dual} and $\phi_+=\phi_-\cprod
\phi$.

For $x,y\in H$, we have \allowdisplaybreaks{
\begin{align*}
\Delta_H(xy)&= \Delta_H(m_H(x\ot y))\\
&= (m_H\ot m_H)\tau_{2,3}(\Delta_H\ot \Delta_H)(x\ot y)\\
&= (m_H\ot m_H)\tau_{2,3}\big( (x\ot 1+1\ot x+\sum_{(x)}x'\ot
x'')\ot (y\ot 1 +1\ot y +\sum_{(y)}y'\ot y'')\big)\\
&= (m_H\ot m_H)(x\ot y\ot 1\ot 1+x\ot 1\ot 1\ot y +\sum_{(y)}x\ot
y'\ot 1\ot y'' \\
& +1\ot y\ot x\ot 1+ 1\ot 1\ot x\ot y +\sum_{(y)}1\ot y'\ot x\ot
y''\\
&+ \sum_{(x)}x'\ot y\ot x''\ot 1 + \sum_{(x)} x'\ot 1\ot x''\ot y
+\sum_{(x),(y)} x'\ot y'\ot x''\ot y'')\\
&=xy\ot 1+x\ot y +\sum_{(y)}xy'\ot y'' + y\ot x+ 1\ot xy
+\sum_{(y)} y'\ot xy'' \\
& + \sum_{(x)}x' y\ot x'' + \sum_{(x)} x'\ot x'' y +\sum_{(x),(y)}
x' y'\ot x'' y''.
\end{align*}}
So \allowdisplaybreaks{
\begin{align*}
\phi_-(xy)&= -Q\big(
\phi(xy)+\phi_-(x)\phi(y)+\phi_-(y)\phi(x) \\
&+
\sum_{(y)}(\phi_-(y')\phi(xy'')+\phi_-(xy')\phi(y''))
+\sum_{(x)}(\phi_-(x'y)\phi(x'')+\phi_-(x')\phi(x''y))\\
&+ \sum_{(x),(y)} \phi_-(x'y')\phi(x''y'')
\end{align*}}

On the other hand, denoting $\phi_-(x)=-Q(X)$,
$\phi_-(y)=-Q(Y)$ and using the Rota--Baxter relation of $Q$,
we have have \allowdisplaybreaks{
\begin{align*}
\phi_-(x)\phi_-(y)&= Q(X)Q(Y)\\
&= Q(Q(X)Y+XQ(Y)-XY)\\
&= -Q(\phi_-(x)Y+X\phi_-(y)+XY)\\
&= -Q\bigg(\phi_-(x)\big(\phi(y)
    +\sum_{(y)}\phi_-(y')\phi(y''))\big)
 +\big(\phi(x)+\sum_{(x)} \phi_-(x')\phi(x'')\big)
    \phi_-(y) \\
& \quad +\big(\phi(x)+\sum_{(x)} \phi_-(x')\phi(x'')\big)
\big(\phi(y)+\sum_{(y)}\phi_-(y')\phi(y'')\big)\bigg) \\
&= -Q\bigg(\phi_-(x)\phi(y)
+\sum_{(y)}\phi_-(x)\phi_-(y')\phi(y'')
+\phi(x)\phi_-(y) \\
&\quad +\sum_{(x)} \phi_-(x')\phi_-(y)\phi(x'')
+\phi(x)\phi(y)
 +\sum_{(x)} \phi_-(x')\phi(x'')\phi(y) \\
&\quad +\sum_{(y)}  \phi_-(y')\phi(x)\phi(y'')
+\sum_{(x),(y)} \phi_-(x')\phi_-(y')\phi(x'')
\phi(y'')\bigg).
\end{align*}}
Here we use the commutativity of $A$ in the last equation. Then by
induction on $\deg x+\deg y$ and the multiplicativity of
$\phi$, we see that this equals to $\phi_-(xy)$.
\end{proof}

\section{Rota-Baxter algebra and Atkinson Factorization}
\mlabel{sec:rb}

We first give
in Section~\mref{pre} the definition of Rota-Baxter algebras and basic
examples. We then prove the Atkinson Factorization in Section~\mref{sec:atk2}.
We then drive \ABD from Atkinson Factorization in Section~\mref{ss:atabc}.
In Section~\mref{sec:sp}, we prove Spitzer's identity
and use it to give an explicit form of \ABD.

\subsection{Definitions and examples}
\mlabel{pre}

\begin{defn}
Let $\lambda$ be a fixed element of $\bfk$. Let $R$ be a
$\bfk$-algebra. A linear operator $P$ on $R$ is called a {\bf
Rota--Baxter operator}, of weight $\lambda$
if $P$ satisfies the {\bf Rota--Baxter equation}
\begin{equation}
P(x)P(y)=P(xP(y))+P(P(x)y)+\lambda P(xy),\quad \forall x,y\in R,
\mlabel{eq:RB}
\end{equation}
or, equivalently,
\begin{equation}
P(x)P(y)+\theta P(xy)=P(xP(y))+P(P(x)y),\quad \forall x,y\in R.
\end{equation}
Here $\theta=-\lambda$. Then $R$ is called a {\bf Rota--Baxter
$\bfk$-algebra}, of weight $\lambda$.
\end{defn}
The algebra is named after the American mathematician Glen E. Baxter~\footnote{Glen E. Baxter (March 19, 1930 - March 30, 1983) was an American
mathematician, born in Minneapolis, Minnesota.
He received his Ph.D. in 1954 from University
of Minnesota. His fields of research include probability,
combinatorial analysis, statistical mechanics and functional
analysis. He had appointments at M.I.T., University of Minnesota,
University of California at San Diego and Purdue University, as
well as a visiting professorship at the University of Aarhus, Denmark.} who introduced this structure in 1960,
and the well-known combinatorist Gian-Carlo Rota who led its study in the 1960s and 1970s and then promoted its study in the 1990s. Rota-Baxter operator is related to the classical Yang-Baxter equation, named after the famous physicists Chen-Ning Yang and Rodney James Baxter. Since the turn of this century, Rota-Baxter algebra has been studied in connection with quite a few areas of mathematics and physics, including Yang--Baxter
equations, shuffle products,
operads, Hopf algebra, combinatorics and number theory. Our focus here is its connection with the
work of Connes and Kreimer in their Hopf algebraic approach to
renormalization theory in perturbative quantum field
theory.

\begin{exam}
{\rm
Let $R$ be the $\RR$-algebra of continuous functions on $\RR$.
Define $P:R\to R$ by the integration
$$ P(f)(x)=\int_0^x f(t)dt.$$
Then $P$ is a Rota-Baxter operator of weight zero. This follows from the
integration by parts formula. For $f,g\in R$, let $F(x)=\int_0^x
f(t)\,dt$ and $G(x)=\int_0^x g(t)\, dt$. Then $F'(x)=f(x)$ and
$G'(x)=g(x)$. So by integration by parts formula, we have
$$ \int_0^x f(t) G(t)\, dt = F(t)G(t)\Big|_0^x-\int_0^x F(t) g(t)\, dt
    = F(x)G(x)-\int_0^x F(t) g(t)\, dt .$$
Rearranging the terms, we get Eq.(\mref{eq:RB}) with $\lambda=0$.
}
\end{exam}

\begin{exam}
{\rm
Let $R$ be the set of sequences $\{a_n\}$ with values in $\bfk$.
Then $R$ is a $\bfk$-algebra with termwise sum, product and scalar
product. Define
$$P:R\to R,\\ P(a_1,a_2,\cdots)=(a_1,a_1+a_2,a_1+a_2+a_3,\cdots).$$
Regarding $\{a_n\}$ as a function $f:\NN\to \RR$, then $P(f)$ is
the sequence of partial sums
$$ P(f):\NN\to \RR,\ P(f)(n)=\sum_{i=1}^{n} f(i),\ \ n\in \NN.$$
We show that $P$ is a Rota-Baxter operator of weight $-1$. For $f,g\in R$, have
$$(P(f)P(g))(n)=P(f)(n)P(g)(n)=\left (\sum_{i=1}^n f(i)\right)
    \left (\sum_{j=1}^n g(j)\right)
    =\sum_{i,j=1}^n f(i)g(j).$$
The sum is over the set of lattice points
$$\{(i,j)\, \big|\, 1\leq i\leq n,\, 1\leq j\leq n\}$$
which is the union of subsets
$$\{(i,j)\, \big|\, 1\leq i\leq j\leq n\}$$
and
$$\{(i,j)\, \big|\, 1\leq j\leq i\leq n\}$$
with the overlapping set
$$\{(i,j)\, \big|\, 1\leq i= j\leq n\}.$$
We also have
$$ P(fP(g))(n)=\sum_{1\leq i\leq n} f(i)P(g)(i)
= \sum_{1\leq i\leq n} f(i)\left (\sum_{1\leq j\leq i} g(j)\right)
= \sum_{1\leq i\leq j\leq n} f(i)g(j)$$ which is summed over the
first subset. Similarly $P(P(f)g)(n)$ and $P(fg)(n)$ are the sums
over the second and third subset respectively. Thus $P$ satisfies
Eq. (\mref{eq:RB}) with $\lambda=-1$. \mlabel{ex:sum}
}
\end{exam}

\begin{exam}
{\rm
For the algebra $R$ in Example~\ref{ex:sum}, define
$$P:R\to R,\\ P(a_1,a_2,\cdots)=(0, a_1,a_1+a_2,\cdots).$$
Then $P$ is a Rota-Baxter operator of weight $1$.
}
\end{exam}

\begin{exam}
{\rm
Let $R$ be a $\bfk$-algebra. For any given $\lambda$, the
operator $P_\lambda :R\to R$ defined by $P_\lambda(a)=-\lambda a$
is a Rota-Baxter operator or weight $\lambda$. In particular, the identity map
is a Rota-Baxter operator of weight $-1$.
}
\mlabel{ex:mult}
\end{exam}

\begin{exam}
We recall from Section~\mref{sec:algb} the algebra $\CC[t^{-1},t]]$ of Laurent series. The operator
$$Q(\sum_{n=-k}^\infty a_n t^{n}) = \sum_{n=-k}^{-1} a_n t^{n}$$
on $\CC[t^{-1},t]]$, whose image is in $t^{-1}\CC[t^{-1}]$, is a Rota-Baxter operator of weight $-1$.
\mlabel{ex:lau}
\end{exam}

\subsection{Atkinson Factorization}
\mlabel{sec:atk2}
We now consider Atkinson Factorization and the related existence and uniqueness problems.

\subsubsection{Atkinson Factorization}
\mlabel{ss:atk2}
We start with a lemma.
\begin{lemma}
Let $(R,P)$ be a Rota-Baxter algebra of weight $\lambda$. and let $\tilde{P}=-\lambda \id - P$. Then
\begin{eqnarray}
P(x)\tilde{P}(y) &=& P\big(x \tilde{P}(y)\big) +
\tilde{P}\big(P(x)y\big) \mlabel{ab}\\
\tilde{P}(x) P(y)&=& P\big(\tilde{P}(x)y\big) +
\tilde{P}\big(xP(y)\big), \;\; x,y \in R. \mlabel{ba}
\end{eqnarray}
\mlabel{lem:ab}
\end{lemma}
\begin{proof}
We have
\begin{eqnarray*}
P(x)\tilde{P}(y)
&=& -\lambda P(x)y-P(x)P(y) \\
&=& -\lambda P(x)y-P(xP(y))-P(P(x)y)-\lambda P(xy)\\
&=& P\big(x \tilde{P}(y)\big) +
\tilde{P}\big(P(x)y\big).
\end{eqnarray*}
The proof of Eq.~(\mref{ba}) is the same.
\end{proof}

The following is Atkinson (multiplicative)
Factorization~\mcite{At}.
\begin{theorem}
{\rm (Atkinson Factorization)} Let $(R,P)$ be a Rota--Baxter algebra of weight $\lambda\neq 0$. Let $a\in R$. Assume that $b_\ell$ and $b_r$ are solutions of the fixed point equations
\begin{equation}
 b_\ell=1+P(b_\ell a), \qquad
b_r=1+\tilde{P}(ab_r).
\mlabel{eq:recurlr}
\end{equation}
Then
\begin{equation}
 b_\ell (1+\lambda a) b_r = 1.
 \mlabel{eq:atk2}
\end{equation}
Thus
\begin{equation}
 1+\lambda a = b_\ell^{-1} b_r^{-1}
\mlabel{eq:atk3}
\end{equation}
if $b_\ell$ and $b_r$ are invertible.
\mlabel{thm:atk2} \mlabel{thm:Atkinson}
\end{theorem}
\begin{proof}
Using Lemma~\mref{lem:ab} and Eq.~(\mref{eq:recurlr}), we compute
\begin{eqnarray*}
b_\ell b_r&=& (1+P(b_\ell a))(1+\tilde{P}(ab_r))\\
&=& 1+P(b_\ell a) + \tilde{P}(a b_r)+P(b_\ell a)\tilde{P}(a b_r) \\
&=& 1+P(b_\ell a) + \tilde{P}(ab_r) +P(b_\ell a \tilde{P}(ab_r))+\tilde{P}(P(b_\ell a) ab_r) \\
&=& 1+P\big(b_\ell a (1+\tilde{P}(ab_r))\big)
+\tilde{P}\big((1+P(b_\ell a))ab_r\big) \\
&=& 1+P(b_\ell ab_r) +\tilde{P}(b_\ell a b_r) \\
&=& 1-\lambda b_\ell a b_r.
\end{eqnarray*}
Therefore,
$$ 1=b_\ell b_r +\lambda b_\ell a b_r = b_\ell(1+\lambda a)b_r,$$
as needed.
\end{proof}

\begin{remark}
{\rm
We note that the factorization (\mref{eq:atk3}) depends on the exitance of invertible solutions of Eq.~(\mref{eq:recurlr}). There are the following three natural questions:
\begin{enumerate}
\item
When does the Atkinson Factorization exist?
\item
In which sense the Atkinson Factorization is unique?
\item
Is there an explicit formula for the factors $b_\ell$ and $b_r$ in the Atkinson Factorization?
\end{enumerate}
We will answer the first two questions in Theorem~\mref{thm:atke} and the third one in Theorem~\mref{thm:Sp3}.
\mlabel{rk:atk}
}
\end{remark}

\subsubsection{Exitance and uniqueness of the Atkinson Factorization}
A Rota-Baxter algebra $(R,P)$ is called {\bf complete} if there are submodules $R_n\subseteq R, n\geq 0,$ such that $(R,R_n)$ is a complete algebra and $P(R_n)\subseteq R_n$.
\begin{theorem}
{\rm (Existence and uniqueness of the Atkinson Factorization)}
Let $(R,P,R_n)$ be a complete Rota-Baxter algebra. Let $a$ be in $R_1$.
\begin{enumerate}
\item
The equations in (\mref{eq:recurlr}) have unique solutions $b_\ell$ and $b_r$. Further $b_\ell$ and $b_r$ are invertible. Hence Atkinson Factorization (\mref{eq:atk3}) exists.
\mlabel{it:atke}
\item
If $\lambda$ has no non-zero divisors in $R_1$ and $P^2=-\lambda P$ (in particular if $P^2=-\lambda P$ on $R$), then
there are unique $c_\ell\in 1+P(R)$ and $c_r\in 1+\tilde{P}(R)$ such that
$$ 1+\lambda a= c_\ell c_r.$$
\mlabel{it:atku}
\end{enumerate}
\mlabel{thm:atke}
\end{theorem}
\begin{proof}
\mref{it:atke} This follows from general results on fix point equations. We give the details to be self-contained. Define
$ b_{\ell,0}= 1$
and inductively define
\begin{equation}
b_{\ell,n+1}=1+P(b_{\ell,n} a), n\geq 0.
\mlabel{eq:fix}
\end{equation}
So
$$ b_{\ell,1}=1+P(a), \quad b_{\ell,2}=1+P(a)+P(P(a)a), \cdots.$$
Then for the metric (\mref{eq:metric}) defined by the filtration on $R$, we have
\begin{eqnarray*}
o(b_{\ell,n+1}-b_{\ell,n})&=&
o\big(P((b_{\ell,n}-b_{\ell,n-1})a)\big)
    \quad (\mbox{by Eq.~(\mref{eq:fix})})\\
&\geq & o \big(((b_{\ell,n}-b_{\ell,n-1})a)\big)
        \quad (\mbox{since $P(R_n)\subseteq P(R_n)$})\\
&>& o(b_{\ell,n}-b_{\ell,n-1})\quad (\mbox{since $a\in A_1$}).
\end{eqnarray*}
Thus $\{b_{\ell,n}\}$ is a Cauchy sequence and hence
$b_{\ell,\infty}=\lim_{n\to \infty} b_{\ell,n}$ exists.
Taking limits on both sides of Eq.~(\mref{eq:fix}), we see that $b_{\ell,\infty}$ is a solution of
$b_\ell= 1+P(b_\ell a)$. It is the unique solution since, if
$b'_\ell=1+P(b_\ell' a)$, then $b_\ell-b_\ell'=P((b_\ell-b_\ell')a)$. Then
$$ |b_\ell -b_\ell'|<|b_\ell -b_\ell'|$$
unless $b_\ell-b_\ell'=0$.

Further by Lemma~\mref{lem:inv}, $b_\ell=1+P(b_\ell a)$ is invertible.

The same proof works for $b_r$.
\medskip

\mref{it:atku}
Suppose there are $a_1,a_2, b_1,b_2\in R_1$ such that
$$ 1 +\lambda a
=(1+P(a_1))(1+\tilde{P}(b_1))=(1+P(a_2))(1+\tilde{P}(b_2)).
$$
Then
\begin{equation} (1+P(b_1))^{-1}(1+P(a_1)) =(1+\tilde{P}(b_2))(1+\tilde{P}(a_2))^{-1}.
\mlabel{eq:cap}
\end{equation}
Since
$$ (1+P(b_1))^{-1}=\sum_{n\geq 0} (-P(b_1))^{n}$$
is in $1+P(R_1)$, by the Rota-Baxter relation we have
$(1+P(b_1))^{-1}(1+P(a_1)) =1+P(c)$ for some $c\in R_1$.
Similarly
$(1+\tilde{P}(b_2))(1+\tilde{P}(a_2))^{-1}= 1+\tilde{P}(d)$ for some $d\in R_1$. Thus Eq.~(\mref{eq:cap} gives $1+P(c)=1+\tilde{P}(d)$ and
hence $P(c)=\tilde{P}(d)$. Applying $P$ to this equation and using $P^2=-\lambda P$, we have
$$-\lambda P(c)=P^2(c)=P(\tilde{P}(d))=0,$$
giving $P(c)=0$ under our assumption on $\lambda$.
Then $\tilde{P}(d)=0$ also. Thus we have
$$1+P(a_1)=1+P(b_1), \qquad 1+\tilde{P}(a_2)=1+\tilde{P}(b_2).
$$
This proves the uniqueness.
\end{proof}

\subsection{From Atkinson Decomposition to Algebraic Birkhoff Decomposition}
\mlabel{ss:atabc}

We now derive \ABD of Connes and Kreimer in Theorem~\mref{thm:algBi} from Atkinson Factorization in Theorem~\mref{thm:atke}.
Adapting the notations in Theorem~\mref{thm:algBi}, let $H$ be a connected filtered cograded bialgebra (hence a Hopf algebra) and let $(A,Q)$ be a commutative Rota-Baxter algebra of weight $\lambda=-1$ with $Q^2=Q$. By Theorem~\mref{thm:inv}, $R:=\Hom(H,A)$ is a complete algebra with the filtration
$R_n=\{ f\in \Hom(H,A)\ |\ f(H^{n-1})=0 \}, n\geq 0$.
Further define
$$ P: R \to R, \quad P(f)(x)=Q(f(x)), f\in \Hom(H,A), x\in H.$$
Then it is easily checked that $P$ is a Rota-Baxter operator of weight $-1$ and $P^2=P$. Thus $(R,R_n,P)$ is a complete Rota-Baxter algebra.

Now let $\phi:H\to A$ be a character (that is, an algebra homomorphism). Consider $e-\phi:H\to A$. Then
$$ (e-\phi) (1_H) =e(1_H)-\phi(1_H)=1_H-1_H=0.$$
Thus $e-\phi$ is in $R_1$. Take $e-\phi$ to be our $a$ in Theorem~\mref{thm:atke}, we see that there are unique
$c_\ell\in P(R_1)$ and $c_r\in P(R_1)$ such that
$$ \phi = c_\ell c_r.$$
Further, by Theorem~\mref{thm:atk2}, for $b_\ell=c_\ell^{-1}$,
$b_\ell=e+P(b_\ell \cprod (e-\phi))$. Thus for $x\in \ker \vep =\ker e$, we have
\begin{eqnarray*}
b_\ell (x)&=& P(b_\ell \cprod (e-\phi))(x) \\
&=& \sum_{(x)} Q(b_\ell(a_{(1)})(e-\phi)(x))\\
&=& Q\big( b_\ell(1_H)(e-\phi)(x)+
\sum_{(a)} b_\ell(x')(e-\phi)(a'') + b_\ell(x)(e-\phi)(1_H)\\
&=& -Q\big( \phi(x) + \sum_{(x)} b_\ell (x')\phi(x'')\big).
\end{eqnarray*}
In the last equation we have used
$e(a)=0, e(a'')=0$ by definition.
Since $b_\ell(1_H)=1_H$, we see that $b_\ell=\phi_-$ in Eq.~(\mref{eq:phi-}).

Further, we have
$$c_r=c_\ell^{-1} \phi=b_\ell \phi
= -b_\ell(e-\phi) + b_\ell
= -b_\ell(e-\phi) + e+P(b_\ell (e-\phi))
= e-\tilde{P}(b_\ell (e-\phi)).$$
With the same computation as for $b_\ell$ above,
we see that $c_r = \phi_+$ in Eq.~(\mref{eq:phi+}).

\subsection{Spitzer's identity and explicit Algebraic Birkhoff Decomposition}
\mlabel{sec:sp}
We now address the third question in Remark~\mref{rk:atk} about Theorem~\mref{thm:Atkinson}, namely on explicit formula for the Atkinson Factorization in Eq.~(\mref{eq:atk3}). By Theorem~\mref{thm:atke}, answer to this question also provide an explicit formula for \ABD. We obtain our answer by generalizing Spitzer's identity. This identity is important both for its
theoretical significance and for its surprisingly wide range of
applications.

\subsubsection{Classical form of Spitzer's identity}

According to Rota~\mcite{Ro2}, Spitzer's formula~\mcite{Sp} was regarded as a remarkable stepping stone in the theory of sums of independent random variables in the
fluctuation theory of probability. It was discovered by the
mathematician Frank Spitzer in 1956. Even though Spitzer's identity in its
original form describes relations in fluctuation theory in
probability, it is better understood in terms of Rota--Baxter
operators. In fact, the very
motivation for Glen Baxter~\mcite{Ba} to introduce this operator was to give a more conceptual proof of
Spitzer's identity.

Let $(A,Q)$ be a Rota-Baxter algebra of weight $\lambda$. Consider the power
series ring  $R:=A[[x]]$ on one variable $x$. Define an
operator
\begin{equation}
 P: R \to R,\quad P(\sum_{n=0}^\infty a_n
x^n)=\sum_{n=0}^\infty Q(a_n) x^n.
\mlabel{eq:comp}
\end{equation}
 Then it is easy to check that $(R,P)$ is
a complete Rota-Baxter algebra with the filtration $R_n:=R x^n$, $n\geq
0$.

The classical form of Spitzer's identity has the following algebraic abstraction~\mcite{Ca,R-S}.

\begin{theorem} {\rm (Spitzer \mcite{R-S})}
Let $(A,Q)$ be a commutative Rota--Baxter $\QQ$-algebra of weight $\lambda = -1$. Then for $a \in A$,
$$
b=\exp\left (-P(\log(1-ax)) \right )
$$
is a solution of the fix point equation
$$
b=1+P(bax).
$$
Thus we have
\begin{equation}
 \exp\left (-P(\log(1-ax)) \right )
    =\sum_{n=0}^\infty x^n \underbrace{P\big( P( P( \cdots (P(a)a ) a ) a) \big)}_{n\mbox{\rm -}{\rm times}}
    \mlabel{eq:si1}
\end{equation}
in the ring of power series $A[[x]]$. \mlabel{thm:Sp1}
\end{theorem}

To get ourselves acquainted with this seemingly unmotivated identity,
let us consider the case where $P$ is the identity map $\id$. Recall from
Example~\mref{ex:mult} that $\id$ is a Rota-Baxter of weight -1.
Then the left hand side of Eq. (\mref{eq:si1}) becomes the power series
$$ \exp\big( -\log(1-ax)\big)=\exp( \log (1-ax)^{-1})= \frac{1}{1-ax}$$
and the right hand side is
$$ \sum_{n=0}^\infty x^n a^n =\sum_{n=0}^\infty (ax)^n$$
So we have the familiar geometric expansion.

The form of Spitzer's identity in Theorem~\mref{thm:Sp1}
is an immediate consequence of Theorem~\mref{thm:Sp3} applied
to the complete Rota-Baxter algebra $A[[x]]$ in Eq.~\mref{eq:comp}.
A direct proof, using Rota's standard Rota-Baxter algebras, can be found in~\mcite{R-S}. See~\mcite{E-G-M} for further generalizations and their applications  to vertex operator algebras, combinatorial Hopf algebras and the Magnus formula.

\subsubsection{Kingman's theorem}
As a preparation, we prove the following theorem of Kingman~\mcite{Ki}. Define the double product $\dprod$ of the product in a Rota-Baxter algebra $(R,P)$ by
$$ a \dprod b= a P(b)+P(a)b+\lambda ab.$$
\begin{prop}
Let $(R,P)$ be a Rota-Baxter algebra of weight $-1$.
\begin{enumerate}
\item $ a \dprod b =
P(a)P(b)-\tilde{P}(a)\tilde{P}(b),\quad
P(a)P(b)=P\big(P(a)P(b)-\tilde{P}(a)\tilde{P}(b)\big ).$
\mlabel{it:king} \item For $n\geq 2$,
$$ \prod_{i=1}^{n} P(a_i)
= P\bigg( \prod_{i=1}^{n} P(a_i) - \prod_{i=1}^{n}
(- \tilde{P}(a_i))\bigg), \;\; a_i \in A,\; i=1 \dots n.$$
\mlabel{it:n} \item {\rm (Kingman, 1962)~\mcite{Ki}}
\begin{equation}
 P(u)^n =P\big(P(u)^n-
(-\tilde{P}(u))^n\big ),\;\; u\in A. \mlabel{eq:Kingman}
\end{equation}
\mlabel{it:sn}
\end{enumerate}
\mlabel{pp:kingman}
\end{prop}
\begin{proof}
\mref{it:king} The first equation follows from $\tilde{P}=I-P$ and then the second equation follows from $P(x)P(y)=P(x\ast_P
y)$.

\mref{it:n} We use induction on $n\geq 2$ with $n=2$ verified in
\mref{it:king}. Assume the equation holds for $n$. Then since
$\tilde{P}$ is also a Rota-Baxter operator of weight $-1$ with
$\tilde{\tilde{P}}=P$, we have
\allowdisplaybreaks{
\begin{eqnarray}
 \prod_{i=1}^{n} \tilde{P}(a_i)
 &=&
\tilde{P}\bigg( \prod_{i=1}^{n} \tilde{P}(a_i) -
\prod_{i=1}^{n} (- P(a_i))\bigg) \notag\\
&=&  (-1)^n\tilde{P}\bigg( \prod_{i=1}^{n}
\tilde{P}(-a_i) - \prod_{i=1}^{n} P(a_i)\bigg). \mlabel{eq:tn}
\end{eqnarray}
} Then we have \allowdisplaybreaks{
\begin{eqnarray*}
 \prod_{i=1}^{n+1} P(a_i) &=&
\big(\prod_{i=1}^n P(a_i)\big ) P(a_{n+1})\\
&=& P\bigg( \prod_{i=1}^{n} P(a_i) -
\prod_{i=1}^{n} (- \tilde{P}(a_i))\bigg)\bigg)P(a_{n+1})\ \
    {\rm (by\ induction)} \\
&=& P\bigg (P\bigg( \prod_{i=1}^{n} P(a_i) -
\prod_{i=1}^{n} (- \tilde{P}(a_i))\bigg)P(a_{n+1}) \\
&& \ \ \ \ \ \ \ \ \ \ \ \ -\tilde{P}\bigg( \prod_{i=1}^{n} P(a_i)
- \prod_{i=1}^{n} (- \tilde{P}(a_i))\bigg)\tilde{P}(a_{n+1})\bigg)
\
\ \ {\rm (by\ \mref{it:king})}\\
&=& P\bigg (\prod_{i=1}^n P(a_i) P(a_{n+1})-
(-1)^{n-1}\prod_{i=1}^n \tilde{P}(a_i) \tilde{P}(a_{n+1})\bigg) \
\ {\rm
(by\ (\mref{eq:tn}))}\\
&=& P\bigg(\prod_{i=1}^{n+1} P(a_i) -
\prod_{i=1}^{n+1}(-\tilde{P}(a_i))\bigg).
\end{eqnarray*}
} This completes the induction.

\mref{it:sn} This follows from \mref{it:n} by taking $a_i=a$.
\end{proof}

\subsubsection{The BCH recursion}
We first recall the following classical result which can be found
in any treatise on Lie groups or at the
web link:\\
http://mathworld.wolfram.com/Baker-Campbell-HausdorffSeries.html

\begin{theorem} {\rm (Baker-Campbell-Hausdorff formula)}
\begin{enumerate}
\item There is a unique power series $BCH(x,y)$ of degree 2 in the
noncommutative power series algebra
$\ZZ\left<\left<x,y\right>\right>$ such that
$$
\exp(x)\exp(y)=\exp\big(x+y+BCH(x,y)\big).
$$
\item For any complete ring $A$ and any $x,y\in A_1$, we have
$$
\exp(x)\exp(y)=\exp\big(x+y+BCH(x,y)\big).
$$
\end{enumerate}
\end{theorem}
The first few terms of $BCH(x,y)$ are
$$BCH(x,y)=\frac{1}{2}[x,y]+\frac{1}{12}[[x,y],y]-\frac{1}{12}[[x,y],x]-
\frac{1}{48}[y,[x,[x,y]]]-\frac{1}{48}[x,[y,[x,y]]]+\cdots$$ where
$[x,y]=xy-yx$ is the commutator of $x$ and $y$.
\begin{prop}
Let $(R,R_n,P)$ be a complete Rota-Baxter algebra of weight $-1$. There is a unique map $\chi:
R_1 \to R_1$ that satisfies the equation
\begin{equation}
\chi(a)=a -
BCH\big(P(\chi(a)),\tilde{P}(\chi(a))\big)
\mlabel{BCH-recur}
\end{equation}
\mlabel{pp:BCH}
\end{prop}
$\chi$ was introduced in \mcite{E-G-K2} and was called the
$BCH$-recursion. See~\mcite{Ma} for a more conceptual proof in the context of Lie algebras.

\begin{proof}
$\chi(a)$ is defined to be $\lim_{n \to \infty} \chi_n(a)$ where
recursively, \allowdisplaybreaks{\begin{eqnarray*}
\chi_0(a) &=& a, \\
\chi_{n+1}(a) &=& a - BCH\big(
P(\chi_n(a)),\tilde{P}(\chi_n(a))
\big).
\end{eqnarray*}}
To see why this gives the unique solution to the recursion
relation (\ref{BCH-recur}), we first define, for $a \in R_1$, a
map $\Lambda : R \to R$ \mcite{E-G-K2}
$$
\Lambda(a):=BCH\big(P(a),\tilde{P}(a)\big).
$$
Then for $s \in R_n, n \geq 1$, $\Lambda(a+s)$ is $\Lambda(a)$
plus a sum in which each term has $s$ occurring at least once, and
hence is contained in $R_{n+1}$. Thus we have
\begin{equation}
 \Lambda(a \mod R_n) \equiv \Lambda(a) \mod R_{n+1}.
\mlabel{eq:cong}
\end{equation}
Now we have
$$
\chi_1(a)= a + \Lambda(\chi_0(a)) = a + \Lambda(a) \equiv a \equiv
\chi_0(a) \mod R_{2}.
$$
By induction on $n$ and (\ref{eq:cong}), we have
\begin{eqnarray*}
 \chi_{n+1}(a)& =& a + \Lambda(\chi_n(a)) \\
&\equiv & a + \Lambda(\chi_{n-1}(a) \mod R_{n+1}) \\
& \equiv & a + \Lambda(\chi_{n-1}(a)) \mod R_{n+2}\\
& \equiv & \chi_n(a) \mod R_{n+2}.
\end{eqnarray*}
Thus $\lim_{n\to \infty} \chi_n(a)$ exists and is a solution of
(\ref{BCH-recur}).

Suppose $b$ is another solution. Then, as above, we have
$$\chi_0(a) = a \equiv a + \Lambda(b) \equiv b \mod R_2.$$
Induction on $n$ gives
\begin{eqnarray*}
\chi_{n+1}(a) &=& a+\Lambda(\chi_n(a)) \\
    &\equiv & a + \Lambda(b \mod R_{n+2})\\
    &\equiv & a + \Lambda(b) \mod R_{n+3}\\
    &\equiv & b \mod R_{n+3}.
\end{eqnarray*}
Thus $b=\lim_{n\to \infty} \chi_n(a).$
\end{proof}

\begin{lemma} Let $R$ be a complete filtered
$\bfk$-algebra. Let $K:R \to R$ be a linear map. The map $\chi$ in Eq.~(\ref{BCH-recur}) solves the following recursion
 \begin{equation}
   \chi(u)=u + BCH\big(-K(\chi(u)),u)\big),\;\; u\in R_1.
   \mlabel{BCHrecursion3}
 \end{equation}
\end{lemma}
\begin{proof}
In general for any $u \in R$  we can write $u = K(u) +
\tilde{K}(u)$ using linearity of $K$. Here $\tilde{K}=\id_R-K$. The map $\chi$ then
implies for $u \in R_1$ that $\exp(u)= \exp\big(K(\chi(u))\big)
\exp\big(\tilde{K}(\chi(u))\big)$. Further,
 \allowdisplaybreaks{
\begin{eqnarray*}
 \exp\big(-K(\chi(u))\big) \exp(u) &=& \exp\big(\tilde{K}(\chi(u))\big)  \\
                                   &=& \exp\big(-K(\chi(u)) + u + BCH(-K(\chi(u)),u)\big).
\end{eqnarray*}}
Bijectivity of $\log$ and $\exp$ then implies, that
 \allowdisplaybreaks{
\begin{eqnarray*}
 \chi(u) - K(\chi(u)) &=& -K(\chi(u)) + u + BCH(-K(\chi(u)),u),
\end{eqnarray*}}
{}from which Equation (\ref{BCH-recur}) follows.
\end{proof}

\subsubsection{Spitzer's identity in the non-commutative case}
\mlabel{ss:spnc}
For $a \in A$, inductively define
\begin{equation}
(P a)^{[n+1]}:=P\big((Pa)^{[n]}\:a\big)\;\; \makebox{and} \;\;(P
a)^{\{n+1\}}:=P\big(a\:(Pa)^{\{n\}}\big)
\mlabel{eq:iter}
\end{equation}
with the convention that $(Pa)^{[1]}=P(a)=(Pa)^{\{1\}}$ and
$(Pa)^{[0]}=1=(Pa)^{\{0\}}$.
\begin{theorem}
Let $(R,R_n,P)$ be a complete filtered Rota--Baxter algebra of
weight $-1$. Let $a \in R_1$.
\begin{enumerate}
\item The equation
\begin{equation}
 b_\ell=1+P(b_\ell a)
\mlabel{eq:recurs}
\end{equation}
has a unique solution
\begin{equation}
b_\ell= \exp\big(-P(\chi(\log (1- a)))\big).
\mlabel{eq:exp1}
\end{equation}
\item The equation
\begin{equation}
b_r=1+\tilde{P}(ab_r) \mlabel{eq:recurs2}
\end{equation}
has a unique solution
\begin{equation}
b_r= \exp\big(-\tilde{P}(\chi(\log (1- a)))\big). \mlabel{eq:exp2}
\end{equation}
\item The Atkinson Factorization~(\mref{eq:atk3}) is given by
\begin{equation}
 1-a = \exp\big(-P(\chi(\log (1- a)))\big)
 \exp\big(\tilde{P}(\chi(\log (1- a)))\big).
\mlabel{eq:expabc}
\end{equation}
\end{enumerate}
\mlabel{thm:Sp3} \mlabel{complRB}
\end{theorem}
\begin{proof}
We only need to verify for the first equation. The proof for the
second equation is similar.

Since $a$ is in $R_1$ and $P$ preserves the filtration, the series (using notation from Eq.~(\mref{eq:iter}))
$$
b_\ell=1+ P(a)+ P(P(a)a) + \cdots +(Pa)^{[n]} + \cdots
$$
defines a unique element in $R$ and is easily seen to be a solution of (\ref{eq:recurs}).

Conversely, if $c \in R$ is a
solution of (\ref{eq:recurs}), then by iterated substitution, we
have
$$
c=1+P(a)+P(P(a)a)+ \cdots +(Pa)^{[n]}+\cdots.
$$
Therefore, the equation (\ref{eq:recurs}) has a unique solution.

To verify that (\ref{eq:exp}) gives this solution, take $u:=\log
(1- a), \; a\in R_1$. Using (\ref{eq:Kingman}), for our
chosen $b$ we have \allowdisplaybreaks{
\begin{eqnarray*}
\lefteqn{\exp\big(- P(\chi(\log (1- a)))\big)= \exp\big(P(-\chi(u))\big) }\\
&=& \sum_{n=0}^\infty \frac{1}{n!}P\big(-\chi(u)\big)^n \\
&=& 1 +P \bigg(\sum_{n=1}^\infty
\frac{1}{n!}\bigg (\big(P(-\chi(u))\big)^n
            -\big(\tilde{P}(\chi(u))\big)^n\bigg)\bigg)\\
&=& 1 +P \bigg(\sum_{n=0}^\infty
\frac{1}{n!}\big(P(\chi(u))\big)^n
            -\sum_{n=0}^\infty \frac{1}{n!}\big(\tilde{P}
            (\chi(u))\big)^n\bigg)\\
&=& 1 +P\bigg(\exp\big(P(-\chi(u))\big) -\exp\big(\tilde{P}(\chi(u))\big)\bigg).
\end{eqnarray*}}
By the definition of the $BCH$-recursion $\chi$ in equation
(\ref{BCH-recur}), we have \allowdisplaybreaks{\begin{eqnarray*}
&&\exp\big(P(\chi(u))\big)\exp\big(\tilde{P}(\chi(u))\big) \\
&=&\exp\bigg(P\big(\chi(u)\big)+\tilde{P}\big(\chi(u)\big)
            +BCH\big(P(\chi(u)),\tilde{P}(\chi(u))\big)\bigg)\\
&=& \exp \bigg( \chi(u)+BCH\big(P(\chi(u)),\tilde{P}(\chi(u))\big)\bigg)\\
&=& \exp(u).
\end{eqnarray*}}
Thus
\allowdisplaybreaks{\begin{eqnarray*}
&&\exp\big(-P(\chi(\log (1- a)))\big)\\
&=& 1+ P\bigg (\exp\big(P(-\chi(u))\big)-\exp\big(P(-\chi(u))\big)\exp(u)\bigg)\\
&=& 1+ P\bigg (\exp\big(P(-\chi(u))\big)\big(1-\exp(u)\big)\bigg)\\
&=& 1+ P\bigg (\exp\big(-P(\chi(\log(1- a)))\big)\big(1-\exp(\log(1- a))\big)\bigg )\\
&=& 1+P\bigg (\exp\big(-P(\chi(\log(1-
a))\big)\;a\bigg)
\end{eqnarray*}}
This verifies the first equation.
\end{proof}

Now apply Theorem~\mref{thm:Sp3} to the case when $R=\Hom(H,A)$ where $H$ is a connected filtered cograded bialgebra (hence a Hopf algebra) and $(A,Q)$ is a commutative Rota-Baxter algebra of weight $\lambda=-1$ with $Q^2=Q$, as we did in \S \mref{ss:atabc}. We obtain
\begin{coro}
Let $\phi: H\to A$ be a character. Then in the \ABD of $\phi=\phi_-^{\ast (-1)}\ast \phi_+$ in Eq.~{\rm(}\mref{eq:decom}{\rm)}, we have the explicit expressions
\begin{eqnarray*}
\phi_-&=& \exp_\ast\big(-P(\chi(\log_\ast (e- \phi)))\big), \\
\phi_+&=& \exp_\ast\big(\tilde{P}(\chi(\log_\ast (e-\phi)))\big).
\end{eqnarray*}
\mlabel{co:expabc}
\end{coro}
We also obtain the following generalization of the classical Spitzer's identity in Eq.~(\mref{eq:si1}).
\begin{coro}
Let $(R,P,R_n)$ be a complete filtered Rota--Baxter algebra. For
$a \in R_1$, we have
\begin{equation}
\sum_{n=0}^\infty \big(Pa\big)^{[n]}
    = \exp\big(-P(\chi(\log (1+\lambda a)))\big)
\mlabel{eq:Spitzer1}
\end{equation}
\begin{equation}
\sum_{n=0}^\infty \big(\tilde{P}a\big)^{\{n\}}
    = \exp\big(-\tilde{P}(\chi(\log (1+\lambda a)))\big)
\mlabel{eq:Spitzer2}
\end{equation}
\mlabel{co:Spitzer}
\end{coro}

\begin{proof}
Both sides of (\ref{eq:Spitzer1}) are solutions of
(\ref{eq:recurs}). This proves (\ref{eq:Spitzer1}) by the
uniqueness of the solution to (\mref{eq:recurs}).

The proof of (\ref{eq:Spitzer2}) is the same, by considering
solutions of the recursive equation (\ref{eq:recurs2})
\end{proof}

\section{Renormalization of divergent integrals}
\mlabel{sec:int}
\subsection{A very rough idea of renormalization}
\mlabel{ss:idea}
Very roughly speaking, the method of renormalization viewed in the framework of Connes and Kreimer can be described as follows.
Let
$$\calf = \{ f_\Gamma\ | \ \Gamma\in \calg\}$$
be a set of formal expressions, such as formal integrals or formal summations, indexed by a set $\calg$. The expressions are formal in the sense that they are divergent that cannot be cured by traditional mathematical methods such as analytic continuation or a limit process as for improper integrals. To apply the renormalization method, we introduce two algebraic structures from the given data. First from the index set $\calg$, we define a Hopf algebra structure on the free $\bfk$-module
$$ H_\calg:= \oplus_{\Gamma\in \calg} \bfk \Gamma.$$
Examples of such Hopf algebras include the Hopf algebras of rooted trees of Kreimer~\mcite{Kr}, of Feynman graphs of Connes and Kreimer~\mcite{C-K1}, and of quasi-shuffle product of Hoffman~\mcite{Ho2} in the work of Guo and Zhang~\mcite{G-Z,G-Z2}, and of Manchon and Payche~\mcite{M-P2}.
Next on the set $\calf$, we define a Rota-Baxter algebra structure, through ``deforming" the formal expressions
$f_\Gamma$ to a function $f_\Gamma(\vep)$ by introducing a new parameter $\vep$, such that $f_\Gamma(\vep)$ is well-defined except when $\vep\to 0$ which returns $f_\Gamma(\vep)$ to the original formal expression. Because of this, the regularized expression $f_\Gamma(\vep)$ has a Laurent series expansion in $\bfk[\vep^{-1},\vep]]$, equipped with the Rota-Baxter operator $P$ introduced in
Example~\mref{ex:lau}.
The two algebraic structures should be compatible in the sense that the index map
$$ \Phi: \calg \to \calf, \quad \Phi(\Gamma)=f_\Gamma, $$
extends to an algebra homomorphism
$$ \phi: H_\calg = \bigoplus_{\Gamma\in \calg} \bfk\, \Gamma \longrightarrow \sum_{\Gamma\in \calg} \bfk\, f_\Gamma(\vep) \subseteq \CC[\vep^{-1},\vep]].$$
Note that the coproduct on $H_\calg$ and the Rota-Baxter operator on $\bfk[\vep^{-1},\vep]]$ are not used in this algebra homomorphism. These two extra structures give the \ABD of $\phi$ stated in Theorem~\mref{thm:algBi}:
\begin{equation}
\phi = \phi_-^{\cprod (-1)} \cprod \phi_+
\mlabel{eq:abc2}
\end{equation}
that decomposes $\phi: H_\calg \to \CC[\vep^{-1},\vep]]$
into the ``converges part"
$$\phi_+:H_\calg \to \CC[[\vep]]$$
such that $\phi_+(\Gamma;\vep)$ is well-defined when $\vep=0$ and the ``divergent part"
$$\phi_-: H_\calg \to \CC[\vep^{-1}]$$
that is responsible for the divergency of $\phi(\Gamma)$. We summarize this in the following diagram.
\begin{equation}
\xymatrix{
&& \CC\, \calg \ar^{\Phi}[rr] \ar^{\phi_+}[dd] \ar^{\phi}[ddrr]
    \ar_{\bar{\Phi}}[ddll] &&
            \bigoplus_{\Gamma\in \calg} \CC f_\Gamma
           \\
&&&&\\
\CC && \CC[[\vep]] \ar@{^{(}->}[rr] \ar_{\vep=0}[ll] && \CC[\vep^{-1},\vep]]  \ar_{\vep=0}[uu]
}
\mlabel{diag:abc}
\end{equation}
Note the difference between the two maps that take $\vep=0$. The one on the right is not well-defined, reflecting the divergent nature of the formal expressions. The one on the left is well-defined, giving the {\bf renormalized values} of the formal expressions.

As the reader will immediately point out, the renormalized value $\bar{\Phi}(\Gamma)$ for $f_\Gamma$ depends on several ingredients of the renormalization process, especially the regularization $\phi$. How the renormalized values depend on the regularization is an intriguing and hard problem. In the case of quantum field theory renormalization, it was verified by the physicists that the renormalized value of a Feynman diagram does not depend on the choice of a renormalization process that observes basis physical restrains. On the other hand, the mathematical study of renormalization is still in its early stage. Its further study should shed new light on the understanding of the renormalization in physics.

In this and the next sections, we give two mathematical applications of the renormalization method. The first example can be regarded as a simplification of the renormalization of Feynman integrals. See~\mcite{C-K1} for a related discussion.

\subsection{The Hopf algebra of rooted trees}
\mlabel{sec:treeh}
We now introduce one of the primary examples of Hopf algebras for physics applications, the Hopf algebra of rooted trees.
It was introduced in~\mcite{Kr,C-K0} as a toy model of the Hopf
algebra of Feynman graphs in QFT to study renormalization of
perturbative QFT. It is also related to the
Hopf algebras of rooted trees of Grossman-Larson~
\mcite{G-L} and of Loday-Ronco~\mcite{L-R}.

A rooted tree is a connected and simply-connected
set of vertices and oriented edges such that there is precisely
one distinguished vertex, called the root, with no incoming edge.
A rooted tree is called non-planar if the branches of the same vertex can be permuted.

$$
  \ta1 \;\;\;\; \tb2 \;\;\;\; \tc3 \;\;\;\; \td31 \;\;\;\; \te4 \;\;\;\; \tf41
   =\tg42 \;\;\;\; \th43 \;\;\;\; \ti5 \;\;\;\; \tj51 = \tk52
  \;\;\;\; \cdots \;\;\;\; \tp56 = \tq57 \;\;\; \cdots
$$
Here the term non-planar means different embeddings of a rooted tree into the plane are identified.

\begin{remark}{\rm
For the rest of this paper, a tree means a non-planar rooted tree unless otherwise stated. The same applies to the concept of a forest introduced below.
}
\end{remark}
Let $\calt$ be the set of (isomorphic classes of) rooted trees. Let
$\calh_\calt$ be the commutative polynomial algebra over $\QQ$
generated by $\calt$: $\calh_\calt=\QQ[\calt]$. Monomials of trees
are called forests. The {\bf depth} $\depth(F)$ of a forest is the longest path from one of its roots to the leafs.  We will define a coalgebra structure on
$\calh_\calt$. A subforest of a tree $T$ consists of a set of vertices of $T$ together with their descendants and edges connecting all these vertices.

The {\bf coproduct} is then defined as follows. Let $\calf_T$ be the
set of subforests of the rooted tree $T$, including the empty subforest, identified with $1$, and the full forest. Define
$$ \Delta(T)= \sum_{F\in \calf_T} F \otimes (T/F). $$
The quotient $T/F$ is obtained by removing the subforest $F$ and edges connecting the subforest to the rest of the tree. We use the convention that if $F$ is the empty subforest, then $T/F=T$, and if $F=T$, then $T/F=1$. For example,
\allowdisplaybreaks{
\begin{eqnarray}
 \Delta(\ta1)            &=& \ta1  \otimes 1 + 1 \otimes \ta1
    \notag \\
 \Delta\big(\!\!\begin{array}{c}
                 \\[-0.4cm]\tb2 \\
               \end{array}\!\!\big)    &=&
                         \begin{array}{c}
                           \\[-0.4cm]\tb2 \\
                         \end{array} \!\! \otimes 1
                         + 1 \otimes \!\!
                         \begin{array}{c}
                           \\[-0.4cm]\tb2 \\
                         \end{array}
                               + \ta1 \otimes \ta1 \notag\\
 \Delta\Big(\!\!\begin{array}{c}
                   \\[-0.4cm]\tc3 \\
                 \end{array}\!\!\Big)    &=&
                            \begin{array}{c}
                             \\[-0.4cm]\tc3 \\
                            \end{array}\!\! \otimes 1
                         + 1 \otimes \!\!
                         \begin{array}{c}
                           \\[-0.4cm]\tc3 \\
                         \end{array}\!\!
                          + \ta1 \otimes\!\!
                         \begin{array}{c}
                           \\[-0.4cm]\tb2 \\
                         \end{array} +
                         \begin{array}{c}
                           \\[-0.4cm]\tb2 \\
                         \end{array}\!\! \otimes \ta1 \mlabel{treeEx} \\
 \Delta\big(\:\td31\big) &=&  \begin{array}{c}
                                 \\[-0.4cm]\td31 \\
                              \end{array}\!\! \otimes 1
                               + 1 \otimes
                              \begin{array}{c}
                                 \\[-0.4cm]\td31 \\
                              \end{array}\!\! +
                                  2\ta1 \otimes \!\!
                                       \begin{array}{c}
                                         \\[-0.4cm]\tb2 \\
                                       \end{array} + \ta1\ta1\otimes \ta1
                                       \notag \\
 \Delta\bigg(\!\! \begin{array}{c}
                   \\[-0.4cm]\te4 \\
                  \end{array}\!\!\bigg)
                     &=& \begin{array}{c}
                           \\[-0.4cm]\te4 \\
                         \end{array} \!\! \otimes 1
                         + 1 \otimes \!\!
                         \begin{array}{c}
                           \\[-0.4cm]\te4 \\
                         \end{array}
                         + \ta1 \otimes \!\!
                         \begin{array}{c}
                           \\[-0.4cm]\tc3 \\
                         \end{array}
                            + \begin{array}{c}
                        \\[-0.4cm]\tb2 \\
                         \end{array} \!\! \otimes \!\!
                         \begin{array}{c}
                        \\[-0.4cm]\tb2 \\
                         \end{array} +
                         \begin{array}{c}
                        \\[-0.4cm]\tc3 \\
                         \end{array}\!\! \otimes \ta1 \notag
\end{eqnarray}}

We then extend $\Delta$ to $\calh_\calt$ by multiplicity with the
convention $\Delta(1)=1\ot 1$. Then $\Delta$ is an algebra homomorphism in the sense of Eq.~(\mref{eq:md0}). Also define
$$\vep: \calh_\calt\to \bfk$$
by $\vep(T)=0$ for $T\in \calt$, $\vep(1)=1$ and extending by
multiplicity. Then $\vep$ is also an algebra homomorphism. It is
easy to verify its compatibility with $\Delta$. Further, define
the degree of a tree and a forest (i.e., monomial in $\CC[\calt]$)
to be its number of vertices. Then the grading defined by the
degree is preserved by both the product and coproduct. It is also
connected. Thus to prove that $\calh_\calt$ is a bialgebra and
hence a Hopf algebra, we only need to verify the coassociativity
of $\Delta$. For this we follow the standard approach.
See~\mcite{Kr,F-G1}

Let $B^+:\calh_\calt \to \calh_\calt$ be the linear
map given by taking the product $T_1\cdots T_k$ of $k$ trees
$T_1,\cdots, T_k$ to the tree $T$ consisting of a new vertex,
subtrees $T_i$ and an edge from the new vertex to the root of each $T_i$. The map $B^+$ is also called the {\bf grafting operator} in combinatorics.
Then~\mcite{C-K0}
$$\Delta B^+=B^+\ot 1+ (\id_H \ot B^+)\Delta.$$

\begin{theorem}
$\calh_\calt$ is a connected graded Hopf algebra. \mlabel{thm:tco}
\end{theorem}
\begin{proof}
By the remark before the theorem, it remains to show that
the map $\Delta$ is coassociative.

Let $H_n$ be the polynomial subalgebra of $\calh_\calt$ generated
by rooted trees with at most $n$ vertices. Then
$\calh_\calt=\cup_{n\geq 0} H_n$. We use induction on $n\geq 0$ to
prove that
$$(\id_H\ot
\Delta)\Delta=(\Delta\ot \id_H)\Delta$$ on $H_n.$ For $n=1$, we
have
\begin{equation}
(\id_H\ot\Delta)\Delta (\bullet)=\bullet\ot 1\ot 1+1\ot \bullet
\ot 1 +1\ot 1\ot \bullet = (\Delta\ot 1)\Delta(\bullet).
\mlabel{eq:b}
\end{equation}
Assume
that $\Delta$ is coassociative on $H_n$ and let $T$ be a rooted
tree with $n+1$ vertices. Then $T=B^+(a)$ where $a=T_1\cdots T_k$
with $T_i$ the subtrees of $T$ immediately descending from the
root of $T$. Then by Eq.~\mref{eq:b} and induction hypothesis,
we have
\begin{align*}
(\id_H\ot \Delta)\Delta(T)&=T\ot 1\ot +(\id_H\ot (\Delta
B^+))\Delta(a)\\
&=T\ot 1\ot 1+(\id_H\ot \id_H\ot B^+)(\Delta\ot
\id_H)\Delta + ((\id_H\ot B^+)\Delta(a))\ot 1\\
&=T\ot 1\ot 1+(\id_H\ot \id_H\ot B^+)(\id_H\ot \Delta)\Delta
+ ((\id_H\ot B^+)\Delta(a))\ot 1\\
&=T\ot 1\ot 1+ (\Delta\ot B^+)\Delta(a)
    +\big((\id_H\ot B^+)\Delta(a)\big)\ot 1\\
&= (\Delta\ot \id_H)\big(T\ot 1+(\id_H\ot B^+)\Delta(a)\big)\\
&= (\Delta\ot \id_H)\Delta(T).
\end{align*}
\end{proof}

We next characterize the set of rooted forests by a universal property. The following concept is the commutative version of the concept of operated semigroups and operated algebras introduced in~\mcite{Guop}.

\begin{defn}{\rm
A {\bf commutative \mapped semigroup} is a commutative semigroup $U$ together with an operator
$\alpha: U\to U$. A morphism between commutative \mapped
semigroups $(U,\alpha)$ and $(V,\beta)$ is a
semigroup homomorphism $f :U\to V$ such that $f \circ \alpha=
\beta \circ f,$ that is, such that the following diagram
commutes.
$$
 \xymatrix{
            U\ar[rr]^\alpha \ar[d]^f && U \ar[d]_f \\
            V\ar[rr]^\beta           && V}
$$
}
\mlabel{de:mapset}
\end{defn}
When a semigroup is replaced by a commutative monoid we obtain the concept of a {\bf commutative \mapped monoid}. Let $\bfk$ be a commutative ring. We similarly define the concepts of a {\bf commutative \mapped $\bfk$-algebra} or
{\bf commutative \mapped nonunitary $\bfk$-algebra}.

A {\bf free commutative \mapped semigroup} on a set $X$ is a commutative \mapped semigroup
$(U_X,\alpha_X)$ together with a map $j_X:X\to U_X$ with the
property that, for any commutative \mapped semigroup $(V,\beta)$ and
any map $f:X\to V$, there is a unique morphism
$\free{f}:(U_X,\alpha_X)\to (V,\beta)$ of commutative \mapped semigroups such
that $ f=\free{f}\circ j_X.$ In other words, the following diagram commutes.
$$
 \xymatrix{
            X \ar[rr]^{j_X} \ar[drr]_{f} && U_X \ar[d]^{\free{f}} \\
                                         && V
            }
$$
Let $\cdot$ be the binary operation on the semigroup $U_X$, we also use the quadruple $(U_X, \cdot, \alpha_X, j_X)$ to denote the free commutative \mapped semigroup on $X$.

We similarly define the concepts of free commutative \mapped unitary and nonunitary $\bfk$-algebras.

\begin{theorem}
With the grafting operator (denoted by $\lc\ \rc$ in~\mcite{Guop}) $B^+:\calh_\calt\to \calh_\calt$,
\begin{enumerate}
\item
$\calf\backslash \{1\}$ is the free commutative \mapped semigroup on one generator.
\mlabel{it:treesg}
\item
$\oplus_{F\in \calf\backslash \{1\}} \CC\, F$ is the free commutative \mapped nonunitary algebra on one generator.
\mlabel{it:treealg}
\end{enumerate}
\mlabel{thm:treefree}
\end{theorem}
The proof follows from the same argument as for the non-commutative case in~\mcite{Guop}.

\subsection{Evaluating divergent integrals by renormalization}

Now for each rooted forest $F$, we define a formal integral $f_F(c), c>0,$ by the recursive structure of rooted forests.
First define the formal expression
$$ f_\onetree (c)= \int_0^\infty \frac{dx}{x+c}, c>0.$$
This allows us to define $f_F(c)$ for rooted forests of height 0 by multiplicity, namely if $F=\underbrace{\onetree \cdots \onetree}_{n-{\rm factors}}$, then

$$ f_F(c)= \prod_{i=1}^n \int_0^\infty \frac{dx_i}{x_i+c}
    = \left (\int_0^\infty \frac{dx}{x+c}\right )^n$$
Assume that $f_F(c)$ have been defined for all rooted forests $F$ with $0\leq \depth(F)\leq k$, and let $T$ be a rooted tree with $\depth(T)=k+1$. Then $T=B_+(\oF)$ for a rooted forest $\oF$ with $\depth(F)=k$. By the induction hypothesis, $f_\oF(c)$ is defined. We then define
$$ f_T(c)=\int_0^\infty f_\oF(x) \frac{dx}{x+c}.$$
Let $F$ be a rooted forest with $\depth(F)=k+1$.
Then we can uniquely write
$$ F=T_1 \cdots T_n$$
where $T_i, 1\leq i\leq n,$ are rooted trees with $\depth(T_i)\leq k+1$. Thus
$$ f_F(c)=\prod_{i=1}^n f_{T_i}(c)
$$
is defined.

Here are the formal expression corresponding to some rooted forests.
$$ f_{\onetree\, \onetree}(c)=\left (\int_0^\infty \frac{dx}{x+c}\right)^2, \quad
f_{\tb2}(c)=\int_0^\infty \left(\int_0^\infty \frac{dx_1}{x_1+x}\right) \frac{dx}{x+c} $$
$$
f_{\td31}(c)=\int_0^\infty \left (\int_0^\infty \frac{dx_1}{x_1+x}\int_0^\infty \frac{dx_2}{x_2+x}\right) \frac{dx}{x+c}$$

Thus we have the formal map
$$ \Phi: \calf \to \{ f_F\ |\ F\in \calf\}$$
sending $F\in \calf$ to the divergent integral $f_F$.
To get renormalized values of these integrals, we first build a Hopf algebra $H_\calt$ from $\calf$ as we did in Theorem~\mref{thm:tco}. We next construct the regularization $f_F(c;\vep)$ of the divergent integrals $f_F(c)$. We do this again recursively as we define $f_F(c)$, except replacing $du$ by $u^{-\vep} du$ for each variable $u$. Thus we define
$$ f_\onetree(c;\vep)=\int_0^\infty \frac{x^{-\vep}dx}{x+c}$$
and for a rooted tree $T=\lc \oF\rc$ with $\depth(F)=k+1$, recursively define
$$ f_T(c;\vep)=\int_0^\infty f_{\oF}(x;\vep) \frac{x^{-\vep} dx}{x+c}$$
By classical analysis, each $f_F(c;\vep)$ is convergent for $\vep\in \CC$ with $Re(\vep)>0$ and can be analytically continued to a convergent Laurent series, still denoted by $f_F(c;\vep)$, in $\CC[\vep^{-1},\vep]]$.
In fact, using the formula
$$\int_0^\infty \frac{x^{-k\vep}d x}{x+1} = \frac{\pi}{\sin (k \pi \vep)},$$
we obtain
$$f_\onetree(c;\vep)=\frac{\pi}{ c^{\vep}\sin(\pi \vep)}$$
and the recursive relation
$$ f_T(c;\vep)=f_{\oF}(c;\vep) \frac{\pi}{c^{\vep}\sin(|T|\pi \vep)}$$
where $|T|$ is the degree (i.e., the number of vertices) of $T$. For a fixed value of $c$, we now define an algebra homomorphism
\begin{equation}
\phi: H_\calt \to \CC[\vep^{-1},\vep]], \quad
\phi(F)=f_F(c;\vep), \quad F\in \calf.
\mlabel{eq:intreg}
\end{equation}
This homomorphism is also compatible with the operated algebra structure on $H_\calt$ given by the grafting and the operated algebra structure on the corresponding integrals given by the integral operator
sending $f(c;\vep)$ to $\int_0^\infty f(x;\vep)\frac{x^{-\vep}\,dx}{x+c}$.
Applying the \ABD, we obtain an algebra homomorphism
$$\phi_+: H_\calt \to \CC[[\vep]].$$
We then define the {\bf renormalized value} of the formal integral $f_F(c)$ to be
$$\bar{f}_F(c):=\lim_{\vep\to 0}\phi_+(F)(c;\vep).$$
For example,
$$f_\onetree(c;\vep)=\int_0^\infty \frac{x^{-\vep}\,dx}{x+c}
= \frac{1}{\vep}-\ln c + \left(\frac{\pi^2}{6}+ \frac{(\ln c)^2}{2}\right)\vep +o(\vep).$$
Hence the renormalized value of $f_\onetree(c)$ is
$\bar{f}_\onetree(c)=-\ln c.$
On the other hand, by Eq.~(\mref{eq:phi+}),
\begin{eqnarray*}
 \phi_+(\tb2)(c;\vep) &=& (\id -P)\big( \phi(\tb2)(c;\vep) - P(\phi(\ta1)(c;\vep))\phi(\ta1)(c;\vep)\big)
  \end{eqnarray*}
which is the power series part of
$\phi(\tb2)(c;\vep) - P(\phi(\ta1)(c;\vep))\phi(\ta1)(c;\vep)$.
Since
$$\phi(\tb2)(c;\vep)=
\frac{1}{2\vep^2}-\frac{\ln c}{\vep} + ((\ln c)^2 +\frac{5}{12}\pi^2) -\frac{\ln c}{6}(4(\ln c)^2+5\pi^2)\vep + o(\vep),$$
we obtain
$$ \bar{f}_{\tb2}(c)=\frac{1}{2}(\ln c)^2 + \frac{\pi^2}{4}.$$

\section{Renormalization of multiple zeta values}
\mlabel{sec:mzv}
Our next application of \ABD is to the study of multiple zeta values. We take the viewpoint in~\mcite{G-Z,G-Z2}. See~\mcite{M-P2,Zh2} for other approaches and applications.

\subsection{Multiple zeta values}
Multiple zeta values (MZVs) are defined to be the convergent sums
\begin{equation} \zeta(s_1,\cdots, s_k)=\sum_{n_1>\cdots>n_k>0}
    \frac{1}{n_1^{s_1}\cdots n_k^{s_k}}
\mlabel{eq:mzv}
\end{equation}
where $s_1,\cdots,s_k$ are positive integers with $s_1>1$. Since the papers of Hoffman~\mcite{Ho0} and Zagier~\mcite{Za} in the early 1990s, their study have attracted interests from several areas of mathematics and physics~\mcite{3BL,B-K,Ca1,Go3,G-M,Ho2}, including number theory, combinatorics, algebraic geometry and mathematical physics.

In order to study the multiple variable function
$\zeta(s_1,\cdots,s_k)$ at integers $s_1,\cdots,s_k$ where the
defining sum (\mref{eq:mzv}) is divergent, one first tries to use
the analytic continuation, as in the one variable case
of the Riemann zeta function. Such an analytic continuation was achieved in~\mcite{AET,Ma2,Zh2},
showing that $\zeta(s_1,\cdots,s_k)$ can be meromorphically continued to $\CC^k$
with singularities on the subvarieties
\begin{eqnarray*}
s_1&=&1;\\
s_1+s_2&=&2,1,0,-2,-4, \cdots; {\rm\ and\ } \mlabel{eq:pole}\\
\sum_{i=1}^j
s_{i} &\in& \ZZ_{\leq j}\ (3\leq j\leq k).
\end{eqnarray*}
Thus, unlike in the one variable case, the multiple zeta function in Eq.~(\mref{eq:mzv}) is still undefined at most non-positive integers even with the analytic continuation. 

\subsection{Quasi-shuffle Hopf algebra}
Let $M$ be a commutative semigroup. For each integer $k\geq 0$, let $\bfk M^k$ be the free $\bfk$-module with basis $M^k$, with the convention that $M^0=\{\bfone\}$. Let
\begin{equation}
 \calh_M=\bigoplus_{k=0}^\infty \bfk\, M^k.
 \mlabel{eq:qprod}
\end{equation}
Following~\mcite{Ho2}, define the {\bf quasi-shuffle product} $\msh$ by
first taking $\bfone$ to be the multiplication identity.
Next for any $m,n\geq 1$ and
$\vec a:=(a_1,\cdots , a_m)\in M^{m}$ and $\vec b:=(b_1, \cdots, b_n)\in M^{n}$, denote
$\vec a\,'=(a_2,\cdots,a_m)$ and $\vec b\,'=(b_2,\cdots,b_n)$. Recursively define
\begin{equation}
\vec{a} \msh \vec{b} =\big(a_1,   \vec{a}\,'\msh \vec b\big )
  + \big(b_1, \vec a \msh \vec{b}\,' \big) \\
  + \big(a_1 b_1,   \vec{a}\,' \msh \vec{b}\,'\big)
\mlabel{eq:qshuf}
\end{equation}
with the convention that $\vec{a}\,'=\bfone$ if $m=1$, $\vec{b}\,'=\bfone$ if $n=1$ and
$(a_1b_1,\vec{a}\,'\msh \vec{b}\,')=(a_1b_1)$ if $m=n=1$.
(Without the third term on the right, Eq.~(\mref{eq:qshuf}) gives the recursive definition of shuffle product.)
For example,
\begin{eqnarray*}
\lefteqn{ (a_1,a_2)\msh (b_1,b_2)= (a_1,(a_2)\msh (b_1,b_2))
 +(b_1,(a_1,a_2)\msh (b_2))+(a_1b_1, (a_2)\msh (b_2))}\\
 &=& \big(a_1,\big((a_2, \bfone \msh (b_1,b_2))+(b_1,(a_2)\msh (b_2))+(a_2b_1,\bfone \msh (b_2))\big)\big)\\
 &&+\big(b_1,\big((a_1,(a_2)\msh (b_2))+(b_2, (a_1,a_2)\msh \bfone)+(a_1b_2, (a_2)\msh \bfone)\big)\big) \\
 && \big(a_1b_1, \big( (a_2, \bfone\msh (b_2))+(b_2, (a_2)\msh \bfone) + (a_2b_2)\big)\big)\\
&=& (a_1,a_2,b_1,b_2)+(a_1,b_1,a_2,b_2)+(a_1,b_1,b_2,a_2) +(a_1,b_1,a_2b_2)+(a_1,a_2b_1,b_2) \\
&& + (b_1,a_1,a_2,b_2)+(b_1,a_1,b_2,a_2)+(b_1,a_2,a_2b_2) + (b_1,b_2,a_1,a_2)+(b_1,a_1b_2,a_2)\\
&& +(a_1b_1,a_2,b_2)+(a_1b_1,b_2,a_2)+(a_1b_1,a_2b_2).
 \end{eqnarray*}
Alternatively~\mcite{G-K1,G-K2}, $\vec{a}\msh \vec{b}$ is the sum of {\bf mixable shuffles} of $\vec{a}$ and $\vec{b}$ consisting of the shuffles of $\vec{a}$ and $\vec{b}$ (see \S\,\mref{sss:exam}) and the {\bf mixed shuffles} by merging some of $(a_i,b_j)$ in a shuffle to $a_ib_j$. In the above example, we have
\begin{eqnarray*}
\lefteqn{ (a_1,a_2)\msh (b_1,b_2)=
(a_1,a_2,b_1,b_2)+(a_1,b_1,a_2,b_2)} \\
&& +(a_1,b_1,b_2,a_2)
+ (b_1,a_1,a_2,b_2)+(b_1,a_1,b_2,a_2)+ (b_1,b_2,a_1,a_2)
\quad {\rm (shuffles)} \\
&&+(a_1,b_1,a_2b_2)+(a_1,a_2b_1,b_2) +(b_1,a_2,a_2b_2) +(b_1,a_1b_2,a_2)\\
&& +(a_1b_1,a_2,b_2)+(a_1b_1,b_2,a_2)+(a_1b_1,a_2b_2)
\quad {\rm (mixed\ shuffles)}.
 \end{eqnarray*}

There are many interpretations of the quasi-shuffle product. It is also known as harmonic product~\mcite{Ho1} and coincides with the stuffle product~\mcite{3BL,Br}
in the study of MZVs. Variations of the stuffle product have also appeared in ~\mcite{Ca,Eh}.
Mixable shuffles are also called overlapping shuffles~\mcite{Ha} and generalized shuffles~\mcite{Go3}, and can be interpreted in terms of Delannoy paths~\mcite{A-H,Fa,Lo}.

By the same proofs as \cite[Theorem 2.1]{Ho2} and \cite[Theorem 3.1]{Ho2} we see that $\calh_M$ is
a bialgebra. In~\cite{Ho2} $M$ has the extra condition of being a locally finite set to ensure that $\calh_M$ is a graded Hopf algebra, not just a filtered Hopf algebra.
By the definition of $\msh$ and $\Delta$, $\calh_M$ is connected filtered cograded with the submodules $\bfk\,M^n, n\geq 0$. Then $\calh_M$ is automatically a Hopf algebra by Theorem~\mref{thm:pt}. Thus we have

\begin{theorem}
Let $M$ be a commutative semigroup.
Equip $\calh_M$ with
the submodules $\calh_M^{(n)}=\oplus_{i=0}^n \bfk\, M^i$,
The quasi-shuffle product $\msh$,
the deconcatenation coproduct
\begin{eqnarray}
&&\Delta: \calh_M \to \calh_M \barot \calh_M,\\
 \Delta( a_1, \cdots ,  a_k)&=&1\barot (a_1,  \cdots ,  a_k)
+ \sum_{i=1}^{k-1} (a_1,  \cdots,  a_i)\barot (a_{i+1},  \cdots ,  a_k) \notag\\
&&+ (a_1,  \cdots ,  a_k) \barot 1
\mlabel{eq:coprod}
\end{eqnarray}
and the projection counit
$\vep: \calh_M \to \bfk$
onto the direct summand $\bfk \subseteq \calh_M$.
Then $\calh_M$ is a commutative connected filtered Hopf algebra.
\mlabel{thm:hopf}
\end{theorem}

\subsection{The Hopf algebra of directional regularized multiple zeta values}
We consider the commutative semigroup
\begin{equation}
\frakM= \{{{\wvec{s}{r}}}\ \big|\ (s,r)\in \ZZ \times \RR_{>0}\}
\mlabel{eq:mbase}
\end{equation}
with the multiplication
$ {\wvec{s}{r}} {\wvec{s'}{r'}}={\wvec{s+s'}{r+r'}}.$
%
By Theorem~\mref{thm:hopf},
$$\calh_{\frakM}:=\sum_{k\geq 0} \CC\, \frakM^k,$$
with the quasi-shuffle product $\msh$ and the deconcatenation coproduct $\Delta$, is a
connected filtered Hopf algebra.
For $w_i=\wvec{s_i}{r_i}\in \frakM,\ i=1,\cdots,k$, we use the notations
$$ \vec{w}=(w_1,\cdots,w_k)
=\wvec{s_1,\cdots,s_n}{r_1,\cdots,r_k}=\wvec{\vec{s}}{\vec{r}},\
{\rm where\ } \vec{s}=(s_1,\cdots,s_k), \vec{r}=(r_1,\cdots,r_k).$$
For
$\vep\in \CC$ with ${\rm
Re}(\vep)<0$, define the {\bf directional regularized MZV}:
\begin{equation}
Z(\wvec{\vec{s}}{\vec{r}};\vep)=\sum_{n_1>\cdots>n_k>0}
\frac{e^{n_1\,r_1\vep} \cdots
    e^{n_k\,r_k\vep}}{n_1^{s_1}\cdots n_k^{s_k}}
\mlabel{eq:reggmzv}
\end{equation}
It converges for any $\wvec{\vec{s}}{\vec{r}}$ and is regarded as the regularization of the {\bf formal MZV}
\begin{equation}
\zeta (\vec{s})= \sum_{n_1>\cdots>n_k>0} \frac{1}{n_1^{s_1} \cdots
    n_k^{s_k}}
\mlabel{eq:formgmzv}
\end{equation}
which converges only when $s_i>0$ and $s_1>1$.
It is related to the multiple polylogarithm
$${\rm Li}_{s_1,\cdots,s_k}(z_1,\cdots,z_k)=\sum_{n_1>\cdots n_k>0}
    \frac{z_1^{n_1} \cdots z_k^{n_k}}{n_1^{s_1}\cdots n_k^{s_k}}$$
by a change of variables $z_i=e^{r_i\vep}, 1\leq i\leq k$.
As is well-known~\mcite{3BL,Go3}, the product of multiple polylogarithms as functions satisfies the quasi-shuffle (stuffle) relation of the nested sums. Therefore the product of regularized MZVs as functions also satisfies the quasi-shuffle relation: if
 $ \wvec{\vec{s}}{\vec{r}}\msh \wvec{\vec{s}\,'}{\vec{r}\,'}
=\sum \wvec{\vec{s}\,''}{\vec{r}\,''}$, then
\begin{equation}
Z(\wvec{\vec{s}}{\vec{r}};\vep)Z(\wvec{\vec{s}\,'}{\vec{r}\,'};\vep)
= Z(\wvec{\vec{s}}{\vec{r}}\msh \wvec{\vec{s}\,'}{\vec{r}\,'};\vep)
:= \sum Z(\wvec{\vec{s}\,''}{\vec{r}\,''};\vep).
\mlabel{eq:qsh1}
\end{equation}
 We thus obtained an algebra homomorphism
\begin{equation}
\phi_\mzv: \calh_\frakM \to \sum_{{\wvec{\vec{s}}{\vec{r}}\in \cup_{n\geq 0} \frakM^n}}\
    \CC\, Z(\wvec{\vec{s}}{\vec{r}};\vep),
    \wvec{\vec{s}}{\vec{r}} \mapsto Z(\wvec{\vec{s}}{\vec{r}};\vep).
\mlabel{eq:regz}
\end{equation}
With this map, $\calh_\frakM$ is a parametrization of the directional regularized MZVs that also reflects their multiplication property.

\subsection{Renormalization of multiple zeta values and examples}

It is shown in~\mcite{G-Z} (see also~\mcite{G-Z2,M-P2,Zh2}) that $Z(\wvec{\vec{s}}{\vec{r}};\vep)$ has Laurent series expansion in the algebra $\CC[T][\vep^{-1},\vep]]$ where $T$ is a variable representing $\ln(-\vep)$ coming from the pole at $\zeta(1)$. By \ABD in Eq.~(\mref{eq:decom}) with $K=\CC[T]$, we obtain an algebra homomorphism
$$\phi_{\mzv,+}: \calf_\frakM \to \CC[T][[\vep]].$$
Then we define the renormalized (directional) multiple zeta value of $\zeta(\vec{s})$ to be
$\gzeta\big(\wvec{\vec{s}}{\vec{r}}\big).$

To finish the paper we illustrate this method by some special cases when $\vec{s}$ has only non-positive components. In this case, $T$ does not occur in the Laurent series expansions of $Z(\wvec{\vec{s}}{\vec{r}};\vep)$. See~\mcite{G-Z,G-Z2} for further details. First recall the following generating series of Bernoulli numbers that goes back to Euler.
\begin{equation} \frac{\vep}{e^\vep-1}=\sum_{k\geq 0} B_k \frac{\vep^k}{k!}
\mlabel{eq:bern}
\end{equation}
It can be easily rewritten as
\begin{equation}
 \frac{e^\vep}{1-e^\vep}= -\frac{1}{\vep}\frac {-\vep}{e^{-\vep}-1}
    = -\frac{1}{\vep}+ \sum_{k\geq 0}\zeta(-k) \frac{\vep^{k}}{k!}
    \mlabel{eq:zeta}
\end{equation}
since $B_0=1$ and $\zeta(-k)=(-1)^k\frac{B_{k+1}}{k+1}$ for $k\geq
0$.

Now consider
$$Z(s;\vep)\big(=Z(\wvec{s}{1};\vep)\big)=\sum_{n\geq 1} \frac{e^{n\vep}}{n^s}, $$ regarded as a deformation or ``regularization" of the series defining the Riemann zeta function
$\zeta(s)=\sum_{n\geq 1} \frac{1}{n^s}$. The regularized series converges for any integer $s$ when ${\rm Re}(\vep)<0$. In particular,
$$Z(0;\vep)=\frac{e^\vep}{1-e^\vep}$$
and Eq.~(\mref{eq:zeta}) gives
the Laurent series expansion of the regularized sum $Z(0;\vep)=\sum_{n\geq 1}
e^{n\vep}$ at $\vep=0$. Then we have, for $\tilde{Q}=\id -Q$,
$$ \tilde{Q} \big(\sum_{n\geq 1}e^{n\vep}\big)\Big|_{\vep=0} =\zeta(0).$$
So the renormalized value of $Z(0;\vep)=\sum_{n\geq 1} e^{n\vep}$ is $\zeta(0)$. Similarly, to evaluate $\zeta(-k)$ for an integer $k\geq 1$, consider the regularized sum
$$ Z(-k;\vep)=\sum_{n\geq 1}{n^k} {e^{n\vep}}
= \frac{d^k}{d\vep} \big(\frac{e^\vep}{1-e^\vep}\big) $$ which
converges uniformly on any compact subset in ${\rm
Re}(\vep)<0$. So its Laurent series expansion at $\vep=0$ is obtained
by termwise differentiation of Eq.~(\mref{eq:zeta}), yielding
\begin{equation}
Z(-k;\vep)=(-1)^{-k-1}(k)!\,\vep^{-k-1}+\sum _{j=0}^{\infty} \zeta
(-k-j)\frac {\vep^j}{j!}. \mlabel{eq:zreg}
\end{equation}
We then have
$$ \tilde{Q}\big(\sum_{n\geq 1}{n^k} {e^{n\vep}}\big)\Big|_{\vep=0}
    = \zeta (-k).$$

Thus the renormalization method does give the correct
Riemann zeta values at non-positive integers. We next extend this to multiple zeta functions and ``evaluate" $\zeta(0,0)$, for example, by consider the regularized sum
$$ Z(0,0;\vep)\big(=Z(\wvec{{0,0}}{{1,1}};\vep)\big)=\sum_{n_1>n_2>0} e^{n_1\vep}e^{n_2\vep} =\frac{e^\vep}{1-e^\vep}\frac{e^{2\vep}}{1-e^{2\vep}}.$$

By the renormalized process (\mref{eq:phi+}) adopt to our case we found that the renormalized value is defined by
\begin{eqnarray*}
&& \tilde{Q}\Big(\sum_{n_1>n_2>0} e^{n_1\vep}e^{n_2\vep}-\sum_{n_2>0} e^{n_2\vep} \big(
\underbrace{\sum_{n_1>0} e^{n_1 \vep}-{\rm\, \tilde{Q}}(\sum_{n_1>0} e^{n_1 \vep})}_{\mbox{subdivergence}}\big)\Big)\Big|_{\vep=0}\\
&=& \tilde{Q}\Big( \frac{1}{2}\frac{1}{\vep^2}-\frac{3}{2}\zeta(0)\frac{1}{\vep} +(-\frac{5}{2}\zeta(-1)+\zeta(0)^2+o(\vep) ) -\big(\frac{1}{\vep^2}-\frac{\zeta(0)}{\vep}-\zeta(-1)+o(\vep) \big) \Big)\Big|_{\vep=0}\\
&=& -\frac{3}{2}\zeta(-1)+\zeta(0)^2=\frac{3}{8}.
\end{eqnarray*}
This value indeed satisfies the
quasi-shuffle relation $ \zeta(0) \zeta(0)= 2\zeta(0,0)+\zeta(0)$.

Note that naively taking the finite part as in the one variable case gives
$$\tilde{Q}\big(\sum_{n_1>n_2>0} e^{n_1\vep}e^{n_2\vep}\big)\big|_{\vep=0}
=\tilde{Q}\big(\frac{1}{2\vep^2}-\frac{3}{2}\zeta(0)\frac{1}{\vep}+
\big(-\frac{5}{2}\zeta(-1)+\zeta(0)^2\big)+o(\vep)\big)\big|_{\vep=0}
=11/24.$$
This
value does not satisfy the quasi-shuffle
(stuffle) relation:
$$ \zeta(0) \zeta(0)\neq 2\, \zeta(0,0)+\zeta(0)$$
since the left hand side is $1/4$ and the right hand side is $5/12$.
This contradicts the well-known quasi-shuffle relation
$$ \zeta(s)\zeta(s)=2\, \zeta(s,s)+\zeta(2s)$$
for any integer $s\geq 2$.
See~\mcite{G-Z,G-Z2,M-P2,Zh2} for more systematic presentations and further progress.


%
%

\end{document}